\documentclass[12pt]{article}
\usepackage{amsmath,amsfonts,amssymb,amsthm}
\usepackage{mathtools}
\usepackage{xcolor}

\usepackage{colortbl}
\RequirePackage[colorlinks,citecolor=blue,urlcolor=blue,backref=page]{hyperref}

\numberwithin{equation}{section}

\newtheorem{proposition}{Proposition}[section]
\newtheorem{theorem}[proposition]{Theorem}
\newtheorem{corollary}[proposition]{Corollary}

\newtheorem{lemma}[proposition]{Lemma}
\newtheorem{Lemma}[proposition]{Lemma}
\newtheorem{remark}[proposition]{Remark}

\newtheorem{example}[proposition]{Example}

\newtheorem{convention}[proposition]{Convention}

\newcommand{\nc}{\newcommand}
\nc{\I}{{\mathbf 1}}

\nc{\bN}{{\mathbf N}}
\nc{\bM}{{\mathbf M}}
\nc{\bH}{{\mathbf H}}
\nc{\bF}{{\mathbf F}}
\nc{\bS}{{\mathbf S}}
\nc{\cB}{{\mathcal B}}
\nc{\cM}{{\mathcal M}}
\nc{\R}{{\mathbb R}}
\nc{\N}{{\mathbb N}}
\nc{\Z}{{\mathbb Z}}
\nc{\BX}{{\mathbb X}}
\nc{\BS}{{\mathbb S}}
\nc{\BD}{{\mathbb D}}
\nc{\BY}{{\mathbb Y}}
\nc{\cX}{{\mathcal X}}
\nc{\cY}{{\mathcal Y}}

\nc{\cE}{{\mathcal E}}

\nc{\bx}{\mathbf{x}}
\nc{\by}{\mathbf{y}}
\nc{\bz}{\mathbf{z}}

\nc{\dd}{\mathrm{d}} 
\nc{\Pn}[2]{#1^{(#2)}} 
\nc{\Pnm}[2]{#1^{[#2]}} 
\nc{\bfm}{\mathbf}

\DeclareMathOperator{\BV}{{\mathbb Var}}
\DeclareMathOperator{\CV}{{\mathbb Cov}}

\nc{\BP}{\mathbb{P}}
\nc{\BE}{\mathbb{E}}
\nc{\BQ}{\mathbb{Q}}

\nc{\Po}{{\mathrm{Po}}}
\nc{\Be}{{\mathrm{Beta}}}
\nc{\Di}{{\mathrm{Dir}}}

\newcommand{\gl}[1]{\textcolor{red}{#1}}

\begin{document} 

\renewcommand{\thefootnote}{\fnsymbol{footnote}}
\author{{\sc G\"unter Last\footnotemark[1] \hspace{0.1mm} and Babette Picker\footnotemark[2]}}
\footnotetext[1]{guenter.last@kit.edu, 
Institute of Stochastics, Karlsruhe Institute of Technology}
\footnotetext[2]{babette.picker@kit.edu,
Institute of Stochastics, Karlsruhe Institute of Technology}

\title{Stochastic analysis for the Dirichlet--Ferguson process}
\date{\today}
\maketitle

\begin{abstract} 
\noindent 
We study a Dirichlet--Ferguson process $\zeta$
on a general phase space. First we reprove the chaos expansion from
\cite{Peccati08}, providing an explicit formula for the kernel functions.
Then we proceed with developing a Malliavin calculus
for $\zeta$. To this end we introduce a gradient, a divergence and a generator
which act  as linear operators on random variables or random fields
and which are linked by some basic formulas such as integration by parts.
While this calculus is strongly motivated
by Malliavin calculus for isonormal Gaussian processes and the general Poisson process,  
the strong dependence properties of $\zeta$ require considerably more combinatorial
efforts. We apply our theory to identify our generator as the generator of the 
Fleming--Viot process and to describe  the associated Dirichlet form
explicitly in terms of the chaos expansion. We also establish the product and
chain rule for the  
gradient and an integral representation of the divergence.
Finally we give a short direct proof of the Poincar\'e inequality 
and establish a sharp reverse  Poincar\'e inequality, similarly to the one in
\cite{SchulteTrapp2024} for the Poisson process.
\end{abstract}

\bigskip

\noindent
{\bf Keywords:} 
Dirichlet--Ferguson process, Dirichlet distribution, 
Mecke type equation, chaos expansion, Malliavin calculus, Dirichlet forms, Fleming--Viot process,
Poincar\'e inequality, reverse  Poincar\'e inequality

\vspace{0.1cm}
\noindent
{\bf AMS MSC 2010:} 60G55, 60G57, 31C25

\section{Introduction}

We consider a {\em Dirichlet--Ferguson process} (DF process) $\zeta$ on some
measurable space $(\BX,\cX)$ with finite parameter measure
$\rho\ne 0$. This is a random purely discrete measure on $\BX$,
whose finite-dimensional distributions are Dirichlet distributions;
see \cite{Ferguson74} and, for instance, \cite{Kingman75,Sethuraman94,Pitman06,Feng2010,LastPenrose18}.
The random measure $\zeta$ is a benchmark model of a random probability
measure. It arises as a limit of a continuous
version of a P\'olya urn scheme \cite{BlackwellQueen73}.
The atom sizes of $\zeta$ can be constructed via a stick-breaking process 
based on a Beta distribution \cite{Pitman06}.
Its distribution is invariant and reversible for the
Fleming--Viot process of population genetics; see \cite{FlemingViot79,EthierKurtz93,Ethier90}.
The Dirichlet--Ferguson process also plays an important role in Bayesian statistics
and machine learning; see e.g.\ \cite{Murphy12}.

For ease of exposition, we assume in this introduction that $\zeta$ is defined over
a probability space $(\Omega,\mathcal{A},\BP)$, where
$\mathcal{A} =\sigma(\zeta)$ is the $\sigma$-field generated by $\zeta$.
A fundamental result from \cite{Peccati08} says that any $F\in L^2(\BP)$
can be written as an infinite series 
\begin{equation}\label{chaosdec}
F= \BE F + \sum_{n=1}^\infty \int f_n(x) \, \zeta^n(\dd x),\quad \BP\text{-a.s.}, 
\end{equation}
converging in $L^2(\BP)$. Here $\zeta^n$ is the $n$-th power of $\zeta$
(a random probability measure on $\BX^n$),
the $f_n$ are measurable (kernel) functions  on $\BX^n$
which are square-integrable w.r.t.\ the measure $\rho^{[n]}$ defined by 
\eqref{mu[n]} and conditionally centered in the sense of 
\eqref{condcent}. 
The summands in the {\em chaos expansion} \eqref{chaosdec}
are mutually orthogonal in $L^2(\BP)$. 
In Section \ref{sec:chaos} we shall rederive the chaos expansion \eqref{chaosdec}
and obtain the explicit representation \eqref{kerneln} for the kernel functions $f_n$.

For Gaussian orthonormal processes,
the chaos expansion is the basis of {\em Malliavin calculus}, a differential calculus
for random variables. The key ingredients of this calculus are linear operators
(gradient, divergence and generator) which are connected via several
relationships \cite{Nualart2006}. This calculus has proven to be very useful
for studying regularity properties of Gaussian functionals. In combination with Stein's method
it provides a powerful tool for deriving explicit bounds  for the normal approximation
of Gaussian functionals \cite{NourdinPeccati2012}.
Malliavin calculus has also been developed for   
the general Poisson  process (see \cite{Last2016} and the references there),
L\'evy processes (see e.g.\ \cite{SoleUtzet16}),
free Brownian motion \cite{BianeSpeicher98} and 
Rademacher sequences \cite{PrivSchout2002}.
In all these cases the underlying stochastic process enjoys strong independence properties.

In Section \ref{sec:Malliavin} we shall introduce and study Malliavin operators
for the Dirichlet--Ferguson process. Discrete versions of a gradient and a 
divergence were proposed in \cite{FlintTorrisi23}.
But to the best of our knowledge our work is the first to develop an intrinsic Malliavin calculus for
a strongly dependent 
process. 
We start by introducing for $F\in L^2(\BP)$ with chaos expansion \eqref{chaosdec}
the {\em gradient} $\nabla F$ by \eqref{defnabla}.
This gradient is a measurable function $(\omega,x)\mapsto \nabla_x F(\omega)$
defined by
\begin{align}\label{defnabla}
(\nabla F)& (\omega, x)\\ \notag
&:=  \sum_{n=1}^\infty n \left( \int f_n (x, y_1, \ldots, y_{n-1}) \, \zeta^{n-1}(\omega, \dd (y_1, \ldots, y_{n-1})) 
- \int f_n (y) \, \zeta^{n}(\omega, \dd y) \right)
\end{align}
provided the kernel functions satisfy \eqref{domNabla}. In the latter case $F$ belongs
to the domain $\operatorname{dom}(\nabla)$ of $\nabla$.
We write  $\nabla_x F$ for the random variable
$\omega\mapsto \nabla_x F(\omega)$. The series \eqref{defnabla}
converges in $L^2(C_\zeta)$, 
where $C_\zeta$ is the measure on $\Omega\times\BX$, defined by
\begin{align}\label{Campbell}
C_\zeta:=\iint \I\{(\omega,x)\in\cdot\}\,\zeta(\omega,\dd x)\,\BP(\dd\omega).
\end{align}
In the theory of random measures this is known as the {\em Campbell measure} of $\zeta$;
see \cite{Kallenberg17}.
For $H\in L^2(C_\zeta)$
we define the divergence $\delta(H)\in L^2(\BP)$ by the {\em partial integration}
formula
\begin{align*}
\BE \int H_x \nabla_xF\,\zeta(\dd x)=\BE \delta(H)F,\quad F\in \operatorname{dom}(\nabla),
\end{align*}
whenever $H$ belongs to the suitably defined domain $\operatorname{dom}(\delta)$ 
of $\delta$. 
Using the Campbell measure, we can rewrite this equation more succinctly  
as
\begin{align*}
C_\zeta(H\nabla F)=\BE \delta(H)F.
\end{align*}
We wish to stress that,
given $\omega\in\Omega$, we do not consider the function $x\mapsto\nabla_xF(\omega)$
as an element of a Hilbert space. In contrast to the independent cases mentioned above, 
our partial integration is not of the form
$\BE \langle H_\cdot, \nabla_{\cdot}F\rangle=\BE \delta(H)F$ 
for a suitable scalar product $\langle\cdot,\cdot\rangle$.
In the independent case the (suitably defined)
Campbell measure of $\zeta$ is close to a product measure with $\BP$ as first factor. 
This is not true for the DF process. In fact, we have
\begin{align*}
C_\zeta=\rho(\BX)^{-1}\int \I\{(\omega,x)\in\cdot\}\,\BP_x(\dd\omega)\,\rho(\dd x), 
\end{align*}
where $\BP_x$ is 
is the distribution of a DF process with directing measure $\rho+\delta_x$.
(The $\BP_x$, $x\in\BX$, are the {\em Palm distributions} of $\zeta$.) 
The third operator is a generator $L$ acting via \eqref{L} on all
$F$  from the domain $\operatorname{dom}(L)$ (defined by \eqref{domL}).
If $F\in \operatorname{dom}(\nabla)$, then $F\in \operatorname{dom}(L)$
and $\delta(\nabla)=-L$.

In Section \ref{sec:FlemingViot} we assume that $\BX$ is a locally compact Hausdorff space.
We consider the generator $L_\rho$ of the Fleming--Viot process, a second order differential
operator acting on smooth functions of $\zeta$, and the associated bilinear form.
We show that the closure of this form is given by
\begin{align}
\cE(F,G):=C_\zeta(\nabla F \nabla G),\quad F,G\in \operatorname{dom}(\nabla),
\end{align}
and that the closure of $2L_\rho$ is $L$. This gives an explicit description of these closures
in terms of Malliavin operators and justifies to call
$L$ the Fleming--Viot operator.
We also use chaos expansion to identify the strongly continuous semigroup $\{T_t:t\ge 0\}$ associated
with $L$.

In Section \ref{sec:furtherproperties} we derive the product and the chain rule for the
Malliavin gradient, which look exactly the same as in the Gaussian case 
\cite{Nualart2006,NourdinPeccati2012}. We also establish a pathwise representation of the divergence,
under appropriate assumptions on the integrand. These formulas are similar to the Gaussian case.
In fact, the continuous nature of jump sizes forces our Malliavin operators (when applied to smooth functions
of $\zeta$)  to be differential operators. This is in contrast to the Poisson case, where the gradient
is a pathwise defined difference operator \cite{Last2016}.
In Section \ref{sec:covariances} we present some formulas for the covariance between
specific functions of $\zeta$.
In Section \ref{sec:Poincare} we use chaos expansion
to give a short and direct proof of the Poincar\'e inequality 
for functions of the DF process. 
This inequality was derived in \cite{Stannat00} using the Poincar\'e inequality for the Dirichlet distribution 
(proved in \cite{Shimakura77}) and a suitable approximation. 
Inspired by the recent paper \cite{SchulteTrapp2024} we establish in Section \ref{s:reversePoincare} a 
(sharp) reverse  Poincar\'e inequality, using a second order difference operator.
In the appendix we collect some (combinatorial) integral formula which are
used in the main text.

\section{The Dirichlet--Ferguson process}\label{preliminaries}

We work on a probability space $(\Omega,\mathcal{A},\BP)$.
Let $(\BX, \cX,\rho)$ be a measure space where $\theta:=\rho(\BX)> 0$ and $\theta<\infty$.
Let $\bM_1\equiv\bM_1(\BX)$ be the space of all probability measures on
$\BX$ equipped with the smallest $\sigma$-field $\mathcal{M}_1$ such that the mapping
$\mu \mapsto\mu(B)$ from $\bM_1$ to $[0,1]$ is measurable for each
$B\in\cX$. A {\em Dirichlet--Ferguson process} (DF Process) $\zeta$ on $\BX$ with parameter 
measure $\rho$ is a random probability measure 
$\zeta$ on $\BX$ (a random element of the space $\bM_1(\BX)$) 
satisfying
\begin{align}\label{e1.1}
\BP((\zeta(B_1),\ldots,\zeta(B_n))\in\cdot)=\Di(\rho(B_1),\ldots,\rho(B_n)),
\end{align}
whenever $B_1,\ldots,B_n$, $n\ge 1$, form a measurable partition of $\BX$. 
Recall that  the {\em Dirichlet distribution} $\Di(\alpha_1,\dots,\alpha_n)$
with $n\ge 1$ parameters $\alpha_1,\dots,\alpha_n\in[0,\infty)$
such that $\alpha_1+\cdots+\alpha_n>0$ is  
the probability measure on 
$$
\Delta_n:=\{(p_1,\dots,p_n)\in[0,1]^n:p_1+\dots+p_n=1\}
$$
defined as follows. If $n\ge 2$ and $\alpha_1,\dots,\alpha_n\in(0,\infty)$,
then $\Di(\alpha_1,\dots,\alpha_n)$ has the density 
\begin{align*}
(x_1,\ldots,x_n)\mapsto \frac{\Gamma(\alpha_1+\dots+\alpha_n)}{\Gamma(\alpha_1)\cdots\Gamma(\alpha_n)}
x_1^{\alpha_1-1}\cdots x_{n}^{\alpha_{n}-1}
\end{align*}
with respect to (properly normalized) $(n-1)$-dimensional Hausdorff measure on $\Delta_n$.
Here $\Gamma(\alpha):=\int^\infty_0t^{\alpha-1}e^{-t}\,\dd t$, $\alpha>0$, denotes the Gamma function. 
If $n=1$, then $\Di(\alpha_1):=\delta_1$. If $n\ge 2$ and 
$\alpha_1=\cdots=\alpha_k=0$ for some $k\le n-1$, then
$\Di(\alpha_1,\dots,\alpha_n):=\delta_0^{\otimes k}\otimes \Di(\alpha_{k+1},\dots,\alpha_n)$.
A similar definition applies if some other set of the $\alpha_i$ vanish.
In the extreme case $\rho=\delta_x$ for some $x\in\BX$ we have $\zeta=\delta_x$.
The distribution of a DF process 
is denoted by  $\Di(\rho)\equiv\Di(\rho,\cdot)$, a probability measure on $\bM_1(\BX)$.  
Sometimes we write $\zeta_\rho:=\zeta$.

Given a measure $\mu$ on $\BX$ and $n \in \N$, we define a measure $\mu^{[n]}$ on $\BX^n$
by
\begin{equation}\label{mu[n]}
	\mu^{[n]}(B) = \idotsint \I_B (x_1, \ldots, x_n)  
(\mu + \delta_{x_1} + \cdots + \delta_{x_{n-1}})(\dd x_n) 
\cdots (\mu + \delta_{x_1})(\mathrm{d} x_2) \mu(\dd x_1).
\end{equation}
Note that
\begin{equation*}
\mu^{[n]}(B^n) = \mu(B) (\mu(B) + 1) \cdots (\mu(B) + n -1) = \Pn{\mu(B)}{n},\quad B \in \cX,
\end{equation*} 
where, for $w\in\R$,  $\Pn{w}{n}:=w(w+1)\cdots(w+n-1)$ is the $n$-th rising factorial of $w$.
As usual we set $\Pn{w}{0}:=1$.
It can be shown that $\mu^{[n]}$ is a symmetric measure. 
It satisfies the recursion
\begin{equation*}
\mu^{[n+1]}=\iint\I\{(x_1,\ldots,x_n,x)\,(\mu+\delta_x)^{[n]}(\dd (x_1,\ldots,x_n))\,\mu(\dd x), 
\end{equation*}
a special case of \eqref{Rekursion mu[m+n]}; see Convention \ref{convention0} for the case $n=0$.

Let $f\colon \bM_1(\BX) \times \BX^n \rightarrow [0,\infty)$ be measurable.
Then the multivariate Mecke-type equation for the DF process says that
\begin{align}\label{Mecke1}\notag
	\iint &f(\mu, x_1, \ldots, x_n) \,
\mu^n(\mathrm \dd (x_1, \ldots, x_n))\, \Di(\rho,\dd\mu) \\ 
&= \frac{1}{\Pn{\theta}{n}} \iint f(\mu, x_1, \ldots, x_n)\, 
\Di(\rho + \delta_{x_1} + \cdots + \delta_{x_n},\dd \mu)\, \rho^{[n]}(\dd (x_1, \ldots, x_n)).
\end{align}
This result can be found in \cite{Ferguson74}, at least in an implicit form.
It can be proved inductively, starting with the Mecke-type equation from 
\cite{DelloShiavoLytvynov23,Last20}.
Sometimes it is convenient to write \eqref{Mecke1} as
\begin{align*}
\BE \bigg[ \int f(\zeta_\rho, x_1, \ldots, x_n)& \, \zeta_{\rho}^n(\dd (x_1, \ldots, x_n))\bigg] \\
   & = \frac{1}{\Pn{\theta}{n}} \BE\bigg[\int f(\zeta_{\rho,x_1,\ldots,x_n},x_1, \ldots, x_n)\, 
\Pnm{\rho}{n}(\dd (x_1, \ldots, x_n))\bigg],
\end{align*}
where $\zeta_{\rho,x_1,\ldots,x_n}:=\zeta_{\rho + \delta_{x_1} + \cdots + \delta_{x_n}}$ is a DF process with
directing measure $\rho + \delta_{x_1} + \cdots + \delta_{x_n}$.
Equation \eqref{Mecke1} shows in particular that the $n$-th moment measure
of $\zeta$ is given by
\begin{equation}\label{nmoment}
\BE \zeta^n (B) = \frac{1}{\Pn{\theta}{n}} \rho^{[n]}(B),\quad B\in\cX^n.
\end{equation}

\section{Chaos expansion}\label{sec:chaos}

In this section we reprove the chaos expansion from \cite{Peccati08}  and derive an explicit
formula for the kernel functions.

For $n\in\N$ we let $\mathbf{H}_n$ denote the set of all symmetric functions $g \in L^2(\rho^{[n]})$ satisfying 
\begin{equation}\label{condcent}
\int g(x_1, \ldots, x_{n-1}, x) (\rho + \delta_{x_1} + \cdots + \delta_{x_{n-1}})(\dd x) = 0,
\quad \rho^{[n-1]}\text{-a.e.\ $(x_1, \ldots, x_{n-1}) \in \BX^{n-1}$.} 
\end{equation}
For $n=1$ this has to be interpreted as $\int g(x)\,\rho(\dd x)=0$.

Since we consider $\mathbf{H}_n$ as a subspace of $L^2(\rho^{[n]})$, the symmetry of
$f\in L^2(\rho^{[n]})$ means that $f$ coincides with a symmetric function
$\rho^{[n]}$-a.e. An  important property of the spaces $\bH_n$ is
\begin{align}\label{e:345}
\int g(x) h(y) \, \rho^{[m+n]}(\dd (x,y)) = \I\{m=n\} n! \int g(x) h(x) \,\rho^{[n]}(\dd x),
\quad g\in\mathbf{H}_m,h\in\mathbf{H}_n,
\end{align}
for $m,n\in\N$. More general versions of these orthogonality relations are proved in the appendix.
Together with \eqref{nmoment} this shows that
\begin{equation}\label{KorOrhtoFmFn}
\BE\left[ \int  g(x) \,\zeta^m(d x) \int h(y) \,\zeta^n(\dd y) \right] 
=\I\{m=n\} \frac{n!}{\Pn{\theta}{2n}} \int g(x) h(x)\, \rho^{[n]}(\dd x).
\end{equation}
As a consequence the spaces $\mathbf{F}_n$, $n\in\N_0$, 
defined by
\begin{equation*}
\mathbf{F}_n = \left\{ \int g(x) \, \zeta^n(\dd x): g \in \mathbf{H}_n \right\}
\end{equation*}
and $\mathbf{F}_0 := \R$ are pairwise orthogonal subspaces of $L^2(\BP)$.
We call $\bF_n$ the $n$-th {\em chaos} associated with $\zeta$.
Let $L^2(\zeta)$ be the space of all measurable $F\colon \bM_1(\BX)  \rightarrow \R$ 
such that $\BE F(\zeta)^2 < \infty$. 

\begin{convention}\rm Given $F\in L^2(\zeta)$ we often write $F:=F(\zeta)$.
If $(\Omega,\mathcal{A},\BP)$ is the canonical probability space 
$(\bM_1,\cM_1,\Di(\rho))$ (with $\zeta$ given as the identity), this is no
abuse of notation. In the general case this leads to a significantly lighter
notation, which should hopefully not cause any confusion. 
\end{convention}

\begin{convention}\label{convention0}\rm Let $\mu$ be a measure on $\BX$.
The space $\BX^0$ is identified with a singleton (as $\{\emptyset\}$, for instance)
and a $[-\infty,\infty]$-valued function on $\BX^0$ can be identified with a constant $c$. 
Then it is both natural and useful to set $\int_ {\BX^0}c\, \dd \mu^{0}=\int_ {\BX^0}c\, \dd \mu^{[0]}:=c$.
With this convention \eqref{KorOrhtoFmFn} remains true for $m=0$ or $n=0$,
upon interpreting $f$ (resp.\ $g$) as constant. If $\BY$ is a set and $f\colon\BY\to[-\infty,\infty]$
is a function, then we can interpret $f$ as a function on $\BY\times\BX^{0}$ and
we have $\int f(y,x)\,\mu^{0}(\dd x)=\int f(y,x)\,\mu^{[0]}(\dd x)=f(y)$, $y\in\BY$.
\end{convention}

The main result in \cite{Peccati08} (cf.\ equation (6) therein) 
says that any $F\in L^2(\zeta)$ can be written as an orthogonal series
\begin{equation}\label{EqChaosZerlegung}
F = \BE F + \sum_{n=1}^\infty \int f_n(x) \, \zeta^n(\dd x),\quad \BP\text{-a.s.}, 
\end{equation}
where the convergence is in $L^2(\BP)$ and
the {\em kernel functions} $f_n\in\mathbf{H}_n$ are uniquely determined
$\rho^{[n]}$-a.e. From \eqref{KorOrhtoFmFn} we obtain
\begin{equation}\label{e:iso0}
\BE FG=(\BE F)(\BE G)
+\sum^\infty_{n=1}\frac{n!}{\Pn{\theta}{2n}}\int f_n(x)g_n(x)\,\rho^{[n]}(\dd x)\,\quad F,G\in L^2(\zeta), 
\end{equation}
where $g_n$, $n\in\N$, are the kernel functions of $G$.

Given $F\in L^2(\zeta)$, our first main result shows that
the kernel functions $f_n$, $n\in\N$, of $F$ are given by 
\begin{equation}\label{kerneln}
f_n (x_1, \ldots, x_n)=  
\frac{\theta + 2n - 1}{n!} \sum_{j=0}^{n}  (-1)^{n-j} 
\Pn{(\theta+j)}{n-1} \sum_{1\le i_1 < \cdots < i_j\le n} \BE F(\zeta_{\rho,x_{i_1},\ldots,x_{i_j}}).
\end{equation}

\begin{theorem}\label{t:chaos} Suppose that $F\in L^2(\zeta)$.
Then the chaos expansion \eqref{EqChaosZerlegung} holds with the functions $f_n$, $n\in\N$,
given by \eqref{kerneln}. The functions $f_n$ are $\rho^{[n]}$-a.e.\ uniquely
determined by \eqref{EqChaosZerlegung}.
\end{theorem}

Equation \eqref{kerneln} provides a representation of the kernel functions
which is more explicit than \cite[eq.\ (10)]{Peccati08}, where the coefficients
are given through the limits in \cite[eq.\ (19)]{Peccati08}.

\begin{remark}\rm
Somewhat similar to the Poisson case 
(see \cite{LastPenrose18}), the function \eqref{kerneln} is an alternating sum
of (multivariate) Palm expectations. Indeed, by \eqref{Mecke1},
\begin{align*}
\{(\Pn{\theta}{n})^{-1}\Di(\rho + \delta_{x_1} + \cdots + \delta_{x_n}): (x_1,\ldots,x_n)\in\BX^n\}
\end{align*}
are the $n$-variate Palm distributions \cite{Kallenberg17} of $\zeta$. 
Formula \eqref{e:iso0} can then be seen as the DF analogue of the Fock space representation
of Poisson functionals derived in \cite{LastPenrose11}. However, in the DF case the kernel
functions cannot be represented as expected (pathwise defined) difference operators. 
\end{remark}

We split the proof of Theorem \ref{t:chaos} into several parts,
fixing $F\in L^2(\zeta)$. The first lemma shows already the uniqueness
part. In the proof and also later
we write $\nu(f)$ for the integral of a measurable function with respect to
a measure $\nu$.

\begin{Lemma}\label{LemEindProj}
Let $n \in \N$ and $g,h\in\bH_n$. Then
$\zeta^n(g)=\zeta^n(h)$ $\BP$-a.s. if and only if $g =h$ $\Pnm{\rho}{n}$-a.e.
\end{Lemma}
\begin{proof} By \eqref{KorOrhtoFmFn}, we have
\begin{equation*}
\BE (\zeta^n(h)-\zeta^n(g))^2 
=\BE (\zeta^n(h-g))^2 
=\frac{n!}{\Pn{\theta}{2n}} \rho^{[n]}(h-g)^2,
\end{equation*}
proving the result.
\end{proof}

It is convenient
to define measurable functions $T_{F,n}\colon\BX^n\to\R$, $n\in\N$, by
\begin{equation}\label{e:TFn}
T_{F,n}(x_1, \ldots, x_n) := \int F(\mu)\, \Di(\rho + \delta_{x_1} + \cdots + \delta_{x_n},\dd \mu)
=\BE F(\zeta_{\rho,x_1,\ldots,x_n});
\end{equation}
see \eqref{kerneln}. Further we set
\begin{equation*}
T_{F,0} = \BE F= \int  F(\mu) \, \Di(\rho, \dd \mu).
\end{equation*}
In the following proofs and some of the lemmas we often abbreviate
$(x_1,\ldots,x_n)\in\BX^n$ as $\bfm{x}_n$ or just $\bfm{x}$
if the meaning of $n$ is clear from the context. 
We also abbreviate $\delta_{\bfm{x}_n}:=\delta_{x_1}+\cdots+\delta_{x_n}$.

\begin{Lemma} \label{LemRekT_{F,n}} Let $k\in\N$. Then 
\begin{equation*}
    \int  T_{F,k} (x_1, \ldots, x_{k-1}, x) \, 
(\rho + \delta_{x_1}+\cdots+\delta_{x_{k-1}})(\dd x) = (\theta + k-1) T_{F, k-1} (x_1, \ldots, x_{k-1})
\end{equation*}
holds for $\Pnm{\rho}{k-1}$-a.e.\ $(x_1, \ldots, x_{k-1})\in \BX^{k-1}$, 
where the case $k=1$ has to be interpreted as $\int T_{F,1} (x) \, \rho (\dd x) = \theta T_{F, 0}$.
\end{Lemma}
\begin{proof}
We first consider the case $k=1$. We obtain from the
definition of $T_{F,1}$, the Mecke-type equation \eqref{Mecke1} and the definition of $T_{F,0}$ that
\begin{equation*}
\int  T_{F,1} (x) \, \rho (\dd x) 
= \int \BE F(\zeta_{\rho+\delta_{x}} ) \, \rho (\dd x) 
= \theta \BE  F(\zeta_\rho)
 = \theta T_{F,0}.
\end{equation*}
Let $k\ge 2$ and take a measurable function $g \colon \BX^{k-1}\to[0,1]$. 
By definition of $T_{F,k}$ and $\Pnm{\rho}{k}$, we have
\begin{align*}
    \iint T_{F,k} (\bfm{x}_{k-1}, x_k) &g(\bfm{x}_{k-1}) \, (\rho + \delta_{\bfm{x}_{k-1}})(\dd x_k) 
\, \Pnm{\rho}{k-1} (\dd \bfm{x}_{k-1}) \\
&=\int  \BE  F(\zeta_{\rho + \delta_{\bfm{x}_{k}}})  g(\bfm{x}_{k-1}) \, \Pnm{\rho}{k} (\dd \bfm{x}_{k}).
\end{align*}
By the Mecke-type equation \eqref{Mecke1} and the fact that $\zeta_{\rho}$ 
is a probability measure, the above equals
\begin{equation*}
    \Pn{\theta}{k} \BE \int  F(\zeta_\rho) g(\bfm{x}_{k-1}) \, \zeta_\rho^k(\dd\bfm{x}_{k}) 
    = \Pn{\theta}{k} \BE \int  F(\zeta_\rho) g(\bfm{x}_{k-1}) \, \zeta_\rho^{k-1}(\dd\bfm{x}_{k-1}). 
\end{equation*}
Using once again the Mecke-type equation and the definition of $T_{F,k-1}$, this equals
\begin{align*}
    \frac{\Pn{\theta}{k}}{\Pn{\theta}{k-1}} \int 
\BE &F(\zeta_{\rho+ \delta_{\bfm{x}_{k-1}}}) g(\bfm{x}_{k-1}) \, \Pnm{\rho}{k-1}(\dd\bfm{x}_{k-1})\\
& = (\theta + k-1) \int T_{F,k-1}(\bfm{x}_{k-1}) g(\bfm{x}_{k-1}) \, \Pnm{\rho}{k-1}(\dd\bfm{x}_{k-1}). 
\end{align*}
This concludes the proof.
\end{proof}

\begin{lemma}\label{l:orthHn} Let $n\in\N$. The space $\bH_n$ is a closed
subspace of $L^2(\rho^{[n]})$, while $\bF_n$ is a closed subspace of $L^2(\zeta)$.
\end{lemma}
\begin{proof}
Let $(g_m)_{m\ge 1}$ be a sequence in
$\bH_n$ with $g_m \to g$ in $L^2(\rho^{[n]})$ as $m\to\infty$, for some function
$g\in L^2(\Pnm{\rho}{n})$. Since $\rho^{[n]}$ is symmetric, it easily follows that
$g$ is symmetric.

By definition, we have $g\in\bH_n$ if and only if
\begin{equation*}
\int  h(x_1, \ldots, x_{n-1}) g(x_1, \ldots, x_n) \, \Pnm{\rho}{n}(\dd (x_1, \ldots, x_n)) = 0 
\end{equation*}
for all measurable  $h \colon \BX^{n-1} \rightarrow [0,1]$. Taking such a $h$, we 
clearly have that $hg_m\to hg$ in $L^2(\rho^{[n]})$ as $m\to\infty$.
In particular, we have $0=\int hg_m\,\dd\rho^{[n]}\to \int hg\,\dd\rho^{[n]}$, 
proving the first assertion.

To prove the second assertion, we take $g_m\in\bH_n$, $m\in\N$, such that
the sequence $(\zeta^n(g_m))_{m\ge 1}$ converges in $L^2(\BP)$. It then follows from
\eqref{KorOrhtoFmFn} that $(g_m)_{m\ge 1}$ is a Cauchy sequence in $\bH_n$.
By the first assertion, the sequence has a limit $g$, say. 
Again by \eqref{KorOrhtoFmFn}, we obtain that
$\BE (\zeta^n(g_m)-\zeta^n(g))^2=\BE (\zeta^n(g_m-g))^2=(\theta^{(n)})^{-1}\rho^{[n]}((g_m-g)^2)$
which goes to $0$ as $m\to\infty$.
\end{proof}

\begin{lemma}\label{l:fnHn} Let $n\in\N$. Then $f_n\in \bH_n$, where $f_n$ is defined by
\eqref{kerneln}.
\end{lemma}
\begin{proof}
By $(c_1+\cdots+c_k)^2\leq k\sum_{i=1}^k c_i^2$ for
  all $c_1, \ldots, c_k \in\R$ and $k\in\N$ and the multivariate
  Mecke-type equation~\eqref{Mecke1}, the integral
  $\Pnm{\rho}{n}(f_n^2)$ is bounded by
\begin{equation*}
    a_n\int\big( \Pn{\theta}{n-1} \BE F(\zeta_{\rho})\big)^2 \, \Pnm{\rho}{n}(\dd x) + a_n n \sum_{j=1}^{n}  \left(\Pn{(\theta+j)}{n-1}\right)^2 \binom{n}{j}^2 \Pn{\theta}{n} 
\BE F(\zeta_{\rho})^2,
\end{equation*}
where $a_n:=(2(\theta + 2n - 1)^2)/(n!)^2$. The finiteness of the first integral in this expression follows from
Jensen's inequality and $\Pnm{\rho}{n}(\BX^n)<\infty$.
Hence, we conclude that $f_n$ is an element of $L^2(\Pnm{\rho}{n})$. 
Moreover, $f_n$ is symmetric. In order to show $f_n \in \bH_n$, it remains to establish  
\begin{equation*}
    \int f_n(\bfm{x}_{n-1}, x_n) \, (\rho + \delta_{\bfm{x}_{n-1}})(\dd x_n) = 0, \quad \Pnm{\rho}{n-1}\text{-a.e. $(x_1, \ldots, x_{n-1}$)}.
\end{equation*}
Let $(x_1, \ldots, x_{n-1}) \in \BX^{n-1}$ and set $b_n:=\frac{\theta + 2n - 1}{n!}$. By definition of $f_n$, the preceding integral equals
\begin{align*} 
 &b_n(-1)^n \Pn{\theta}{n-1}(\theta+n-1)T_{F,0} \\
&+ b_n\sum_{j=1}^{n}  (-1)^{n-j} \Pn{(\theta+j)}{n-1} \int \sum_{i_1 < \cdots < i_j} 
T_{F,j}(x_{i_1}, \ldots, x_{i_{j}}) \, (\rho + \delta_{\bfm{x}_{n-1}} )(\dd x_n)=:b_nA_1+b_nA_2.
\end{align*}
Let us write
\begin{align*}
A_2&=\Pn{(\theta+n)}{n-1} \int \sum_{i_1 < \cdots < i_n} T_{F,n}(\bfm{x}_n) \, (\rho + \delta_{\bfm{x}_{n-1}} )(\dd x_n)\\
&\quad+\sum_{j=1}^{n-1}  (-1)^{n-j} \Pn{(\theta+j)}{n-1} \int \sum_{i_1 < \cdots < i_j} T_{F,j}(x_{i_1}, \ldots, x_{i_{j}}) 
\, (\rho + \delta_{\bfm{x}_{n-1}} )(\dd x_n)=:A_{2,1}+A_{2,2}.
\end{align*}
By Lemma~\ref{LemRekT_{F,n}}, 
\begin{align*}
A_{2,1}=(\theta+n-1)^{(n)} T_{F,n-1}(\bfm{x}_{n-1})
\end{align*}
whereas $A_{2,2}$ equals 
\begin{align*}
    A_{2,2} &= \sum_{j=1}^{n-1} \sum_{i_1 < \cdots < i_j \leq n-1} (-1)^{n-j} \Pn{(\theta+j)}{n-1} T_{F,j} (x_{i_1}, \ldots, x_{i_j}) (\theta + n-1) \\
        &\quad+ \sum_{j=1}^{n-1} \sum_{i_1 < \cdots < i_{j-1} \leq n-1} \int (-1)^{n-j} \Pn{(\theta+j)}{n-1} 
T_{F,j} (x_{i_1}, \ldots, x_{i_{j-1}}, x_n) \, (\rho + \delta_{\bfm{x}_{n-1}})(\dd x_n) \\
        &\eqqcolon A_{2,2,1} + A_{2,2,2}.
\end{align*}

Lemma~\ref{LemRekT_{F,n}}, the symmetry of $T_{F,j}$ and the combinatorial identity
\begin{align*}
\sum_{i_1 < \cdots < i_{j-1} \leq n-1} &\sum_{{\substack{l=1\\ l \notin \{i_1, \ldots, i_{j-1}\}}}}^{n-1} (-1)^{n-j} \Pn{(\theta+j)}{n-1} 
T_{F,j} (x_{i_1}, \ldots, x_{i_{j-1}}, x_l) \\
    &=  j \sum_{i_1 < \cdots < i_j \leq n-1} (-1)^{n-j} \Pn{(\theta+j)}{n-1} T_{F,j} (x_{i_1}, \ldots, x_{i_{j}}), \quad j\in\{2,\ldots, n-1\},
\end{align*}
yield
\begin{align*}
    A_{2,2,2}&=  \sum_{l=1}^n (-1)^{n} \Pn{(\theta+1)}{n-1}(\theta+n-1)  T_{F,1}(x_l)  - (n-1) \Pn{(\theta+n-1)}{n-1} T_{F,n-1} (\bfm{x}_{n-1})    \\
    &\quad- (-1)^{n} \Pn{\theta}{n} T_{F,0} - \sum_{j=2}^{n-2}  \sum_{i_1 < \cdots < i_j \leq n-1} 
(-1)^{n-j} \Pn{(\theta+j)}{n-1}(\theta+n-1) T_{F,j} (x_{i_1}, \ldots, x_{i_{j}})
\end{align*}
so that
\begin{equation*}
    A_{2,2} = A_{2,2,1} + A_{2,2,2} = - (-1)^{n} \Pn{\theta}{n} T_{F,0} - (\theta + 2n -2) \Pn{(\theta+n-1)}{n-1} T_{F,n-1} (\bfm{x}_{n-1}). 
\end{equation*}
Thus,
\begin{equation*}
    A_1 + A_2 = A_1 + A_{2,1} + A_{2,2} = 0. 
\end{equation*}
Hence, $f_n$ is an element of $\bH_n$.    
\end{proof}

\begin{proposition}\label{l:fnprojection} Let $n\in\N_0$. Then $\zeta^n(f_n)$ is the orthogonal
projection of $F$ onto $\bF_n$, where $\zeta^0(f_0)=f_0 :=\BE F$.
\end{proposition}
\begin{proof}
The last claim follows from $\BE F(\zeta)c = c \,\BE F(\zeta)$ for all $c\in\mathbf{F}_0 = \R$. 
Let $n\in\N$ and $g \in \bH_n$.
On the one hand, the multivariate Mecke-type equation~\eqref{Mecke1} for the DF process yields 
\begin{equation} \label{eq Bew Projformel li Seite}
	\BE F(\zeta_\rho) \int g(x) \, \zeta_\rho^n(\dd x) 
    = \frac{1}{\Pn{\theta}{n}} \int \BE F(\zeta_{\rho+ \delta_{\bfm{x}_{n}}}) g(\bfm{x}_{n})  \, \Pnm{\rho}{n}(\dd \bfm{x}_{n}).
\end{equation}
On the other hand, \eqref{KorOrhtoFmFn} implies
\begin{equation*} 
    \BE \left[ \int f_n(y)  \, \zeta_\rho^n(\dd y) \int g(x)  \, \zeta_\rho^n(\dd x) \right] = \frac{n!}{\Pn{\theta}{2n}} \int f_n (x) g(x)\,\Pnm{\rho}{n}(\dd x).
\end{equation*}
By the definition of $f_n$, the right-hand side equals
\begin{equation*}
    \frac{n!}{\Pn{\theta}{2n}} \int\frac{\theta + 2n - 1}{n!}  
\Bigg(\sum_{j=0}^{n}  (-1)^{n-j} \Pn{(\theta+j)}{n-1} \sum_{1\le i_1 < \cdots < i_j\le n} 
T_{F,j}(x_{i_1}, \ldots, x_{i_{j}})\Bigg)   g(\bfm{x}_{n})  \, \Pnm{\rho}{n}(\dd \bfm{x}_{n}).  
\end{equation*}
Since $g$ is an element of $\bH_n$ and thus Corollary~\ref{KorIntegrationInIH_mAndIH_N} (with $k=j$)
is applicable, this reduces to
\begin{equation*}
    \frac{n!}{\Pn{\theta}{2n}} \int \left(\frac{\theta + 2n - 1}{n!}   
(-1)^{n-n} \Pn{(\theta+n)}{n-1} T_{F,n}(\bfm{x}_{n})\right) g(\bfm{x}_{n})  \, \Pnm{\rho}{n}(\dd \bfm{x}_{n})  
\end{equation*}
which is, because of $T_{F,n}(\bfm{x}_{n}) = \BE F(\zeta_{\rho+ \delta_{\bfm{x}_{n}}}) $, 
in turn equal to~\eqref{eq Bew Projformel li Seite}.
Therefore, $\zeta^n(f_n)$ is indeed the orthogonal projection of $F$ onto $\mathbf{F}_n$. 
\end{proof}

In the following we abbreviate  $[m]\coloneqq \{1,\ldots,m\}$ for $m\in\N$.

\begin{lemma}\label{l:3.5} Let $m\in\N$ and assume that $F=\zeta^m(f)$ for some
$f\in L^2(\rho^{[m]})$. Then we have for each $k\in [m]$ and for $\rho^{[k]}$-a.e.\
$(x_1,\ldots,x_k)$ that
\begin{align}\label{e:3.25}
f_k&(x_1,\ldots,x_k)\\ \notag
&=    \frac{\theta + 2k - 1}{k!}
\left(\sum_{j=0}^{k}  (-1)^{k-j} \frac{\Pn{(\theta+j)}{k-1}}{\Pn{(\theta+j)}{m}} \sum_{1 \le i_1 < \cdots < i_j\le k} 
\Pnm{(\rho + \delta_{x_{i_1}} + \cdots + \delta_{x_{i_j}})}{m}(f)  
      \right),
\end{align}
where the term with $j=0$  has to be interpreted as
$(-1)^{k} (\Pn{\theta}{k-1}/\Pn{\theta}{m}) \Pnm{\rho}{m}(f)$.
Furthermore, we have for $n\ge m+1$ that $f_n=0$ $\rho^{[n]}$-a.e.
\end{lemma}
\begin{proof} To see \eqref{e:3.25} we have to use formula
\eqref{kerneln} with  $F(\mu)=\mu^m(f)$. 
The result follows from the moment formula \eqref{nmoment}.

To prove the second assertion, we take $n>m$ and  $g \in \bH_n$. By
\eqref{nmoment},
\begin{equation*}
\BE \zeta^m(f)\zeta^n(g) 
= \frac{1}{\Pn{\theta}{m+n}} \int f(x) g(y)\,\Pnm{\rho}{m+n} (\dd (x,y)) = 0,
\end{equation*}
where the second equality follows from
Corollary~\ref{KorIntegrationInIH_mAndIH_N} (with $k=0$). 
Hence, $F$ is orthogonal to $\bF_n$. Proposition \ref{l:fnprojection} and
Lemma \ref{LemEindProj} then show that $f_n=0$ $\rho^{[n]}$-a.e.
\end{proof}

For $m\in\N$ and a function $f\colon\BX^m\to\R$, we denote by
\begin{equation*}
\tilde{f}(x_1,\ldots,x_m):=\frac{1}{n!}\sum_\pi f(x_{\pi(1)},\ldots,x_{\pi(m)}),\quad x_1,\ldots,x_m\in\BX, 
\end{equation*}
the {\em symmetrization} of $f$, where the summation is over all
permutations $\pi$ of $[m]$.

\begin{lemma}\label{l:3.7} Let $m\in\N$ and $f\in L^2(\rho^{[m]})$.
Assume for all $k\in\{0,\ldots,m\}$ that
\begin{equation}\label{e:3.4}
(\rho + \delta_{x_1} + \cdots + \delta_{x_k})^{[m]}(f)=0,\quad \rho^{[k]}\text{-a.e.\ $(x_1,\ldots,x_k)$}.  
\end{equation}
Then $\tilde{f}=0$ $\rho^{[m]}$-a.e.
\end{lemma}
\begin{proof} Since $(\rho + \delta_{x_1} + \cdots + \delta_{x_k})^{[m]}$ is a 
symmetric measure, we can assume without loss
of generality that $f$ is symmetric. Assume $k\ge 1$ and take a measurable function $h\colon\BX^k\to[0,1]$.
By \eqref{Rekursion mu[m+n]} and \eqref{nmoment},
\begin{align*}
\Pn{\theta}{k+m}&\int h(x_1,\ldots,x_k) (\rho + \delta_{x_1} + \cdots + \delta_{x_k})^{[m]}(|f|)\,\rho^{[k]}(\dd (x_1,\ldots,x_k))\\
&= \Pn{\theta}{k+m} \int h(x_1,\ldots,x_k) |f(x_{k+1}, \ldots, x_{k+m})|\,\rho^{[k+m]}(\dd (x_1,\ldots,x_{k+m})) \\
&=\BE \zeta^k(h) \zeta^m(|f|) 
\le (\Pnm{\theta}{m})^{-1}\Pnm{\rho}{m}(|f|). 
\end{align*}
The last term is finite since $f\in L^2(\rho^{[m]})$ and $\rho^{[m]}$ is a finite measure. Therefore, we have
$(\rho + \delta_{x_1} + \cdots + \delta_{x_k})^{[m]}(|f|)<\infty$ for  $\rho^{[k]}$-a.e.\ $(x_1,\ldots,x_k)$.
In the remainder of the proof we ignore $\rho^{[m]}$-null sets.

Since $f$ is symmetric, we obtain from Proposition \ref{PropRho+delta_x_1+...+delta_x_k} that
\begin{align}\label{eq987}
   \int & f (z) \,\Pnm{(\rho + \delta_{x_1} + \cdots + \delta_{x_k})}{m}(\dd z) = \int f(z)  \,\Pnm{\rho}{m}(\dd z)\\ \notag
 &   + \sum_{r=1}^{m} \frac{m!}{(m-r)!}\sum_{1\le j_1\le\cdots\le j_r\le k} \int f (x_{j_1}, \ldots, x_{j_r}, z_1, \ldots, z_{m-r}) 
\,\Pnm{\rho}{m-r}(\dd (z_1, \ldots, z_{m-r}))
\end{align}
for all $k\in\{0,\ldots,m\}$ and all $(x_1,\ldots,x_m)\in\BX^m$. 
The case $k=0$ of our assumption \eqref{e:3.4} yields 
$\rho^{[m]}(f)=0$, so that the case $k=1$ gives
\begin{equation*}
\sum_{r=1}^{m} \frac{m!}{(m-r)!}\int f (x_1, \ldots, x_1, z_1, \ldots, z_{m-r}) \,\Pnm{\rho}{m-r}(\dd (z_1, \ldots, z_{m-r}))=0.
\end{equation*}
Since $\rho^{[m]}$ is symmetric, we can replace $x_1$ with $x_2$ here.
Therefore, we obtain from the case $k=2$ of \eqref{e:3.4} that
\begin{equation*}
\sum_{r=2}^{m} \frac{m!}{(m-r)!}\quad \sideset{}{^{*}}\sum_{1\le j_1\le\cdots\le j_r\le 2} \int f (x_{j_1}, \ldots, x_{j_r}, z_1, \ldots, z_{m-r}) 
\,\Pnm{\rho}{m-r}(\dd (z_1, \ldots, z_{m-r}))=0,
\end{equation*}
where the upper index * means  that the sum is restricted to those $j_1,\ldots, j_r\in\{1,2\}$ satisfying $\{1,2\}\subset \{j_1,\ldots, j_r\}$.
Continuing this way all the way up to $k=m$, we see that indeed $f(x_1,\ldots,x_m)=0$.
\end{proof}

\begin{lemma}\label{l:3.9} Let $m\in\N_0$ and $f\in L^2(\rho^{[m]})$. Assume that
$\zeta^m(f)$ is orthogonal to $\bF_n$ for each $n\in\{0,\ldots,m\}$. Then $\tilde{f}=0$ $\rho^{[m]}$-a.e. 
\end{lemma}
\begin{proof} Since $\zeta^m(f)$ is orthogonal to $\bF_0=\R$, we obtain $\BE \zeta^m(f)=(\theta^{(m)})^{-1}\rho^{[m]}(f)=0$. 
We use Lemma \ref{l:3.5}. It follows from our assumption, Proposition \ref{l:fnprojection}
and Lemma \ref{l:3.5}
that $f_k=0$ $\rho^{[k]}$-a.e.\ for each $k\in [m]$, where the $f_k$ are given by \eqref{e:3.25}. 
Since $\rho^{[m]}(f)=0$, we obtain  for $k=1$ that
$\Pnm{(\rho + \delta_{x})}{m}(f) =0$ for $\rho$-a.e.\ $x$. Since $\rho^{[2]}$ is symmetric
and $\rho^{[2]}(\cdot\times \BX)=(\theta+1)\rho$, we obtain for $k=2$ that 
$(\rho + \delta_{x_1} + \delta_{x_2})^{[m]}(f)=0$  for $\rho^{[2]}$-a.e.\ $(x_1,x_2)$.
Inductively it follows for each $k\in [m]$  that \eqref{e:3.4} holds.
Therefore, we obtain the assertion from Lemma \ref{l:3.7}.
\end{proof}

\begin{proof}[Proof of Theorem \ref{t:chaos}] In view of Proposition \ref{l:fnprojection} 
it remains to show that
\begin{align}\label{e:oplus}
\bigoplus^\infty_{n=0} \bF_n=\{F(\zeta): F\in L^2(\zeta)\}.
\end{align}
It is not difficult to see that the closure of the span of
$\{\zeta^m(f):f\in L^2(\rho^{[m]}),m\in\N_0\}$ coincides with the right-hand side of \eqref{e:oplus};
see \cite[Lemma 2]{Peccati08}.
On the other hand, the left-hand side of \eqref{e:oplus} is the orthogonal sum of
closed spaces and hence closed. Therefore, it suffices to show that 
$\zeta^m(f)$ belongs to the left-hand side of \eqref{e:oplus},
where $f\in L^2(\rho^{[m]})$ for some $m\in\N_0$.
Define $G:=\zeta^m(f)-\sum^m_{k=0}\zeta^k(f_k)$, where $f_1,\ldots,f_m$ are given by
\eqref{e:3.25} and $f_0:=\BE \zeta^m(f)$.
By Proposition \ref{l:fnprojection} and Lemma \ref{l:3.5} (and Lemma \ref{l:fnHn}), the orthogonal projection
of $G$ onto $\bF_n$ vanishes a.s.\ for each $n\in\N_0$. Since $\zeta$ is a probability measure,
we can write $G=\zeta^m(f^*)$ for a suitably defined $f^*\in L^2(\rho^{[m]})$. Lemma \ref{l:3.9} shows
that the symmetrization of $f^*$ vanishes $\rho^{[m]}$-a.e. Since $\zeta^m$ is symmetric,
we obtain  that $G=0$ $\BP$-a.s.  Therefore,
$\zeta^m(f)=\sum^m_{k=0}\zeta^k(f_k)$, as required to conclude the proof.
\end{proof}

\section{Malliavin operators}\label{sec:Malliavin}

In this section we use the chaos expansion to introduce Malliavin operators acting on
random variables and fields, which depend measurably on $\zeta$. Our approach is
inspired by the Malliavin calculus for classical isonormal Gaussian processes \cite{Nualart2006} 
and the general Poisson  process \cite{Last2016}. 
In the next section we shall see that the Malliavin  operators admit a natural
interpretation in the context of  Fleming--Viot processes.

\subsection{The gradient}

Let $\operatorname{dom}(\nabla)$ denote the set of all $F \in L^2(\zeta)$ with chaos decomposition
\eqref{EqChaosZerlegung} such that
\begin{equation}\label{domNabla}
     \sum_{n=1}^\infty \frac{(\theta + n -1)nn!}{\Pn{\theta}{2n}}\int f_n(x)^2 \, \Pnm{\rho}{n}(\dd x) < \infty.
\end{equation}
Define $\nabla \colon \operatorname{dom}(\nabla) \rightarrow L^2(C_\zeta)$ 
by \eqref{defnabla}.
This definition is justified by the following proposition.

\begin{proposition}\label{p:convnabla} Suppose that $F\in\operatorname{dom}(\nabla)$.
Then the right-hand side of \eqref{defnabla} converges in $L^2(C_\zeta)$.
\end{proposition}
\begin{proof}
Given $n\in\N$, let $H_n\colon \Omega \times \BX \rightarrow \R$ be defined by
\begin{align*}
	H_n(\omega,x) &\coloneqq \int f_n (x, \bfm{y}_{n-1}) \, \zeta^{n-1}(\omega, \dd \bfm{y}_{n-1}) 
- \int f_n (\bfm{y}) \, \zeta^{n}(\omega, \dd \bfm{y}). 
\end{align*}
Let $m,n\in\N$. Using\eqref{nmoment} and appplying Corollary~\ref{KorIntegrationInIH_mAndIH_N} twice, 
we obtain
\begin{align*}
 \int   H_m(\omega,x) & H_n(\omega,x) \, C_{\zeta} (\dd (\omega,x)) 
= \BE  \int  H_m (x) H_n(x) \, \zeta(\dd x)  \\
 &= \I\{m=n\} \frac{(n-1)!}{\Pn{\theta}{2n}}(\theta + n -1) \int  f_n(\bfm{x})^2 
\, \Pnm{\rho}{n}(\dd \bfm{x}).
\end{align*}
Hence, we obtain  for each $n_0\in\N$ that
\begin{equation*}
    \int  \left(\sum_{n=1}^{n_0} n H_n(\omega,x) \right)^2 \, C_{\zeta} (\dd (\omega,x)) 
    = \sum_{n=1}^{n_0} \frac{n^2 (n-1)!}{\Pn{\theta}{2n}}(\theta + n -1) 
\int  f_n(\bfm{x})^2 \, \Pnm{\rho}{n}(\dd \bfm{x}).
\end{equation*}
Since $F\in \operatorname{dom}(\nabla)$,
we conclude that 
$\sum_{n=1}^{n_0} n H_n$, $n_0\in\N$,
is a Cauchy sequence in $L^2(C_{\zeta})$ and therefore convergent. 
The limit is $\nabla F$.
Moreover, we see that
\begin{equation} \label{eq Gamma Gradient}
    \BE \int \left( \sum_{n=1}^\infty n H_n(x) \right)^2 \, \zeta(\dd x) 
= \sum_{n=1}^\infty \frac{nn!}{\Pn{\theta}{2n}}(\theta + n -1) 
\int  f_n(\bfm{x})^2 \, \Pnm{\rho}{n}(\dd \bfm{x}). \qedhere
\end{equation}
\end{proof}

The operator $\nabla$ is linear in the following (natural) sense. If $F,G\in\operatorname{dom}(\nabla)$ 
and $a,b\in\R$, then  $aF+bG\in\operatorname{dom}(\nabla)$ and  $\nabla(aF+bG)=a\nabla F+b\nabla G$ holds
$C_\zeta$-a.e. We work with a measurable version of the gradient 
(which can be constructed via a $C_\zeta$-a.e.\ convergent subsequence)
and often suppress the dependence on $\omega$, 
by writing $\nabla_x F$, $x\in\BX$, for the random variable $\omega\mapsto (\nabla F) (\omega, x)$.
The next proposition provides an important isometry property.

\begin{proposition} Suppose that $F,G\in \operatorname{dom}(\nabla)$. Then
\begin{align}\label{e:iso}
\BE\int \nabla_x F \nabla_x G\,\zeta(\dd x)
=\sum^\infty_{n=1}\frac{(\theta+n-1) nn!}{\Pn{\theta}{2n}}\rho^{[n]}(f_ng_n),\quad F,G\in \operatorname{dom}(\nabla),
\end{align}
where $f_n$ (resp.\ $g_n$), $n\in\N$,  are the kernel functions of $F$ (resp.\ $G$).
\end{proposition}
\begin{proof} For $F=G$ we obtain the asserted formula from \eqref{eq Gamma Gradient}.
The general case can either be derived by the same calculation or by polarization.
\end{proof}

The gradient is pathwise centred w.r.t.\ $\zeta$, that is  
\begin{equation}\label{gradientmeanzero}
\int \nabla_xF\,\zeta(\dd x)=0,\quad \BP\text{-a.s.}
\end{equation}
To see this, we may define for $m\in\N$ and $x\in\BX$ the random variable $H^m_x$
by truncating the infinite series \eqref{defnabla} at $m\in\N$. By definition, we then have
$\int H^m_x\,\zeta(\dd x)=0$. On the other hand, it follows from the $L^2(C_\zeta)$-convergence and 
Jensen's inequality (applied to $\zeta$) that $\int H^m_x\,\zeta(\dd x)$ converges to
$\int \nabla_xF\,\zeta(\dd x)$ in $L^2(\BP)$ as $m\to\infty$. Therefore,
\eqref{gradientmeanzero} follows.

The next result shows that the gradient is a {\em closed operator}.

\begin{lemma} \label{LemGradientAbg}
Let $F_n\in\operatorname{dom}(\nabla)$, $n\in \N$, and assume that $(\nabla F_n)_{n\in\N}$ 
forms a Cauchy sequence in $L^2(C_{\zeta})$.
Then there exists $F\in\operatorname{dom}(\nabla)$ with
\begin{equation} \label{eq LemGammaAbg}
    \lim\limits_{n\to\infty} \BE \int (\nabla_x F - \nabla_x F_n)^2\, \zeta(\dd x) = 0
\end{equation}
and $F_n-\BE F_n \to F$ in $L^2(\BP)$.
\end{lemma}
\begin{proof}
Let for each $n\in\N$ the chaos expansion  \eqref{EqChaosZerlegung} of $F_n \in L^2(\zeta)$ be given by
\begin{equation*}
    F_n (\zeta) = \BE F_n + \sum_{k=1}^\infty \int f_{n,k}(x) \, \zeta^k(\dd x). 
\end{equation*}
The assumed convergence and \eqref{e:iso} yield
\begin{align*}
    0 &=  \lim_{m,n\to\infty} \BE \left[ \int (\nabla_x F_m - \nabla_x F_n)^2\, \zeta(\dd x) \right] \\
    &= \lim\limits_{m,n\to\infty} \sum_{k=1}^\infty \frac{kk!}{\Pn{\theta}{2k}}(\theta + k -1) 
\int   (f_{m,k}(x) - f_{n,k}(x))^2 \, \Pnm{\rho}{k} (\dd x).
\end{align*} 
Set
\[
a_k\coloneqq\frac{k k!(\theta+k-1)}{\binom{\theta}{2k}}, \quad k\in\N.
\]
By the assumed convergence and~\eqref{e:iso},
\[
\lim_{m,n\to\infty}
\sum_{k=1}^\infty a_k
\int (f_{m,k}-f_{n,k})^2\,\Pnm{\rho}{k}(\dd x)=0.
\]
Hence, for every $k\in\N$, $(f_{n,k})_{n\in\mathbb N}$ is a Cauchy
sequence in $L^2(\Pnm{\rho}{k})$. Thus, there exists
$f_k\in L^2(\Pnm{\rho}{k})$ such that $f_{n,k}\longrightarrow f_k$ in $L^2(\Pnm{\rho}{k})$.
By Lemma~\ref{l:orthHn}, we have $f_k\in\bH_k$. Moreover, Fatou's lemma gives
\begin{align*}
\sum_{k=1}^\infty a_k
\int f_k^2\,\Pnm{\rho}{k}(\dd x)
&\leq
\liminf_{n\to\infty}\sum_{k=1}^\infty a_k\int f_{n,k}^2\,\Pnm{\rho}{k}(\dd x)\\
&=\lim_{n\to\infty}\sum_{k=1}^\infty \BE \left[ \int (\nabla_x F_n)^2\, \zeta(\dd x)\right]<\infty.
\end{align*}
Consequently,
\[
F\coloneqq\sum_{k=1}^\infty \int f_k(x)\,\zeta^k(\dd x)
\]
belongs to $\operatorname{dom}(\nabla)$ and satisfies \eqref{eq LemGammaAbg}.

The second assertion follows from \eqref{e:iso0} and \eqref{e:iso}, showing that
the convergence of $\nabla F_n\to\nabla F$ in $L^2(C_\zeta)$ implies the asserted
$L^2(\BP)$-convergence.
\end{proof}

\subsection{The divergence}

In this subsection we introduce and discuss the {\em divergence operator} 
acting as an adjoint of $\nabla$. 
To do so, it is convenient to introduce the canonical version $C'_\zeta$ of the
Campbell measure \eqref{Campbell}, which is the measure on $\bM_1\times\BX$ defined by
\begin{align}\label{Campbellcan}
C'_\zeta:=\iint \I\{(\mu,x)\in\cdot\}\,\mu(\dd x)\,\BP(\zeta \in\dd\mu).
\end{align}
Let 
$\operatorname{dom}(\delta)$ denote the set of all
$H\in L^2(C'_\zeta)$ for which there exists a  $c\ge 0$ such that
\begin{equation*}
\BE \int H(\zeta,x)\nabla_x G\,\zeta(\dd x)\le c\, \|G(\zeta)\|_2,\quad G\in \operatorname{dom}(\nabla),
\end{equation*}
where $\|\cdot\|_2$ denotes the norm in $L^2(\BP)$. Since $\operatorname{dom}(\nabla)$
is dense in $L^2(\zeta)$, 
we can apply  the Riesz representation theorem to obtain for each 
$H\in\operatorname{dom}(\delta)$ the existence of an $\BP$-a.s.\ uniquely determined 
$\sigma(\zeta)$-measurable $\delta(H)\in L^2(\BP)$ satisfying
\begin{equation}\label{PI}
\BE \int H_x \nabla_xG\,\zeta(\dd x)=\BE \delta(H)G,\quad G\in \operatorname{dom}(\nabla),
\end{equation}
where we write $H_x:=H(\zeta,x)$, $x\in\BX$. This equation is often referred to as
{\em partial integration}. It implies that $\delta$ is a linear operator. 
Choosing $G\equiv 1$ shows that $\BE \delta(H)=0$.

We wish to identify a reasonably large subset
of $\operatorname{dom}(\delta)$ along with an explicit formula for the divergence
in terms of chaos expansion. The next lemma is crucial.
In the proof and also later we use the notation
$\sideset{}{^{\ne}}\sum_{i_1,\dots,i_r\in[m]} $ to indicate a summation 
over $r$-tuples $(i_1,\dots,i_r)\in [m]^r$ ($m,r\in\N$) with pairwise different entries.

\begin{lemma}\label{l:Edelta2} Let $n\in\N$ and suppose that $h\in L^2(\rho^{[n+1]})$ such
that $h(x,\cdot)\in\bH_n$ for each $x\in\BX$. Define 
\begin{equation*}
Z:=(\theta+n)\int h(z) \, \zeta^{n+1}(\dd z) 
- \iint h(x,y_1, \ldots,y_n) \, (\rho + \delta_{y_1} + \cdots + \delta_{y_n})(\dd x) \, \zeta^n(\dd(y_1, \ldots, y_n)).
\end{equation*}
Then $\BE Z^2 <\infty$ and
\begin{align*}
\BE Z^2 &= \frac{\theta n!}{\Pn{\theta}{2n+1}} \Pnm{\rho}{n+1}(h^2) 
    +  \frac{\theta nn!}{\Pn{\theta}{2n+1}} \int h(x_1, \bfm{x}_n)^2 \ \Pnm{\rho}{n}(\dd \bfm{x}_n) \\
    &\quad +\frac{n!(n^2-\theta)}{\Pn{\theta}{2n+2}} \int h(x,\bfm{x}_n) h(y, \bfm{x}_n) \, \Pnm{\rho}{n+2}(\dd (x,y,\bfm{x}_n)) \\
    &\quad +\frac{nn!(n^2-\theta)}{\Pn{\theta}{2n+2}} \int h(x,x,\bfm{x}_{n-1}) h(y, y,\bfm{x}_{n-1})  \, \Pnm{\rho}{n+1}(\dd (x,y,\bfm{x}_{n-1})) \\
    &\quad - \frac{2(n+1)(\theta+n)nn!}{\Pn{\theta}{2n+2}} \int h(x,x,\bfm{x}_{n-1}) h(y, x,\bfm{x}_{n-1})  \, \Pnm{\rho}{n+1}(\dd (x,y,\bfm{x}_{n-1})) \\
    &\quad + \frac{nn!(\theta+n)^2}{\Pn{\theta}{2n+2}}\int h(x, y,\bfm{x}_{n-1}) h(y, x, \bfm{x}_{n-1}) \, \Pnm{\rho}{n+1}(\dd (x,y,\bfm{x}_{n-1}))  \\
    &\quad + \frac{nn!(\theta+n)^2}{\Pn{\theta}{2n+2}}\int h(x_1,\bfm{x}_{n}) h(x_2,  \bfm{x}_{n}) \, \Pnm{\rho}{n}(\dd \bfm{x}_{n}).
\end{align*}
\end{lemma}
\begin{proof} 
By~\eqref{nmoment},
\begin{align*}
\BE Z^2 
&= \frac{(\theta+n)^2}{\Pn{\theta}{2n+2}} \int (h \otimes h)(\bfm{x}) \,\Pnm{\rho}{2n+2}(\dd \bfm{x})  \\
&\quad -\frac{2(\theta+n)}{\Pn{\theta}{2n+1}} \iint h(\bfm{x}_{n+1}) h(z,\bfm{y}_{n}) \, (\rho + \delta_{\bfm{y}_{n}})(\dd z)\, 
\Pnm{\rho}{2n+1}(\dd (\bfm{x}_{n+1},\bfm{y}_{n}))  \\
    &\quad  + \frac{1}{\Pn{\theta}{2n}} \iiint h(x,\bfm{x}_{n})h(y,\bfm{y}_{n})\, (\rho + \delta_{\bfm{x}_{n}})(\dd x) 
(\rho + \delta_{\bfm{y}_{n}})(\dd y)\, \Pnm{\rho}{2n}(\dd (\bfm{x}_{n},\bfm{y}_{n})) \\
    &\eqqcolon \frac{(\theta+n)^2}{\Pn{\theta}{2n+2}} A_1 - \frac{2(\theta+n)}{\Pn{\theta}{2n+1}} A_2 
+ \frac{1}{\Pn{\theta}{2n}} A_3 , 
\end{align*}
where the tensor product notation $h\otimes h$ is defined in \eqref{Def Tensorprodukt}.

Since $h\in L^2(\rho^{[n+1]})$, we obtain  $h\otimes h\in L^1( \rho^{[2n+2]})$ from the Cauchy--Schwarz 
inequality. Since $h(x,\cdot)\in\bH_n$ for each $x\in\BX$, we can apply
Lemma~\ref{LemAllgIsoFormel1Funktion} to obtain 
\begin{equation*}
A_1 =\sideset{}{^{\ne}}\sum_{i_1, \ldots, i_n\in [n+2]} \int h(x, \bfm{x}_n) h(y, x_{i_1}, \ldots, x_{i_n}) \, \Pnm{\rho}{n+2}(\dd (x,y,\bfm{x}_n)),
\end{equation*}
where $x_{n+1}\coloneqq x$, $x_{n+2}\coloneqq y$.
By the symmetry of $h(x,\cdot)$, 
\begin{align*}
    A_1 =&\;\sideset{}{^{\ne}}\sum_{i_1, \ldots, i_n \in [n]} \int h(x, \bfm{x}_n) h(y, x_{i_1}, \ldots, x_{i_n}) \, \Pnm{\rho}{n+2}(\dd (x,y,\bfm{x}_n)) \\
    &+ n\; \sideset{}{^{\ne}}\sum_{i_1, \ldots, i_{n-1} \in [n]} \int h(x, \bfm{x}_n) h(y, x,  x_{i_1}, \ldots, x_{i_{n-1}}) \, \Pnm{\rho}{n+2}(\dd (x,y,\bfm{x}_n)) \\
    &+n\; \sideset{}{^{\ne}}\sum_{i_1, \ldots, i_{n-1} \in [n]} \int h(x, \bfm{x}_n) h(y, y, x_{i_1}, \ldots, x_{i_{n-1}}) \, \Pnm{\rho}{n+2}(\dd (x,y,\bfm{x}_n)) \\
    &+n(n-1)\;\sideset{}{^{\ne}}\sum_{i_1, \ldots, i_{n-2} \in [n]} \int h(x, \bfm{x}_n) h(y, x, y, x_{i_1}, \ldots, x_{i_{n-2}}) \, \Pnm{\rho}{n+2}(\dd (x,y,\bfm{x}_n)) \\
    \eqqcolon& I_1 + I_2 + I_3 + I_4.
\end{align*}
For $n=1$, we need to interpret $I_2$ as $\int h(x,x_1)h(y,x)\,\rho^{[3]}(\dd (x,y,x_1))$, while
$I_3$ has to be interpreted as $\int h(x,x_1)h(y,y)\,\rho^{[3]}(\dd (x,y,x_1))$.

We now treat each of the terms $I_1,I_2,I_3,I_4$ separately. We have
\begin{equation}\label{I1}
    I_1 = n! \int h(x,\bfm{x}_n) h(y, \bfm{x}_n) \, \Pnm{\rho}{n+2}(\dd (x,y,\bfm{x}_n)).
\end{equation}
By the recursion \eqref{Rekursion mu[m+n]} and $h(x,\cdot)\in\bH$, $x\in\BX$,
\begin{align}\notag
    I_2 =& nn! \int h(x,\bfm{x}_n) h(y, x,\bfm{x}_{n-1}) (\rho + \delta_x + \delta_y + \delta_{\bfm{x}_{n-1}})(\dd x_n) \, 
\Pnm{\rho}{n+1}(\dd (x,y,\bfm{x}_{n-1}))  \\ \notag
    =& nn! \int h(x,x,\bfm{x}_{n-1}) h(y, x,\bfm{x}_{n-1})  \, \Pnm{\rho}{n+1}(\dd (x,y,\bfm{x}_{n-1})) \\
\label{I2}
    &+ nn! \int h(x,y,\bfm{x}_{n-1}) h(y, x,\bfm{x}_{n-1})  \, \Pnm{\rho}{n+1}(\dd (x,y,\bfm{x}_{n-1})).
\end{align}
With the same argument, we get
\begin{align}\notag
    I_3 &= nn! \int h(x,x,\bfm{x}_{n-1}) h(y, y,\bfm{x}_{n-1})  \, \Pnm{\rho}{n+1}(\dd (x,y,\bfm{x}_{n-1})) \\
\label{I3} 
   &\quad + nn! \int h(x,y,\bfm{x}_{n-1}) h(y, y,\bfm{x}_{n-1})  \, \Pnm{\rho}{n+1}(\dd (x,y,\bfm{x}_{n-1})).
\end{align}
By symmetry, the multiple sum in $I_4$ equals
\begin{equation*}
\frac{n!}{2}\int h(x, \bfm{x}_n) h(y, x, y, x_1, \ldots, x_{n-2}) \, \Pnm{\rho}{n+2}(\dd (x,y,\bfm{x}_n)).
\end{equation*}
Therefore, we obtain from the recursion \eqref{Rekursion mu[m+n]} and 
the properties of $h$ that
\begin{align*}
I_4    &= \frac{n(n-1)n!}{2} \int h(x,x,y, \bfm{x}_{n-2}) h(y, x, y, \bfm{x}_{n-2})  \, \Pnm{\rho}{n}(\dd (x,y,\bfm{x}_{n-2})) \\
    &\quad + \frac{n(n-1)n!}{2} \int h(x,y,x, \bfm{x}_{n-2}) h(y, x, y, \bfm{x}_{n-2})  \, \Pnm{\rho}{n}(\dd (x,y,\bfm{x}_{n-2})) \\
    &= n(n-1)n! \int h(x,x,y, \bfm{x}_{n-2}) h(y, x, y, \bfm{x}_{n-2})  \, \Pnm{\rho}{n}(\dd (x,y,\bfm{x}_{n-2}))
\end{align*}
or, equivalently,
\begin{equation}\label{I4}
I_4= n(n-1)n! \int h(x_1, \bfm{x}_{n}) h(x_2, \bfm{x}_{n})  \, \Pnm{\rho}{n}(\dd\bfm{x}_{n}).
\end{equation}

Next we consider 
\begin{align*}
    A_2 = &\iint h(x,\bfm{x}_{n}) h(z,\bfm{y}_{n}) \, (\rho + \delta_x + \delta_{\bfm{x}_{n}} + \delta_{\bfm{y}_{n}})(\dd z)\, 
\Pnm{\rho}{2n+1}(\dd (x,\bfm{x}_{n},\bfm{y}_{n})) \\
&-\iint h(x,\bfm{x}_{n}) h(z,\bfm{y}_{n}) \, (\delta_x + \delta_{\bfm{x}_{n}} )(\dd z)\, \Pnm{\rho}{2n+1}(\dd (x,\bfm{x}_{n},\bfm{y}_{n}))
    \eqqcolon A_1 - J,
\end{align*}
noting that the above integrals are finite when replacing $h$ by $|h|$. 
Applying Lemma~\ref{LemAllgIsoFormel1Funktion} once again (and letting $x_{n+1} \coloneqq x$), we obtain
\begin{align*}
J &=\; \sideset{}{^{\ne}}\sum_{i_1, \ldots, i_n \in [n+1]} \iint h(x, \bfm{x}_n) h(z, x_{i_1}, \ldots, x_{i_n}) \, (\delta_x + \delta_{\bfm{x}_{n}} )(\dd z)\,
\Pnm{\rho}{n+1}(\dd (x,\bfm{x}_n)) \\
& =\; \sideset{}{^{\ne}}\sum_{i_1, \ldots, i_n \in [n+1]} \int h(x, \bfm{x}_n) h(x, x_{i_1}, \ldots, x_{i_n}) \, \Pnm{\rho}{n+1}(\dd (x,\bfm{x}_n)) \\
&\quad +\; \sideset{}{^{\ne}}\sum_{i_1, \ldots, i_n \in [n+1]} \sum_{i=1}^{n} \int h(x, \bfm{x}_n) h(x_i, x_{i_1}, \ldots, x_{i_n}) \, \Pnm{\rho}{n+1}(\dd (x,\bfm{x}_n)) 
    \eqqcolon J_1 + J_2.
\end{align*}
By symmetry,
\begin{align*}
    J_1 =&\; \sideset{}{^{\ne}}\sum_{i_1, \ldots, i_n \in [n]} \int h(x, \bfm{x}_n) h(x, x_{i_1}, \ldots, x_{i_n}) \, \Pnm{\rho}{n+1}(\dd (x,\bfm{x}_n)) \\
    &+ n\;  \sideset{}{^{\ne}}\sum_{i_1, \ldots, i_{n-1} \in [n]} \int h(x, \bfm{x}_n) h(x,x, x_{i_1}, \ldots, x_{i_{n-1}}) \, \Pnm{\rho}{n+1}(\dd (x,\bfm{x}_n))
    \eqqcolon J_{11} + J_{12},
\end{align*}
where $J_{12}:=n\int h(x,x_1)h(x,x)\,\rho^{[2]}(\dd (x,x_1))$ for $n=1$. We have
\begin{equation*}
    J_{11} = n! \int h(x,\bfm{x}_n)^2 \, \Pnm{\rho}{n+1}(\dd (x, \bfm{x}_{n})) 
    = n! \Pnm{\rho}{n+1}(h^2) 
\end{equation*}
and
\begin{align*}
    J_{12} &= n  n! \iint h(x, \bfm{x}_n) h(x,x, \bfm{x}_{n-1}) \, (\rho + \delta_x + \delta_{\bfm{x}_{n-1}})(\dd x_n) \, 
\Pnm{\rho}{n}(\dd (x,\bfm{x}_{n-1})) \\
    &= n  n! \int h(x,x, \bfm{x}_{n-1}) h(x,x, \bfm{x}_{n-1})  \, \Pnm{\rho}{n}(\dd (x,\bfm{x}_{n-1})) \\
    &= nn! \int h(x_1, \bfm{x}_n)^2 \ \Pnm{\rho}{n}(\dd \bfm{x}_n).
\end{align*}
The term $J_2$ is given by
\begin{align*}
   J_2 &= n \sideset{}{^{\ne}}\sum_{i_1, \ldots, i_n\in [n+1]} \int h(x, \bfm{x}_n) h(x_1, x_{i_1}, \ldots, x_{i_n}) \, \Pnm{\rho}{n+1}(\dd (x,\bfm{x}_n)) \\
   &= n \sideset{}{^{\ne}}\sum_{i_1, \ldots, i_n\in [n]} \int h(x, \bfm{x}_n) h(x_1, x_{i_1}, \ldots, x_{i_n}) \, \Pnm{\rho}{n+1}(\dd (x,\bfm{x}_n)) \\
   &\quad+ n^2 \sideset{}{^{\ne}}\sum_{i_1, \ldots, i_{n-1}\in [n]} \int h(x, \bfm{x}_n) h(x_1, x, x_{i_1}, \ldots, x_{i_{n-1}}) \, \Pnm{\rho}{n+1}(\dd (x,\bfm{x}_n))
   \eqqcolon J_{21} + J_{22},
\end{align*}
where $J_{22}:=n^2\int h(x,x_1)h(x_1,x)\,\rho^{[2]}(\dd (x,x_1))$ for $n=1$.
We have
\begin{equation*}
    J_{21} = nn! \int h(x, \bfm{x}_n) h(x_1,  \bfm{x}_n) \, \Pnm{\rho}{n+1}(\dd (x,\bfm{x}_n))
\end{equation*}
To treat $J_{22}$,  we write $J_{22} = J_{221} + J_{222}$, distinguishing the cases 
$1\notin\{i_1, \ldots, i_{n-1}\}$ and $1\in\{i_1, \ldots, i_{n-1}\}$,
where $J_{222}:=0$ for $n=1$. We have
\begin{align*}
    J_{221} &= n^2\; \sideset{}{^{\ne}}\sum_{i_1, \ldots, i_{n-1}\in ([n]\setminus\{1\})} 
\int h(x, \bfm{x}_n) h(x_1, x, x_{i_1}, \ldots, x_{i_{n-1}}) \, \Pnm{\rho}{n+1}(\dd (x,\bfm{x}_n)) \\
    &= n^2 (n-1)! \int h(x, \bfm{x}_n) h(x_1, x, x_{2}, \ldots, x_{n}) \, \Pnm{\rho}{n+1}(\dd (x,\bfm{x}_n)) \\
    &= nn! \int h(x, y,\bfm{x}_{n-1}) h(y, x, \bfm{x}_{n-1}) \, \Pnm{\rho}{n+1}(\dd (x,y,\bfm{x}_{n-1}))
\end{align*}
and, for $n\ge 2$,
\begin{align*}
    &J_{222} = n^2(n-1)\; \sideset{}{^{\ne}}\sum_{i_1, \ldots, i_{n-2} \in ([n]\setminus\{1\})} 
\int h(x, \bfm{x}_n) h(x_1, x, x_1, x_{i_1}, \ldots, x_{i_{n-2}}) \, \Pnm{\rho}{n+1}(\dd (x,\bfm{x}_n)) \\
    &= n(n-1) n! \iint h(x, \bfm{x}_n) h(x_1, x, x_{1}, \ldots, x_{n-1}) \, 
(\rho + \delta_x + \delta_{\bfm{x}_{n-1}})(\dd x_n) \,  \Pnm{\rho}{n}(\dd (x,\bfm{x}_{n-1})) \\
    &= n(n-1)n! \int h(x, x,\bfm{x}_{n-1}) h(x_1, x, \bfm{x}_{n-1}) \, \Pnm{\rho}{n}(\dd (x,\bfm{x}_{n-1})) \\
    &= n(n-1)n! \int h(x_1,\bfm{x}_{n}) h(x_2,  \bfm{x}_{n}) \, \Pnm{\rho}{n}(\dd \bfm{x}_{n}).
\end{align*}

Let us now consider
\begin{align*}
    A_3 &= \iint  h(x,\bfm{x}_{n}) h(y,\bfm{y}_{n}) \, (\rho + \delta_{\bfm{x}_{n}})(\dd x)  \, (\rho + \delta_{\bfm{x}_{n}} + \delta_{\bfm{y}_{n}})(\dd y) \, 
\Pnm{\rho}{2n}(\dd (\bfm{x}_{n},\bfm{y}_{n})) \\
&\quad - \iint  h(x,\bfm{x}_{n}) h(y,\bfm{y}_{n}) \, (\rho + \delta_{\bfm{x}_{n}})(\dd x)  \,  \delta_{\bfm{x}_{n}}(\dd y) \, \Pnm{\rho}{2n}(\dd (\bfm{x}_{n},\bfm{y}_{n})) \\
 &\eqqcolon (A_1 - J) - K.
\end{align*}
We have
\begin{align*}
    K &=\; \sideset{}{^{\ne}}\sum_{i_1, \ldots, i_{n} \in [n]} \iiint  h(x,\bfm{x}_{n}) h(y,x_{i_1}, \ldots, x_{i_n}) \, (\rho + \delta_{\bfm{x}_{n}})(\dd x)  
\,  \delta_{\bfm{x}_{n}}(\dd y) \, \Pnm{\rho}{n}(\dd \bfm{x}_{n}) \\
    &= n! \iint  h(x,\bfm{x}_{n}) h(y,\bfm{x}_{n}) \,  \delta_{\bfm{x}_{n}}(\dd y) \, \Pnm{\rho}{n+1}(\dd (x,\bfm{x}_{n})) \\
    &= nn! \int  h(x,\bfm{x}_{n}) h(x_1,\bfm{x}_{n})  \, \Pnm{\rho}{n+1}(\dd (x,\bfm{x}_{n}))\\
&= nn! \int  h(x,x,\bfm{x}_{n-1}) h(y,x,\bfm{x}_{n-1})  \, \Pnm{\rho}{n+1}(\dd (x,y,\bfm{x}_{n-1})).
\end{align*}

To finish the proof, we note that
\begin{align*}
 \BE Z^2 &= \frac{(\theta+n)^2}{\Pn{\theta}{2n+2}} A_1 - \frac{2(\theta+n)}{\Pn{\theta}{2n+1}} A_2 
+ \frac{1}{\Pn{\theta}{2n}} A_3\\ 
&= \frac{(\theta+n)^2}{\Pn{\theta}{2n+2}} A_1 - \frac{2(\theta+n)}{\Pn{\theta}{2n+1}} (A_1 -J) 
     + \frac{1}{\Pn{\theta}{2n}} (A_1 - J - K).
\end{align*}
Inserting here the preceding formulas for $A_1$, $J$ and $K$, we obtain the assertion  with the help
of some algebra.
\end{proof}

\begin{lemma}\label{l:4.5} Let the assumptions of Lemma \ref{l:Edelta2}
be satisfied. Then there exists $C>0$ such that
\begin{equation*}
    \BE Z^2\le C \frac{(n+1)!(\theta+n)^2}{\Pn{\theta}{2n+2}}\int h(z)^2 \, \Pnm{\rho}{n+1}(\dd z).
\end{equation*}
\end{lemma}
\begin{proof}
We need to bound the terms occurring in Lemma \ref{l:Edelta2}. By symmetry of $\rho^{[n]}$,
\begin{align*}
\int h(x_1, \bfm{x}_n)^2 \ \Pnm{\rho}{n}(\dd \bfm{x}_n)
&=\frac{1}{n} \sum ^n_{i=1}\int h(x_i, \bfm{x}_n)^2 \, \Pnm{\rho}{n}(\dd \bfm{x}_n) \le \frac{1}{n}\rho^{[n+1]}(h^2), 
\end{align*}
where we have used the recursion \eqref{Rekursion mu[m+n]} to obtain the inequality.
Multiplying this inequality  with 
$\frac{ nn!}{\Pn{\theta}{2n+1}}=\frac{nn!(\theta+2n+1)}{\Pn{\theta}{2n+2}}$,
we see that the required inequality holds.
Next we bound
\begin{align*}
2\int &|h(x,\bfm{x}_n) h(y, \bfm{x}_n)| \, \Pnm{\rho}{n+2}(\dd (x,y,\bfm{x}_n)) \\
&\le \int h(x,\bfm{x}_n)^2\, \Pnm{\rho}{n+2}(\dd (x,y,\bfm{x}_n))+ \int h(y,\bfm{x}_n)^2\, \Pnm{\rho}{n+2}(\dd (x,y,\bfm{x}_n))\\
      &=2(\theta+n+1)\rho^{[n+1]}(h^2).
\end{align*}
Multiplying this inequality  with $\frac{n!|n^2-\theta|}{\Pn{\theta}{2n+2}}$ yields the requested inequality.
Next we bound
\begin{align*}
2\int & |h(x,x,\bfm{x}_{n-1}) h(y, y,\bfm{x}_{n-1})|  \, \Pnm{\rho}{n+1}(\dd (x,y,\bfm{x}_{n-1}))\\
&\le 2(\theta+n) \int h(x,x,\bfm{x}_{n-1})^2\,\rho^{[n]}(\dd (x,\bfm{x}_{n-1}))\\
&=2(\theta+n) \int h(x_1,\bfm{x}_{n})^2\,\rho^{[n]}(\dd \bfm{x}_n)
\le \frac{2(\theta+n)}{n}\rho^{[n+1]}(h^2), 
\end{align*}
where the final inequality has been derived above. Multiplying this with 
$\frac{nn!|n^2-\theta|}{\Pn{\theta}{2n+2}}$ yields the requested bound.
Next we bound
\begin{align*}
2 \int& |h(x,x,\bfm{x}_{n-1}) h(y, x,\bfm{x}_{n-1})|  \, \Pnm{\rho}{n+1}(\dd (x,y,\bfm{x}_{n-1}))\\
&\le (\theta+n)\int h(x,x,\bfm{x}_{n-1})^2\,\rho^{[n]}(\dd (x,\bfm{x}_{n-1}))+\rho^{[n+1]}(h^2).
\end{align*}
Here the first term as well as  
$\int |h(x,y,\bfm{x}_{n-1}) h(y, x,\bfm{x}_{n-1})|  \, \Pnm{\rho}{n+1}(\dd (x,y,\bfm{x}_{n-1}))$
can be treated as above. Finally we can bound
\begin{align*}
2\int |h(x_1,\bfm{x}_{n}) h(x_2,  \bfm{x}_{n})| \, \Pnm{\rho}{n}(\dd \bfm{x}_{n})
&\le \int h(x_1,\bfm{x}_{n})^2 \, \Pnm{\rho}{n}(\dd \bfm{x}_{n}) 
+\int h(x_2,  \bfm{x}_{n})^2 \, \Pnm{\rho}{n}(\dd \bfm{x}_{n})\\
&\le \frac{1}{n} \Pnm{\rho}{n+1}(h^2),
\end{align*}
finishing the proof.
\end{proof}

We also need the following lemma.

\begin{lemma}\label{l:4.1}
Let $h_n\colon \BX^{n+1} \to \R$, $n\in\N_0$, be measurable functions such that $h_n(x, \cdot) \in \bH_n$ for each $x\in\BX$
and $n\in\N$. Assume further that
\begin{equation}\label{e:4.11}
\sum_{n=0}^\infty \frac{n!}{\Pn{\theta}{2n+1}} \Pnm{\rho}{n+1} (h_{n}^2) < \infty.
\end{equation}
Then the series of functions
\begin{equation}\label{e:33}
(\mu,x)\mapsto h_0(x) + \sum_{n=1}^\infty \int h_n(x,y) \,\mu^n(\dd y)
\end{equation}
converges in $L^2(C'_\zeta)$.
\end{lemma}
\begin{proof}
For $n\in\N$, we define $H_n\colon \Omega \times \BX \to \R$ by 
\begin{equation*}
    H_n(\omega,x) \coloneqq \int h_n(x,\bfm{y}_n)\, \zeta^n(\omega,\dd\bfm{y}_n).
\end{equation*}
Let $m,n\in\N$ with $m\ge n$. 
By~\eqref{nmoment} and Corollary~\ref{c:A3}, we have
\begin{align*}
    &\int H_m(\omega,x) H_n(\omega,x) \, C_\zeta(\dd (\omega,x))\\ 
    &= \I\{m=n\} m! \left( \int  h_m(\bfm{x}_{m+1})^2\, \Pnm{\rho}{m+1}(\dd \bfm{x}_{m+1}) 
    +   \sum_{r=1}^m \int h_m(x_r,\bfm{x}_{m})^2 \, \Pnm{\rho}{m}(\dd \bfm{x}_{m}) \right) \\
    &\hphantom{=} + \I\{m=n+1\} (n+1)! \int h_{n+1}(x_1,\bfm{x}_{n+1}) h_{n}(\bfm{x}_{n+1}) \, \Pnm{\rho}{n+1}(\dd \bfm{x}_{n+1}).
\end{align*}
Hence, we obtain  for $n_0,n_1\in\N$ with $n_0 \le n_1$
\begin{align*}
    \BE&\int\bigg( \sum_{n=n_0}^{n_1} \int h_n(x,\bfm{y}) \, \zeta^n(\dd \bfm{y}) \bigg)^2 \zeta(\dd x)
    = \sum_{m,n=n_0}^{n_1} \BE\int H_m(x)H_n(x) \,\zeta(\dd x) \\
    &= \sum_{n=n_0}^{n_1} \frac{n!}{\Pn{\theta}{2n+1}}  \bigg( \int  h_n(\bfm{x}_{n+1})^2 \, \Pnm{\rho}{n+1}(\dd \bfm{x}_{n+1}) 
    +   \sum_{r=1}^n \int h_n(x_r,\bfm{x}_{n})^2 \, \Pnm{\rho}{n}(\dd \bfm{x}_{n}) \bigg) \\
    &\quad + \sum_{n=n_0}^{n_1-1} \frac{(n+1)!}{\Pn{\theta}{2n+2}} \int h_{n+1}(x_1,\bfm{x}_{n+1}) h_{n}(\bfm{x}_{n+1}) \, \Pnm{\rho}{n+1}(\dd \bfm{x}_{n+1})
    \eqqcolon S_1 + S_2.
\end{align*}
We treat the sums $S_1$ and $S_2$ individually and establish upper bounds for both. 

Let $n\in\N$. Then
\begin{equation*}
    \sum_{r=1}^n \int h_n(x_r,\bfm{x}_{n})^2 \,\Pnm{\rho}{n}(\dd \bfm{x}_{n})
    = \iint  h_n(y,\bfm{x}_{n})^2 \, \delta_{\bfm{x}_{n}} (\dd y) \, \Pnm{\rho}{n}(\dd \bfm{x}_{n}) 
    \le 
    \Pnm{\rho}{n+1}(h_n^2), 
\end{equation*}
where we used the symmetry of $h_n$ in its last $n$ arguments. Therefore,
\begin{equation*}
S_1 \le 2 \sum_{n=n_0}^{n_1} \frac{n!}{\Pn{\theta}{2n+1}}   \Pnm{\rho}{n+1}(h_n^2).
\end{equation*}
To bound $S_2$, we can use the inequalities
\begin{align*}
\int &h_{n+1}(x_1,\bfm{x}_{n+1}) h_{n}(\bfm{x}_{n+1}) \,\Pnm{\rho}{n+1}(\dd \bfm{x}_{n+1}) \\
    &\le \frac{1}{2} \int h_{n+1}(x_1,\bfm{x}_{n+1})^2 \,\Pnm{\rho}{n+1}(\dd \bfm{x}_{n+1})
+\frac{1}{2}\int h_{n}(\bfm{x}_{n+1})^2 \,\Pnm{\rho}{n+1}(\dd \bfm{x}_{n+1}) 
\end{align*}
and
\begin{equation*}
(n+1) \int h_{n+1}(x_1,\bfm{x}_{n+1})^2 \,\Pnm{\rho}{n+1}(\dd \bfm{x}_{n+1}) 
\le \int h_{n+1}(\bfm{x}_{n+2})^2 \,\Pnm{\rho}{n+2}(\dd \bfm{x}_{n+2}),
\end{equation*}
where we used the symmetry of $h_{n+1}$ in its last $n+1$ arguments. Hence, we obtain
\begin{align*}
    S_2 &\le \frac{1}{2}\sum_{n=n_0}^{n_1-1} \frac{n!}{\Pn{\theta}{2n+2}}\big(\Pnm{\rho}{n+2}(h_{n+1}^2) 
+ (n+1)\Pnm{\rho}{n+1}(h_n^2) \big) \\
    &= \frac{1}{2}\sum_{n=n_0+1}^{n_1} \frac{(n-1)!}{\Pn{\theta}{2n}} \Pnm{\rho}{n+1}(h_{n}^2) 
+ \frac{1}{2}\sum_{n=n_0}^{n_1-1} \frac{(n+1)!}{\Pn{\theta}{2n+2}}\Pnm{\rho}{n+1}(h_n^2). 
\end{align*}
Combining our findings, we arrive at 
\begin{equation*}
    \BE \int \left( \sum_{n=n_0}^{n_1} \int h_n(x,\bfm{y}) \,\zeta^n(\dd \bfm{y}) \right)^2 \, \zeta(\dd x) 
    \le \sum_{n=n_0}^{n_1} \frac{n!(4+\frac{1}{2}\theta)}{\Pn{\theta}{2n+1}} \Pnm{\rho}{n+1}(h_n^2). \qedhere
\end{equation*} 
\end{proof}

The proof of the Lemma \ref{l:4.1} yields the following result.

\begin{corollary} Let the assumptions of Lemma \ref{l:4.1} be satisfied
and define the function $H\in L^2(C_\zeta)$ by the right-hand side of \eqref{e:33}.
Then
\begin{align*}
\BE \int H(\zeta,x)^2\,\zeta(\dd x)&=\sum^\infty_{n=0}\frac{n!}{\theta^{(2n+1)}}\rho^{[n+1]}(h_n^2)
+\sum^\infty_{n=0}\frac{nn!}{\theta^{(2n+1)}}\int h_n(x_1,\bfm{x}_n)^2\,\rho^{[n]}(\dd \bfm{x}_n)\\
& \quad+\sum^\infty_{n=0}\frac{(n+1)!}{\theta^{(2n+2)}}\int h_{n+1}(x_1,\bfm{x}_{n+1})h_{n}(\bfm{x}_{n+1})
\,\rho^{[n+1]}(\dd \bfm{x}_{n+1}).
 \end{align*}
\end{corollary}

\begin{remark}\rm 
Let $F\in L^2(\zeta)$ with kernel functions $f_n\in\bH_n$, $n\in\N$. Let $g_n\colon\BX\to\R$, $n\in\N_0$, be
measurable functions such that there exists $c\ge 0$ with $|g_n|\le c$ for all $n\in\N_0$.
Then $h_n(x,x_1\ldots,x_n):=g_n(x)f_n(x_1,\ldots,x_n)$, $n\in\N$, and $h_0(x):=g_0(x)$, satisfy the assumptions of 
Lemma \ref{l:4.1}. Indeed, we clearly have $h_n(x,\cdot)\in\bH_n$ for each $n\in\N$ and each $x\in\BX$.
Furthermore, the series \eqref{e:4.11} can be written as
\begin{equation*}
\sum_{n=0}^\infty \frac{n!}{\Pn{\theta}{2n+1}}\iint g_n(x)^2f_n(\bfm{x})^2
\,(\rho+\delta_{\bfm{x}})(\dd x)\,\Pnm{\rho}{n}(\dd \bfm{x})\\
\end{equation*}
and can be bounded by
\begin{equation*}
c^2\sum_{n=0}^\infty \frac{n!(\theta+n)}{\Pn{\theta}{2n+1}}
\int f_n(\bfm{x})^2 \,\Pnm{\rho}{n}(\dd \bfm{x}).
\end{equation*}
By \eqref {e:iso0}, this is finite. Therefore, \eqref{e:4.11} holds.

If $g_n\equiv g$ does not depend on $n\in\N_0$, then the function $H$ 
on the right-hand side of \eqref{e:33} can be chosen as $H(\mu,x)=g(x)F(\mu)$.
\end{remark}

\begin{theorem}\label{t:delta} Suppose that $h_n\colon\BX^{n+1}\to\R$, $n\in\N_0$,
are measurable functions such that $h_n(x,\cdot)\in \mathbf{H}_n$ for each $x\in\BX$
and $n\in\N$
and
\begin{align}\label{e:34}
    \sum_{n=0}^\infty \frac{(\theta + n)^2(n+1)!}{\Pn{\theta}{2n+2}}\int h_n(z)^2 \, \Pnm{\rho}{n+1}(\dd z) < \infty.
\end{align}
Define $H\in L^2(C'_\zeta)$ by the $L^2(C'_\zeta)$-limit \eqref{e:33}.
Then $H\in \operatorname{dom}(\delta)$ and
\begin{align}\label{delta'}
    \delta(H)&= \sum_{n=0}^\infty  \bigg((\theta+n)\int h_n(z) \, \zeta^{n+1}(\dd z)  \\ \notag
 &\quad- \iint h_n(x,y_1, \ldots,y_n) \, (\rho + \delta_{y_1} + \cdots + \delta_{y_n})(\dd x) \, \zeta^n(\dd(y_1, \ldots, y_n))\bigg),
\quad \BP\text{-a.s.},
\end{align}
where the series converges in $L^2(\BP)$.
\end{theorem}

\begin{proof} We first show that the series \eqref{delta'} converges in $L^2(\BP)$.
Given $n\in\N$, we define
\begin{align*}
    X_n \coloneqq \int h_n(\bfm{z}) \,\zeta^{n+1}(\dd \bfm{z}) &&\text{and}&& 
Y_n \coloneqq \frac{1}{\theta+n}\iint h_n(x,\bfm{y}_{n}) \, (\rho + \delta_{\bfm{y}_{n}})(\dd x) \,\zeta^n(\dd\bfm{y}_{n}), 
\end{align*}
as well as $ Z_n \coloneqq (\theta+n) (X_n - Y_n)$.
Let $m,n\in\N$. 
By Jensen's inequality and~\eqref{nmoment},
\[
\BE X_n^2
\leq \frac{1}{\Pn{\theta}{n+1}}
\int h_n(\bfm{z})^2\,\Pnm{\rho}{n+1}(\dd\bfm{z}),
\]
and, similarly,
\[
\BE Y_n^2
\leq \frac{1}{\Pn{\theta}{n+1}}
\int h_n(\bfm{z})^2\,\Pnm{\rho}{n+1}(\dd\bfm{z}).
\]
Hence, $X_n,Y_n\in L^2(\BP)$. 
Let $m\in\N$ and $g\in\bH_m$. By~\eqref{nmoment}, we have
\begin{equation*}
    \BE X_n \zeta^m(g)
    = \frac{1}{\Pn{\theta}{m+n+1}} \int h_n(\bfm{z}) g(\bfm{x})\, \Pnm{\rho}{m+n+1} (\dd (\bfm{x},\bfm{z}))
\end{equation*}
and 
\begin{equation*}
    \BE Y_n\zeta^m(g) 
    = \frac{1}{\Pn{\theta}{m+n}} \frac{1}{\theta + n} \iint g(\bfm{x}_{m}) h_n(t, \bfm{z}_n) \, (\rho + \delta_{\bfm{z}_{n}})(\dd t) \, 
\Pnm{\rho}{m+n} (\dd (\bfm{x}_{m},\bfm{z}_n)).
\end{equation*}
By Corollaries~\ref{c:A3} and \ref{c:A6}, these expectations vanish unless $m \in \{n+1,n,n-1\}$ and $m \in \{n,n-1\}$, respectively.
Thus, we conclude $X_n \in \mathbf{F}_{n-1} \oplus \mathbf{F}_n \oplus \mathbf{F}_{n+1}$ and $Y_n \in \mathbf{F}_{n-1} \oplus \mathbf{F}_n$, 
which implies $Z_n \in \mathbf{F}_{n-1} \oplus \mathbf{F}_n \oplus \mathbf{F}_{n+1}$.
Let $n_0,n_1\in\N$ with $n_0 \le n_1$. 
The established orthogonality yields
\begin{align*}
    \BE \left(\sum_{n=n_0}^{n_1} Z_n\right)^2
    &= \sum_{m,n= n_0}^{n_1}  \I{\{|m-n|< 3\}} \BE Z_m Z_n \\
    &= \sum_{n= n_0}^{n_1}   \BE Z_n^2 
    + 2 \sum_{n= n_0}^{n_1-1}   \BE Z_n Z_{n+1} 
    +2 \sum_{n= n_0}^{n_1-2}   \BE Z_n Z_{n+2}. 
\end{align*}
Since $2\,\BE |Z_n Z_{m}|\le \BE Z_n^2 + \BE Z_m^2$, $m,n\in\N$, it therefore suffices to show
$\sum_{n=0}^{\infty}   \BE Z_n^2<\infty$.
This follows from Lemmas \ref{l:Edelta2} and \ref{l:4.5}.

Let
\begin{equation*}
    \delta'(H)\coloneqq \sum_{n=0}^\infty  \bigg((\theta+n)\int h_n(z) \, \zeta^{n+1}(\dd z)  \\
    - \iint h_n(x,\bfm{y}_n) \, (\rho + \delta_{\bfm{y}_n})(\dd x) \, \zeta^n(\dd\bfm{y}_n))\bigg).
\end{equation*}
We show that $\delta'(H)$ satisfies partial integration, thus proving that it coincides $\BP$-a.s.\ with $\delta(H)$. 
Let $F\in\operatorname{dom}(\nabla)$ with kernel functions $f_n$, $n\in\N$, and $f_0 \coloneqq \BE F$.

For $m_0\in\N$,
set
\[
    F_{\le m_0}\coloneqq \sum_{n=0}^{m_0}\zeta^n(f_n),
    \qquad
    H_{\le m_0}(\mu,x)
    \coloneqq \sum_{m=0}^{m_0}\int h_m(x,\bfm{y}_m)\,\mu^m(\dd\bfm{y}_m),
\]
and let $\delta'_{m_0}(H)$ denote the corresponding truncation of
$\delta'(H)$.
We claim that
\begin{equation}\label{eq:finite-duality}
    \BE\delta'_{m_0}(H)F_{\le m_0}
    =
    \BE\int H_{\le m_0}(\zeta,x)\nabla_xF_{\le m_0}\,\zeta(\dd x).
\end{equation}
Indeed, after expanding both sides by~\eqref{nmoment}, we obtain double sums over
$m$ and $n$. Corollaries~\ref{c:A3}, \ref{c:A4}, and \ref{c:A6} imply
that the terms involving $h_m$ and $f_n$ vanish unless $n\in \{m-1,m,m+1\}$.
Grouping the remaining terms according to these three cases yields identical expressions on both sides of \eqref{eq:finite-duality}.
(More details of this calculation can be found in the first two arXiv-versions of this paper.)
As $m_0\to\infty$, the left-hand side of \eqref{eq:finite-duality}
converges to $\BE \delta'(H)F$ since
$\delta'_{\le m_0}(H)\to\delta'(H)$ and $F_{\le m_0}\to F$ in $L^2(\BP)$.
Moreover, Lemma~\ref{l:4.1} and Proposition~\ref{p:convnabla} imply that
the right-hand side tends to $\BE \int H_x\nabla_xF\,\zeta(\dd x)$ as $m_0\to\infty$.
This concludes the proof.
\end{proof}

\begin{corollary}\label{c:4.9} Let $n\in\N_0$ and $h\in L^2(\rho^{[n+1]})$. For $n\ge 1$ assume that
$h(x,\cdot) \in \bH_n$ for each $x\in\BX$. Define $H\colon\bM_1\times \BX\to \R$ by 
$H(\mu,x):=\int h(x,x_1,\ldots,x_n)\,\mu^n(\dd (x_1,\ldots,x_n))$. Then $H\in\operatorname{dom}(\delta)$
and
\begin{align}\notag
\delta(H)&=(\theta+n)\int h(z) \, \zeta^{n+1}(\dd z)\\
\label{e:2078} 
&\quad - \iint h(x,y_1, \ldots,y_n) \, (\rho + \delta_{y_1} + \cdots + \delta_{y_n})(\dd x) \, \zeta^n(\dd(y_1, \ldots, y_n)).
\end{align}
\end{corollary}

A function $h\in L^2(\rho)$ can be identified with a function
in $L^2(C'_\zeta)$ which does not depend on the first argument.
In this case Corollary \ref{c:4.9} simplifies as follows. 

\begin{corollary}\label{r:delta(h)}\rm Suppose that $h\in L^2(\rho)$. 
Then $h\in \operatorname{dom}(\delta)$ and
\begin{equation}
\delta(h)=\theta \int h(x)\,\zeta(\dd x)-\int h(x)\,\rho(\dd x).
\end{equation}
\end{corollary}

\begin{remark}\rm Let the assumptions of Theorem \ref{t:delta} be satisfied and define
$H_n(\mu,x):=\int h_n(x,\bfm{x})\,\mu^n(\dd \bfm{x})$, $n\in\N_0$. By Lemma \ref{l:4.1}, the series
$H(\mu,x):=\sum^\infty_{n=0}H_n(\mu,x)$  converges in $L^2(C'_\zeta)$. By Theorem \ref{t:delta}
and Corollary \ref{c:4.9}, we have $H\in\operatorname{dom}(\delta)$ and
\begin{equation}\label{e:014}
\delta(H)=\sum^\infty_{n=0}\delta(H_n),\quad \BP\text{-a.s.},
\end{equation}
where the series converges in $L^2(\BP)$. The proof of Theorem
\ref{t:delta} and Lemma \ref{l:Edelta2} show that these two statements
hold under more general (but more technical) assumptions on the
functions $h_n$.  It suffices to have convergence of the series,
arising by summing the terms constituting $\BE \delta(H_n)^2$
($=\BE Z^2$) in Lemma \ref{l:Edelta2}. Given the sequence
$(h_n)_{n\ge 0}$ it seems difficult to find simple necessary and
sufficient conditions for $H\in\operatorname{dom}(\delta)$. In
contrast to the classical Gaussian and Poisson cases, the summands in
\eqref{e:014} are not orthogonal.
\end{remark}



Our next lemma provides some information on the chaos expansion
of $H\in\operatorname{dom}(\delta)$, provided that $H$ is given by 
\eqref{e:33}  and \eqref{e:4.11} holds. Note that \eqref{e:4.11}  is significantly
weaker than \eqref{e:34} assumed in  Theorem \ref{t:delta}. 

\begin{lemma}\label{l:4.12} Suppose that the functions $h_n$, $n\in\N_0$, are given as
in Lemma \ref{l:4.1}. Assume that \eqref{e:4.11} holds
and let $H$ be given by  \eqref{e:33}.
Assume that $H\in\operatorname{dom}(\delta)$ and let $F\in\bF_n$ for some
$n\in\N$. Then
\begin{align*}
\BE \delta(H)F=\BE (\delta(H_{n-1})+\delta(H_n)+\delta(H_{n+1}))F,
\end{align*}
where $H_i(\mu,x):=\int h_i(\mu,x)\,\mu^n(\dd x)$ for $i\in\N$ and 
$\delta(H_i)$ are given by \eqref{e:2078}. 
\end{lemma}
\begin{proof} Define $H_{\le k}:=H_0+\cdots+H_k$ for
$k\in\N_0$. By partial integration and Lemma \ref{l:4.1},
\begin{align*}
\BE \delta(H)F=\BE \int H(\zeta,x)\nabla_xF\,\zeta(\dd x)
=\lim_{k\to\infty}\BE\int H_{\le k}(\zeta,x) \nabla_xF\,\zeta(\dd x).
\end{align*} 
By Corollary \ref{c:4.9}, we have $H_{\le k}\in\operatorname{dom}(\delta)$ and hence,
by partial integration,
\begin{equation*}
\BE \delta(H)F=\lim_{k\to\infty}\BE\delta(H_{\le k}) F.
\end{equation*}
We have seen in the proof of Theorem \ref{t:delta} that
$\delta(H_{m})\in \mathbf{F}_{m-1} \oplus \mathbf{F}_m \oplus \mathbf{F}_{m+1}$
for each $m\in\N$. Since $F\in\bF_n$, we therefore obtain
\begin{equation*}
\BE\delta(H_{\le k}) F=\BE\delta(H_{\le n+1}) F,\quad k\ge n+2,
\end{equation*}
and hence
\begin{equation*}
\BE \delta(H)F=\BE\delta(H_{\le n+1}) F.
\end{equation*}
Since $\BE\delta(H_{m}) F=0$ for $m\le n-2$, the assertion follows.
\end{proof}

\begin{remark}\rm Assume that $H$ satisfies the assumptions of Lemma \ref{l:4.12}.
For $n\in\N_0$ let $\pi_n$ denote the orthogonal projection from $L^2(\BP)$ to
$\bF_n$.  Lemma \ref{l:4.12} implies
\begin{equation*}
\pi_n\delta(H)=\pi_n\delta(H_{n-1})+\pi_n\delta(H_n)+\pi_n\delta(H_{n+1}).
\end{equation*}
In principle, these projections can be computed with Lemma \ref{l:3.5}, thus
providing the chaos expansion of $\delta(H)$.
We omit further details.
\end{remark}

\subsection{The Fleming--Viot operator}

Let $\operatorname{dom}(L)$ stand for the set of all $F\in L^2(\zeta)$ such that 
\begin{equation}\label{domL}
   \sum_{n=1}^\infty \frac{(\theta + n -1)^2n^2 n!}{\Pn{\theta}{2n}}
\int f_n(x)^2\, \Pnm{\rho}{n}(\dd x) < \infty, 
\end{equation}
where the kernel functions $f_n$, $n\in\N$, are again given by \eqref{EqChaosZerlegung}.
A comparison with \eqref{domNabla} shows that 
$\operatorname{dom}(L)\subset\operatorname{dom}(\nabla)$.
Justified by the following lemma we define for $F\in \operatorname{dom}(L)$ 
\begin{equation}\label{L}
LF:=-\sum_{n=1}^\infty (\theta + n -1)n \int  f_n(x) \, \zeta^n(\dd x).
\end{equation}

\begin{Lemma} Let $F\in \operatorname{dom}(L)$ with chaos expansion \eqref{EqChaosZerlegung}. Then the series
on the right-hand side of \eqref{L}
converges in $L^2(\BP)$.
\end{Lemma}

\begin{proof}
Let $m, n \in\N$ with $m\le n$. The orthogonality relations \eqref{KorOrhtoFmFn} yield
\begin{equation*}
    \BE \left(\sum_{k=m}^{n} k(\theta + k -1) \int  f_k(x) \, \zeta^k(\dd x) \right)^2
    = \sum_{k=m}^{n} \frac{k^2 (\theta+k-1)^2k!}{\Pn{\theta}{2k}} \int  f_k(x)^2\, \Pnm{\rho}{k}(\dd x).
\end{equation*}
Consequently, the partial sums of the series under consideration form a Cauchy sequence in $L^2(\BP)$ if and only if $F\in\operatorname{dom}(L)$.
\end{proof}

Since \eqref{L} provides the chaos expansion of $LF$, we immediately have
\begin{equation}\label{Lposdef}
\BE(LF)G=-\sum_{n=1}^\infty  \frac{(\theta + n -1)n n!}{\Pn{\theta}{2n}}
\int  f_n(x)g_n(x)\,\rho^{[n]}(\dd x), \quad F\in \operatorname{dom}(L),\,G\in L^2(\zeta). 
\end{equation}
In particular, we see that $L$ is negative semi-definite and that 
\begin{equation}\label{Lsymmetric}
\BE (LF)G=\BE F (LG),\quad F,G\in \operatorname{dom}(L).
\end{equation}
Since $LF$ is the $L^2$-limit of centred random variables, we further
have $\BE LF=0$ for each $F\in\operatorname{dom}(L)$.

For reasons that will become clear in the next section, we refer
to the linear mapping $L\colon \operatorname{dom}(L)\to L^2(\BP)$ 
as the {\em Fleming--Viot operator}. It plays a similar role as the
Ornstein--Uhlenbeck operator in a Gaussian or Poisson context.
This fact is supported by our next theorem.

\begin{theorem}\label{t:deltanabla} Suppose that $F\in \operatorname{dom}(\nabla)$. Then 
$\nabla F\in \operatorname{dom}(\delta)$ iff $F\in\operatorname{dom}(L)$. In this case
\begin{equation}\label{deltanabla}
\delta(\nabla F)=-LF,\quad \BP\text{-a.s.}
\end{equation}
\end{theorem}
\begin{proof}
Let $F\in\operatorname{dom}(\nabla)$ with chaos expansion \eqref{EqChaosZerlegung}.

Assume first that $F\in\operatorname{dom}(L)$. 
Consider identity \eqref{e:iso} for $G\in\operatorname{dom}(\nabla)$. 
Applying 
the Cauchy--Schwarz inequality twice, we see that the right-hand side is bounded by
\begin{equation*}
    \left(\sum_{n=1}^\infty \frac{n^2n!(\theta +n-1)}{\Pn{\theta}{2n}} \int f_n(x)^2 \, \Pnm{\rho}{n}(\dd x)\right)^{\frac{1}{2}}
    \left(\sum_{n=1}^\infty \frac{n!}{\Pn{\theta}{2n}} \int  g_n(x)^2 \, \Pnm{\rho}{n}(\dd x) \right)^{\frac{1}{2}}.
\end{equation*}
Since $F\in\operatorname{dom}(L)$ and $G\in\operatorname{dom}(\nabla)$, we obtain
\begin{equation}
\BE \int \nabla_xF\nabla_xG\,\zeta(\dd x)\le c\|G\|_2,
\end{equation}
where $c=\|LF\|_2< \infty$. Thus, $\nabla F$ is an element of $\operatorname{dom}(\delta)$.

Assume, conversely, that $\nabla F \in \operatorname{dom}(\delta)$. 
Then we obtain from the integration by parts formula \eqref{PI} that
\begin{equation}\label{4.3}
    \BE \delta(\nabla F) G = \BE \int  (\nabla_x F) (\nabla_x G) \, \zeta(\dd x)
\end{equation}
for all $G\in\operatorname{dom}(\nabla)$. We set $H\coloneqq \delta(\nabla F)$. 
Let the chaos expansions of $H$ and $G\in\operatorname{dom}(\nabla)$ be given by
\begin{equation*}
 H=  \sum_{n=1}^\infty \int h_n(x) \, \zeta^n(\dd x) \qquad \text{and}
\qquad G = \BE G + \sum_{n=1}^\infty \int g_n(x) \, \zeta^n(\dd x),
\end{equation*}
respectively, where we recall that $\BE H=0$. By \eqref{e:iso0},
\begin{equation} \label{eq SKP in L^2(P) und L^2(Czeta) 1}
\BE \delta(\nabla F) G 
=  \sum_{n=1}^\infty \frac{n!}{\Pn{\theta}{2n}} \int h_n(x) g_n(x) \, \Pnm{\rho}{n}(\dd x). 
\end{equation}
Therefore, we obtain from \eqref{4.3} and \eqref{e:iso}
\begin{equation}\label{e:4.678}
\sum_{n=1}^\infty \frac{nn!(\theta +n-1)}{\Pn{\theta}{2n}} \int f_n(x) g_n(x) \, \Pnm{\rho}{n}(\dd x)
=\sum_{n=1}^\infty \frac{n!}{\Pn{\theta}{2n}} \int h_n(x) g_n(x) \, \Pnm{\rho}{n}(\dd x)
\end{equation}
and consequently 
\begin{equation*}
 \int   (n(\theta+n-1)f_n(x) - h_n(x)) g(x) \, \Pnm{\rho}{n}(\dd x) = 0
\end{equation*}
for each $n\in\N$ and each $g\in\bH_n$. Choosing $g=n(\theta+n-1)f_n - h_n$, gives
\begin{equation}\label{e:4.56}
    h_n = n(\theta+n-1)  f_n, \quad \Pnm{\rho}{n}\text{-a.e.},\, n\in\N.
\end{equation}
Since $H\in L^2(\zeta)$, the convergence of
\begin{equation*}
    \sum_{n=1}^\infty \frac{n!}{\Pn{\theta}{2n}} \int h_ n(x)^2 \, \Pnm{\rho}{n}(\dd x) 
= \sum_{n=1}^\infty \frac{n!n^2(\theta + n-1)^2}{\Pn{\theta}{2n}} \int  f_n(x)^2 \, \Pnm{\rho}{n}(\dd x)
\end{equation*}
yields $F\in\operatorname{dom}(L)$. Finally we obtain from
\eqref{eq SKP in L^2(P) und L^2(Czeta) 1} and \eqref{e:4.56} that
\begin{equation*}
\BE \delta(\nabla F) G 
=  \sum_{n=1}^\infty \frac{n!n(\theta + n-1)}{\Pn{\theta}{2n}} \int f_n(x) g_n(x) \, \Pnm{\rho}{n}(\dd x)
    = \BE (-L(F))G. 
\end{equation*}
Since this is true for all $G\in \operatorname{dom}(\nabla)$, we obtain that $\delta(\nabla F)$
and $-LF$ coincide $\BP$-a.s.
\end{proof}

\section{The Fleming--Viot process}\label{sec:FlemingViot}

In this section we assume $\BX$ to be a locally compact Polish space equipped with the Borel $\sigma$-field 
$\mathcal{X}$. The probability measure $\Di(\rho)$ is the unique invariant 
distribution of the measure-valued Fleming--Viot process
(with parent-independent mutation) as
introduced and studied in \cite{FlemingViot79}. 
The generator of this process can be explicitly defined on smooth functions
as follows. 
Let $\mathbf{F}_b(\BX)$ denote the space of all bounded and measurable functions $f\colon\BX\to\R$.
Let $\bS$ denote the space of all functions $F\colon \bM_1(\BX)\to\R$ of
the form $F(\mu)=\varphi(\mu(f_1), \ldots, \mu(f_d))$, where $d\in\N$,
$f_1,\ldots,f_d\in \mathbf{F}_b(\BX)$ and $\varphi\in C^\infty(\R^d)$.
For such an $F$ we define $L_\rho F\colon \bM_1(\BX) \to \R$ by
\begin{align*}
    (L_{\rho} F) (\mu):=\frac{1}{2} &\sum_{i,j= 1}^d (\partial^2_{i,j} \varphi)\left(\mu(f_1), \ldots, \mu(f_d) \right) 
\CV_\mu(f_i,f_j) \\
    &\quad+  \frac{1}{2}\sum_{i= 1}^d (\partial_i\varphi)\left(\mu(f_1), \ldots, \mu(f_d) \right) \mu (A f_i),\quad \mu\in \bM_1(\BX),
\end{align*}
where $\partial_i\varphi$ and $\partial^2_{i,j}\varphi$ denote the first and second order partial derivatives of $\varphi$
and the {\em mutation operator} $A\colon \mathbf{F}_b(\BX)\to \mathbf{F}_b(\BX)$ is defined by
\begin{equation} \label{eq Mutationsop}
    Af (x):=\int (f(y) - f(x)) \, \rho(\dd y), \quad x \in \BX. 
\end{equation}
We will see below that
\begin{equation}\label{Lsmooth}
(L_\rho F)(\zeta)=\frac{1}{2}L F,\quad \BP\text{-a.s.},\, F\in\bS.
\end{equation}
In fact, we shall describe the closure (in $L^2(\Di(\rho)$) of the bilinear  form associated with $L_\rho$
explicitly in terms of Malliavin operators. Moreover, we shall see that $\frac{1}{2}L$ is 
the generator of this form. We start with the following lemma from
\cite{Peccati08}.

\begin{lemma}\label{Sdense}
The set $\bS$ is dense in $L^2(\Di(\rho))$.
\end{lemma}
\begin{proof} 
Let $m\in\N$ and $h_1,\ldots,h_m\in \mathbf{F}_b(\BX)$. Then
$\mu\mapsto \mu^m(h_1\otimes\cdots\otimes h_m)$ is an element of $\bS$.
Furthermore, $\bS$ contains the constant functions. Since $\bS$ is a linear space,
the assertion follows from \cite[Lemma 2]{Peccati08}. 
\end{proof}

Let us introduce a bilinear form 
$\cE\colon \operatorname{dom}(\nabla)\times \operatorname{dom}(\nabla)\to\R$
by 
\begin{equation}
\cE(F,G):=\BE \int (\nabla_xF)(\nabla_xG)\,\zeta(\dd x),\quad F,G\in \operatorname{dom}(\nabla).
\end{equation}
By \eqref{gradientmeanzero}, we can write this as
\begin{equation}
\cE(F,G)=\BE \CV_\zeta (\nabla F,\nabla G),
\end{equation}
where we use the (quite natural) notation 
\begin{equation*}
\CV_\zeta (H,\tilde{H}):=\int H_x\tilde{H}_x\,\zeta(\dd x)-\int H_x\,\zeta(\dd x) \int\tilde{H}_x\,\zeta(\dd x)
\end{equation*}
for measurable functions $(\omega,x)\mapsto H_x(\omega)$ and $(\omega,x)\mapsto \tilde{H}_x(\omega)$ in $L^2(C_\zeta)$.
As usual we abbreviate $\cE(F):=\cE(F,F)$.
The following lemma shows that the form $\cE$ is {\em closed}; see e.g.\ 
\cite{FukushimaOshimaTakeda1994}. This means that
$\operatorname{dom}(\nabla)$ equipped with the scalar product 
$(F,G) \mapsto \cE(F,G) + \BE F G$ is complete.

\begin{lemma} \label{LemGammaAbg}
Let $F_n\in\operatorname{dom}(\nabla)$, $n\in\N$, and assume that $\lim_{m,n\to\infty}\cE(F_m - F_n) = 0$
as well as $\lim_{m,n\to\infty}(\BE F_m - \BE F_n) = 0$.
Then there exists $F\in\operatorname{dom}(\nabla)$ with
$\lim_{n\to\infty} \cE(F - F_n) = 0$ and $F_n\to F$ in $L^2(\BP)$.
\end{lemma}
\begin{proof}
The assertion is a consequence Lemma~\ref{LemGradientAbg}.
\end{proof}

For $F\in \operatorname{dom}(L)$ we can use Theorem \ref{t:deltanabla}
and partial integration  to obtain that 
\begin{equation}\label{e:dirichletgenerator}
\BE\int \nabla_x F \nabla_x G\,\zeta(\dd x)=\BE(-LF)G,\quad G\in \operatorname{dom}(\nabla),
\end{equation}
that is
\begin{equation}\label{PI2}
\BE (-LF)G=\cE(F,G),\quad F\in \operatorname{dom}(L),G\in \operatorname{dom}(\nabla).
\end{equation}
We wish to identify $\cE$ with the {\em Dirichlet form} associated with
the generator
$L_\rho$; see  \cite{FukushimaOshimaTakeda1994} for an introduction into
the theory  of Dirichlet forms
and \cite{Overbecketal95} for a thorough study of Dirichlet forms
associated with the Fleming--Viot
process and some of its generalizations.
In order to do so, we need to relate our Malliavin gradient to a
pathwise defined gradient defined on $\bS$ \cite{Stannat00,Shao11}.

For $F\in \bS$ we define $\nabla^*F\colon \Omega\times\BX\to \R$ by
\begin{equation}\label{e:nabla*}
\nabla^*F(\omega,x):=\sum^d_{i=1}\partial_i\varphi(\zeta(\omega)(g_1),\ldots,\zeta(\omega)(g_d))(g_i(x)-\zeta(g_i)). 
\end{equation}
Similarly as before we
denote the random variable $\omega\mapsto \nabla^*F(\omega,x)$ by $\nabla^*_xF$.
Clearly we have
\begin{equation*}
\nabla^*_xF=\frac{\dd}{\dd t}F((1-t)\zeta+t\delta_x)\Big|_{t=0},\quad x\in\BX,
\end{equation*}
showing that  $\nabla^*_xF$ can be interpreted as a directional derivative.

The following proposition shows
that $\nabla$ and $\nabla^*$ coincide on $\bS$ almost everywhere with respect to $C_\zeta$. 
In particular, the definition $\eqref{e:nabla*}$ does $C_\zeta$-a.e.\ not depend on the chosen 
representation of $F$.  

\begin{proposition}\label{l:nabla=nabla*} Let $F\in\bS$. Then $F\in\operatorname{dom}(\nabla)$ and
\begin{equation*}
\nabla F=\nabla^*F,\quad C_\zeta\text{-a.e.}
\end{equation*}
\end{proposition}
\begin{proof}
First we examine $\nabla^\ast F$. 
Since $F$ can be represented as
\begin{equation*}
	F(\mu)  = \varphi\big(\mu(f_1), \ldots, \mu(f_m) \big), \quad \mu \in \bM_1(\BX),
\end{equation*}
with $f_i \coloneqq h_i$, $i \in  [m]$, and $\varphi \colon \R^m \rightarrow \R$, $\varphi(x_1, \ldots, x_m) \coloneqq \prod_{i=1}^m x_i$, 
it is an element of $\bS$ and formula \eqref{e:nabla*} yields
\begin{equation*}
	\nabla_x^\ast F  = \sum_{i=1}^m  \big(\partial_{i} \varphi\big) \big(\zeta(h_1), \ldots, \zeta(h_m) \big)  \big(h_i (x)-\zeta(h_i)\big) 
    = \sum_{i=1}^m h_i (x) \prod_{\substack{j= 1 \\ j \neq i}}^m \zeta(h_j)  -  m\zeta^m(h).
\end{equation*}

Second we consider $\nabla F$. 
Let $x\in\BX$.
By Lemma~\ref{l:3.5}, we have \ $F \in \bigoplus_{i=0}^m \mathbf{F}_i$. Consequently, $F$ belongs to $\operatorname{dom}(\nabla)$ and 
we have
\begin{equation*}
	\nabla_x F 
= \sum_{n=1}^m n  \int f_n (x, \bfm{y}_{n-1}) \, \zeta^{n-1}(\dd \bfm{y}_{n-1}) 
-  \sum_{n=1}^m n\int f_n (\bfm{y}) \, \zeta^{n}(\dd \bfm{y}) \eqqcolon S(x) - \int S(z) \, \zeta(\dd z),
\end{equation*}
where the kernel functions of $F$ are denoted by $f_n$, $n\in\N$.  
Thus, it suffices to determine the part of $S(x)$ that depends on $x$.
Using~\eqref{e:3.25}, we split $S(x)$ according to whether the
distinguished variable $x$ occurs. The
terms not involving $x$ vanish
in $\nabla F$. Applying
Corollary~\ref{Kor Rho+delta_x_1+...+delta_x_k} to the remaining (i.e.\ $x$-dependent) terms
shows that these terms equal
\begin{align*}
&\sum_{n=1}^m \frac{\theta+2n-1}{(n-1)!}
\sum_{j=1}^n(-1)^{n-j}
\frac{\Pn{(\theta+j)}{n-1}}{\Pn{(\theta+j)}{m}}
\binom{n-1}{j-1}                                      \\
&\quad \times \sum_{r=1}^m \quad \sideset{}{^{\ne}}\sum_{i_1,\dots,i_r\in [m]} \int h_{\otimes i_1, \ldots, i_r}^j (\bfm{y}_{j-1},x) \, 
\zeta^{j-1}(\dd \bfm{y}_{j-1}) \Pnm{\rho}{m-r}(h^{\otimes i_1, \ldots, i_r})
\eqqcolon T(x).
\end{align*}
Interchanging the sums over $n$ and $j$ and applying
Lemma~\ref{LemHypergeoSumme 1}, gives 
\begin{align}\label{e:5.8} 
T(x)
=\sum_{r,j=1}^m & \quad \sideset{}{^{\ne}}\sum_{i_1,\dots,i_r\in [m]} \frac{(-1)^{m-j}}{(j-1)!(m-j)!} \times\\ \notag
&\times \sum_{ 1\le l_1 \le \cdots \le l_r\le j} \int h_{\otimes i_1, \ldots, i_r} (y_{l_1}, \ldots,y_{l_r}) 
\, \zeta^{j-1}(\dd \bfm{y})\Pnm{\rho}{m-r}(h^{\otimes i_1, \ldots, i_r}),
\end{align}
where we set $y_j \coloneqq x$, $j\in  [m]$. 

Let $r\in\N$ and suppose $1\le l_1 \le \cdots \le l_r\le j$ 
for some $j\in  [m]$. To exploit the symmetry in \eqref{e:5.8},
we decompose the summation
as follows. First we choose $i\in\{0,\ldots,r\}$ such that 
$l_{r-i+1} = \cdots = l_r =j$ and $l_{r-i}<j$.    
For $s\in [r-i]$ let $n_s$ be the number of distinct elements  
of the set $\{l_1,\ldots,l_{r-i}\}$ which occur $s$ times 
as component of the vector $(l_1,\ldots,l_{r-i})$. Then $k:=n_1+\cdots+n_{r-i}$
is the cardinality of $\{l_1,\ldots,l_{r-i}\}$ and $\sum^{r-i}_{s=1}sn_s=r-i$.
If $k\ge 1$ (i.e.\ $i<r$), then
we write the non-zero entries of $(n_1,\cdots,n_{r-i})$ in weakly descending order
as $\boldsymbol{\lambda}=(\lambda_1,\ldots,\lambda_k)$. The case $k=0$ is
encoded by $\boldsymbol{\lambda}\coloneqq\emptyset$.
In accordance with the literature, we call $\boldsymbol{\lambda}$ a {\em partition}
of $r-i$ of {\em length} $k$. Given such a partition $\boldsymbol{\lambda}$ with length $k$ there are 
$\binom{j-1}{k}k!(\prod_{s=1}^{k} \lambda_s!)^{-1}$ sequences 
$1\le l_1 \le \cdots \le l_{r-i}\le j-1$ with $k$ different elements associated with $\boldsymbol{\lambda}$.
Since 
\begin{equation*}
    y \mapsto \sideset{}{^{\ne}}\sum_{i_1,\dots,i_r\in[m]} h_{\otimes i_1, \ldots, i_r}(y) \Pnm{\rho}{m-r}(h^{\otimes i_1, \ldots, i_r})
\end{equation*}
is a symmetric function on $\BX^r$, equation \eqref{e:5.8} can be written as
\begin{equation*}
\gl{T(x)}=\sum_{r,j=1}^m  \frac{(-1)^{m-j}}{(j-1)!(m-j)!} \sum_{i=0}^{r} 
\sum_{k=0}^{( j-1) \land (r-i)} \sum_{\boldsymbol{\lambda}\in \mathcal{P}_{k,r-i}} \binom{j-1}{k} \frac{k!}{\prod_{s=1}^{k} \lambda_s!}
I_{\boldsymbol{\lambda},k,r,i}(x),
\end{equation*}
where for $k\in\N$, $r\in\N_0$ and $i\in\{0,\ldots,r\}$, $\mathcal{P}_{k,r-i}$ is the set of all partitions of $r-i$  of size $k$ and
\begin{align*}
&I_{\boldsymbol{\lambda},k,r,i}(x)\\
&:=\int \sideset{}{^{\ne}}\sum_{i_1,\dots,i_r\in[m]} h_{\otimes i_1, \ldots, i_r} (\underbrace{y_{1}, \ldots, y_1}_{\lambda_1 \text{ times}}, \ldots, 
\underbrace{y_{k}, \ldots, y_k}_{\lambda_k \text{ times}}, \underbrace{x, \ldots, x}_{i \text{ times}}) 
\,\Pnm{\rho}{m-r}(h^{\otimes i_1, \ldots, i_r}) \,\zeta^{k}(\dd \bfm{y}_{k}).
\end{align*}
(For $\boldsymbol{\lambda}=\emptyset$, that is  $r=i$ and $k=0$, we set $\prod_{s=1}^{k} \lambda_s!:=1$.)
Changing the order of summation and simplifying the combinatorial coefficients gives
\begin{equation*}
T(x)=\sum_{r=1}^m\sum_{i=0}^r\sum_{k=0}^{r-i}  \sum^m_{j=k+1}\frac{(-1)^{m-j}}{(j-1)!(j-1-k)!} 
\sum_{\boldsymbol{\lambda}\in \mathcal{P}_{k,r-i}}\frac{1}{\prod_{s=1}^{k} \lambda_s!}
I_{\boldsymbol{\lambda},k,r,i}(x).
\end{equation*}
Here the sum over $j$ comes to $0$ unless $k=m-1$, in which case it equals $1$.
Since $k=m-1\le r-i\le m-i$, necessarily $i\in\{0,1\}$.
The terms corresponding to $i=0$ do not depend on $x$ and therefore
disappear in $\nabla F$. If $i=1$, then
$r=m$. The remaining contribution of $T(x)$ is
\[
    \sum_{\boldsymbol{\lambda}\in \mathcal{P}_{m-1,m-1}}\frac{1}{\prod_{s=1}^{m-1} \lambda_s!}
    I_{\boldsymbol{\lambda},m-1,m,1}(x) =  \sum_{i_m = 1}^m h_{i_m}(x) \int h^{\otimes i_m} (\bfm{y}_{m-1}) \,\zeta^{m-1}(\dd \bfm{y}_{m-1}).
\]
In conclusion we obtain 
\begin{align*}
\nabla_x F &= S(x) - \int S(z) \, \zeta(\dd z) \\
&= \sum_{i = 1}^m h_{i}(x) \int h^{\otimes i} (\bfm{y}) \,\zeta^{m-1}(\dd \bfm{y}) 
- \int \sum_{i = 1}^m h_{i}(z) \int h^{\otimes i} (\bfm{y}) \,\zeta^{m-1}(\dd \bfm{y}) \, \zeta(\dd z),
\end{align*}
which equals  $\nabla_x^\ast F$. This finishes the proof.
\end{proof}

\begin{remark}\label{r:discretegradient}\rm Let $F\colon\bM_1(\BX)\to \R$ be measurable.
The authors of \cite{FlintTorrisi23} have introduced a 
discrete space-size gradient  $(x,t)\mapsto D_{x,t}F$ of $F$ by
\begin{equation}
D_{x,t}F(\mu):=F((1-t)\mu+t\delta_x)-F(\mu),\quad (x,t)\in\BX\times[0,1],\,\mu\in \bM_1(\BX).
\end{equation}
Proposition \ref{l:nabla=nabla*} shows that
\begin{equation*}
\nabla_xF(\zeta(\omega))=\frac{\dd}{\dd t}D_{x,t}F(\zeta(\omega))\Big|_{t=0},\quad C_\zeta\text{-a.e.\ $(\omega,x)$},
\end{equation*}
provided that $F\in\bS$. As far as we can see, there is no simple way to  relate the associated discrete divergence from
\cite{FlintTorrisi23} with our Malliavin divergence; see also Remark \ref{r:discretedivergence}.
\end{remark}

Let us introduce a bilinear form $\cE_1$ on 
$\operatorname{dom}(\nabla)\times \operatorname{dom}(\nabla)$
by
\begin{equation*}
    \cE_1(F,G) \coloneqq \BE F(\zeta)G(\zeta) + \cE(F,G), \quad F,G\in\operatorname{dom}(\nabla).
\end{equation*}

\begin{proposition}\label{t:closure} The closure of $\bS$ with respect to 
$\cE_1$ coincides with $\operatorname{dom}(\nabla)$.
\end{proposition}
\begin{proof} We already know from Lemma \ref{LemGammaAbg} and Proposition \ref{l:nabla=nabla*} 
that $\bS\subset \operatorname{dom}(\nabla)$ and that $\operatorname{dom}(\nabla)$ is closed.
It remains to show that each $F\in\operatorname{dom}(\nabla)$ can be approximated by members of
$\bS$. By \eqref{e:iso0} and \eqref{e:iso}, we have
\begin{equation*}
\cE_1(F,F)=
\sum^\infty_{n=0}\frac{n!+(\theta+n-1) nn!}{\Pn{\theta}{2n}}\rho^{[n]}\big(f_n^2\big). 
\end{equation*}
Let $\varepsilon>0$ and choose $k\in\N$ such that
\begin{equation*}
\sum^\infty_{n=k+1}\frac{n!+(\theta+n-1) nn!}{\Pn{\theta}{2n}}\rho^{[n]}\big(f_n^2\big)
\le\varepsilon.
\end{equation*}
We consider the square integrable random variable
\begin{equation*}
F_0:=\sum^k_{n=0}\zeta^n(f_n).
\end{equation*}
Let us introduce the subspace $\bS_0$ of $\bS$  spanned by
the functions
\begin{equation*}
\mu\mapsto \mu(h_1)\cdot\ldots\cdot \mu(h_m),
\end{equation*}
where $m\in\N$ and $h_1,\ldots,h_m\in \mathbf{F}_b(\BX)$.
We have seen in the proof of Lemma \ref{Sdense} that
$\bS_0$ is dense in $L^2(\Di(\rho))$. Hence, there exists 
$\tilde{G}\in\bS_0$ such that $\BE (F_0-\tilde{G}(\zeta))^2\le c^{-1}\varepsilon$, where
$c:=1+(\theta+k-1)k$. Let $G$ be the orthogonal projection of
$\tilde{G}(\zeta)$ onto $\bF_0\oplus\cdots \oplus\bF_k$.
Then $\BE(F_0-G)^2\le \BE(F_0-\tilde{G})^2$. 
Let $g_n$, $0\le n\le k$, denote the kernel functions of $G$. 
By definition of $\bS_0$, Lemma \ref{l:3.5} and Corollary \ref{Kor Rho+delta_x_1+...+delta_x_k},
we have $\BP$-a.s.\ that $G=G_0(\zeta)$ for some $G_0\in\bS_0$. Moreover,
\begin{align*}
\cE_1(F-G,F-G)&=
\sum^\infty_{n=0}\frac{n!+(\theta+n-1) nn!}{\Pn{\theta}{2n}}\rho^{[n]}\big(f_n-g_n\big)^2\\
&\le\varepsilon+\sum^k_{n=0}\frac{n!+(\theta+n-1) nn!}{\Pn{\theta}{2n}}\rho^{[n]}\big(f_n-g_n\big)^2\\ 
&\le\varepsilon+c\sum^k_{n=0}\frac{n!}{\Pn{\theta}{2n}}\rho^{[n]}\big(f_n-g_n\big)^2
=\varepsilon+c\,\BE (F_0-G)^2\le 2\varepsilon.
\end{align*}
This concludes the proof.
\end{proof}

Define a bilinear form  $\cE^*\colon \bS\times \bS\to\R$ by 
\begin{equation}\label{e:E*}
\cE^*(F,G):=\BE\int \nabla^*_xF \nabla^*_xG\,\zeta(dx),\quad F,G\in \bS.
\end{equation}
It follows from Propositions \ref{l:nabla=nabla*} and \ref{t:closure} that
$(\operatorname{dom}(\nabla)\times \operatorname{dom}(\nabla),\cE)$ 
is the closure of  $(\bS\times\bS,\cE^*)$; see \cite{FukushimaOshimaTakeda1994}
for terminology. Therefore, we obtain from the general theory in \cite{FukushimaOshimaTakeda1994}
(and elementary properties of $\cE^*$) that 
$(\operatorname{dom}(\nabla),\cE)$  is a Dirichlet form; see also \cite{Stannat00}.

The {\em generator} $\tilde L$ associated with the Dirichlet form $(\operatorname{dom}(\nabla),\cE)$ 
is  a linear mapping from a subspace of $\operatorname{dom}(\nabla)$ into $L^2(\zeta)$
which is defined as follows. The domain
$\operatorname{dom}(\tilde{L})$ of $\tilde{L}$ is the set of all $F\in \operatorname{dom}(\nabla)$
such that there exists $H\in L^2(\zeta)$ satisfying
\begin{equation}\label{e:pi}
\cE(F,G)= \BE H G,\quad G\in\operatorname{dom}(\nabla).
\end{equation}
In this case one defines $\tilde{L}(F):=H(\zeta)$. Our next theorem shows  that $\tilde L$ coincides
with the operator $L$, introduced by \eqref{L} via chaos expansion. This justifies
to refer to $L$ as Fleming--Viot operator.

\begin{theorem}\label{t:generator} We have
$\operatorname{dom}(L)= \operatorname{dom}(\tilde{L})$ and
\begin{equation}\label{chaosgenerator}
\tilde{L}F=LF,\quad \BP\text{-a.s.},\, F\in \operatorname{dom}(L).
\end{equation}
\end{theorem}
\begin{proof}
The inclusion  $\operatorname{dom}(L)\subset \operatorname{dom}(\tilde L)$
and $\tilde L F = L F$, $\BP$-a.s.\ for   $F\in\operatorname{dom}(L)$ follow from
\eqref{PI2}.

Conversely, let $F\in\operatorname{dom}(\tilde L)$,
$G\in\operatorname{dom}(\nabla)$ and $H\in L^2(\zeta)$ satisfy
\eqref{e:pi}. 
Choosing $G \equiv 1$
shows $\BE H=0$. Let $h_n$, $n\in\N$, denote the kernel
functions in the chaos expansion of $H$. We can proceed exactly as in
the proof of Theorem~\ref{t:deltanabla}
(cf.~\eqref{eq SKP in  L^2(P) und L^2(Czeta) 1} and
\eqref{e:4.678}) to show  $h_n = (\theta + n-1) f_n$, $n\in\N$. 
Because of $H\in L^2(\zeta)$, the series
\begin{equation*}
\BE H(\zeta)^2= \sum_{n=1}^\infty \frac{n!}{\Pn{\theta}{2n}}\Pnm{\rho}{n}(h_n^2) 
=\sum_{n=1}^\infty \frac{n!n^2(\theta+n-1)^2}{\Pn{\theta}{2n}}\Pnm{\rho}{n}(f_n^2) 
\end{equation*}
converges. Thus, $F\in\operatorname{dom}(L)$. 
\end{proof}

\begin{remark}\label{r:5.3}\rm
It was shown in \cite{Stannat00} that
$(L_\rho F)(\zeta)=\frac{1}{2}\tilde{L} F$ holds $\BP$-a.s.\ for each $F\in\bS$.
Hence, we obtain \eqref{Lsmooth} from Theorem \ref{t:generator}.
Since there is a one-to-one correspondence between
closed symmetric forms and non-positive definite
self-adjoint operators (cf.\ Theorem~1.3.1
in \cite{FukushimaOshimaTakeda1994}) and $\cE$ is generated by $L$,
we further obtain that $L$ is the closure of $2L_{\rho}$.
\end{remark}

Given $t\ge 0$, we introduce a linear operator $T_t \colon L^2(\zeta) \to L^2(\zeta)$
by
\begin{equation*}
    T_t(F) := \sum_{n=0}^\infty e^{-n(\theta+n-1)t}\int f_n(x) \, \zeta^n(\dd x),
\end{equation*}
where the $f_n$, $n\in\N_0$, are the kernel function of $F\in L^2(\zeta)$. 
By \eqref{e:iso0} and dominated convergence, this series converges in $L^2(\BP)$ and we have
$T_0F=F$ and the contraction property
\begin{equation*}
\BE T_t(F)^2 \le \BE F^2,\quad t\ge 0.
\end{equation*}
Furthermore, we have the semigroup property $T_s(T_t F)=T_{s+t}F$,  $\BP$-a.s., $s,t\ge 0$,
the $L^2(\BP)$-convergence $T_tF\to F$ as $t\to 0$ and the symmetry property
\begin{equation}
\BE T_t(F)G=\BE FT_t(G),\quad F,G\in L^2(\zeta).
\end{equation}
Proofs are straightforward and left to the reader. These properties
make $\{T_t:t\ge 0\}$ a strongly
continuous semigroup in the sense of \cite{FukushimaOshimaTakeda1994}.
The following result shows that $L$ is the generator of this semigroup.

\begin{proposition} Suppose that $H\in L^2(\zeta)$. Then 
$t^{-1}(T_tF-F)\to H$ in $L^2(\BP)$ as $t\to 0$ if and only if
$F\in \operatorname{dom}(L)$. In this case $LF=H$ $\BP$-a.s.
\end{proposition}
\begin{proof} Let $f_n$, $n\in\N_0$, denote the kernel functions of $F$.
Assume that $F\in\operatorname{dom}(L)$. 
By the definitions and orthogonality, 
\begin{align*}
\BE \left(\frac{1}{t} (T_t(F) - F) - L(F) \right)^2 
 &= \BE \left(\sum_{n=0}^\infty 
\left(\frac{1}{t}\left(e^{-n(\theta+n-1)t}-1\right) + n(\theta + n -1)\right) \zeta^n(f_n) \right)^2  \\
 &= \sum_{n=0}^\infty \left(\frac{e^{-n(\theta+n-1)t}-1}{t} + n(\theta + n -1)\right)^2
\frac{n!}{\Pn{\theta}{2n}}\rho^{[n]}(f_n^2). 
\end{align*}
Since $|t^{-1}(e^{-n(\theta+n-1)t} - 1)| \le n(\theta + n -1)$ for all
$n\in \N$ and all $t\ge 0$, we can apply dominated convergence to see that
the latter series tends to $0$ as $t\to 0$.

Assume now that $t^{-1}(T_t(F) - F) \to H$ in $L^2(\BP)$ as $t\to 0$
and denote the kernel functions of $H$ by $h_n$, $n\in\N_0$.
Since
\begin{equation*}
\BE \left(\frac{1}{t}(T_t(F) - F) - H\right)^2
    = \sum_{n=0}^\infty \frac{n!}{\Pn{\theta}{2n}} 
\int  \left(\frac{e^{-n(\theta+n-1)t}-1}{t} f_n(x) - h_n(x)\right)^2 \Pnm{\rho}{n}(\dd x),
\end{equation*}
we obtain that $f_0=h_0$ and, for each $n\in\N$,
\begin{equation*}
\lim_{t\to 0} t^{-1}\big(e^{-n(\theta+n-1)t}-1\big) f_n(x)=n(\theta+n-1)f_n(x)=h_n(x),
\quad \rho^{[n]}\text{-a.e.\ $x$}.
\end{equation*}
We conclude that $F\in\operatorname{dom}(L)$ and $\BP(H = L(F))=1$.
\end{proof}

\section{Further properties of the Malliavin operators}\label{sec:furtherproperties}

In this section we provide further properties of gradient and divergence.
The results are similar to the Gaussian case, but some proofs
require different arguments.
We begin with an analogue of \cite[Proposition 1.3.3]{Nualart2006};
see also \cite[Proposition 2.5.4]{NourdinPeccati2012}.

\begin{proposition}\label{p:deltaF} Suppose that  $F\in\operatorname{dom}(\nabla)$ and
$H\in\operatorname{dom}(\delta)$. Assume that
$\BE F^2\delta(H)^2<\infty$ and $\BE \int F^2H^2_x\,\zeta(\dd x)<\infty$. Then 
$FH\in\operatorname{dom}(\delta)$ and
\begin{equation}\label{e:deltaF}
\delta(FH)=F\delta(H)-\int \nabla_xF H_x\,\zeta(\dd x),\quad \BP\text{-a.s.}
\end{equation}
\end{proposition}
\begin{proof} Our argument is based on the useful product rule
\begin{equation}\label{e:productrule}
\nabla(FG)=(\nabla F)G+F(\nabla G),\quad C_\zeta\text{-a.e.}
\end{equation}
for $F,G\in\bS$. This follows from Proposition \ref{l:nabla=nabla*}
and the elementary fact that $\nabla^*$ (defined by \eqref{e:nabla*}) satisfies
this rule everywhere on $\Omega\times\BX$.

We first assume that $F\in\bS$. We take $G\in\bS$ and note that 
$FG\in\bS\subset \operatorname{dom}(\nabla)$.
We then obtain from
\eqref{e:productrule} and partial integration that
\begin{align*}
\BE \int (\nabla_xG)FH_x\,\zeta(\dd x)&=\BE \int \nabla_x(GF)H_x\,\zeta(\dd x)-\BE \int (\nabla_xF)GH_x\,\zeta(\dd x)\\
&=\BE GF\delta(H)-\BE G\int (\nabla_xF)H_x\,\zeta(\dd x).
\end{align*}
Since $G\in\bS$ is arbitrary, this shows both 
$FH\in\operatorname{dom}(\delta)$ and \eqref{e:deltaF}.

Now we take a general $F\in \operatorname{dom}(\nabla)$. By Proposition \ref{t:closure},
there exist $F_n\in\bS$, $n\in\N$, such that $F_n\to F$ in $L^2(\BP)$ and
$\nabla F_n\to \nabla F$ in $L^2(C_\zeta)$. We have already proved that
\begin{equation}\label{e:deltaFn}
\delta(F_nH)=F_n\delta(H)-\int (\nabla_xF_n) H_x\,\zeta(\dd x),\quad \BP\text{-a.s.}
\end{equation}
We treat both sides of this equation  separately and take $G\in\bS$. By partial integration,
\begin{equation*}
\BE G\delta(F_nH)=\BE\int F_nH_x\nabla_xG \,\zeta(\dd x).
\end{equation*}
Furthermore,
\begin{equation*}
\bigg|\BE\int F H_x\nabla_xG \,\zeta(\dd x)-\BE\int F_nH_x\nabla_xG \,\zeta(\dd x)\bigg|
\le c\,\BE |F-F_n|\bigg|\int H_x \,\zeta(\dd x)\bigg|,
\end{equation*}
where $c$ is an upper bound for $|\nabla ^*G|$.
By the Cauchy–Schwarz inequality, this can be further bounded by
\begin{equation*}
c\,\|F-F_n\|_2\bigg(\BE\bigg(\int H_x \,\zeta(\dd x)\bigg)^2\bigg)^{1/2}
\le c\,\|F-F_n\|_2\bigg(\BE\bigg(\int H^2_x \,\zeta(\dd x)\bigg)\bigg)^{1/2}
\end{equation*}
which tends to $0$ as $n\to\infty$. Therefore,
\begin{equation*}
\lim_{n\to\infty}\BE G\delta(F_nH)=\BE\int F H_x\nabla_xG \,\zeta(\dd x).
\end{equation*}
Turning to the right-hand side of \eqref{e:deltaFn},
we note first that
\begin{equation*}
\lim_{n\to\infty} \BE GF_n\delta(H)=\BE GF\delta(H)
\end{equation*}
since $F_n\to F$ in $L^2(\BP)$ and $G$ is bounded. Furthermore,
\begin{align*}
\bigg|\BE G\int (\nabla_xF) &H_x\,\zeta(\dd x)-\BE G\int (\nabla_xF_n) H_x\,\zeta(\dd x)\bigg|
\le c\,\BE \int |\nabla_xF-\nabla_x F_n| |H_x|\,\zeta(\dd x)\\
&\le c\,\bigg(\BE \int (\nabla_xF-\nabla_x F_n)^2 \,\zeta(\dd x)\bigg)^{1/2}  \bigg(\BE \int H^2_x\,\zeta(\dd x)\bigg)^{1/2}
\end{align*}
which goes to $0$ as $n\to\infty$. Summarising, we obtain from \eqref{e:deltaFn} that
\begin{equation*}
\BE\int FH_x\nabla_xG \,\zeta(\dd x)=\BE GF\delta(H)-\BE G\int (\nabla_xF) H_x\,\zeta(\dd x).
\end{equation*}
Hence, the assertions follow from our integrability assumptions.
\end{proof}

The next result provides a pathwise representation of $\delta(H')$ for specific
$H'\in \operatorname{dom}(\delta)$. Similarly as in the classical Gaussian case,
$\delta(H')$ is the difference of a stochastic integral and a divergence term;
see \cite[Proposition 1.3.5]{Nualart2006} or \cite[eqn.\ (2.5.3)]{NourdinPeccati2012}.

\begin{proposition}\label{p:divergence} Suppose $F'\in\operatorname{dom}(\nabla)$ and $h\in L^2(\rho)$.
Define $H'(\mu,x):=F'(\mu)h(x)$. Assume that $F'$ is bounded. Then $H'\in\operatorname{dom}(\delta)$ and
\begin{equation}\label{e:067}
\delta(H')=\int h(x)F'\, (\theta\zeta-\rho)(\dd x)-\int h(x)\nabla_xF' \,\zeta(\dd x),\quad \BP\text{-a.s.}
\end{equation}
\end{proposition}
\begin{proof} We wish to apply Proposition \ref{p:deltaF} with $H(\mu,x)$ replaced
by $H^*(\mu,x):=h(x)$. By Corollary \ref{c:4.9}, the function $H^*$ is in the domain of $\delta$ and we have
\begin{equation*}
\delta(H^*)=\delta(h)=\theta\int h(x)\,\zeta(\dd x)-\int h(x)\,\rho(\dd x);
\end{equation*}
see Corollary \ref{r:delta(h)}.
Since $|F'|$ is bounded by some $c\ge 0$, say, we have 
$\BE (F')^2\delta(H^*)^2<\infty$ and 
\begin{equation}
\BE \int (F')^2(H^*_x)^2\,\zeta(\dd x)\le c^2\BE \int h(x)^2\,\zeta(\dd x)=\frac{c^2}{\theta}\int h(x)^2\,\rho(\dd x)<\infty.
\end{equation} 
Therefore, we obtain the assertion from Proposition \ref{p:deltaF}.
\end{proof}

\begin{remark}\label{r:discretedivergence}\rm Formula \eqref{e:067} can be written as
\begin{equation*}
\delta(H')=\int h(x)(\theta F'-\nabla_xF')\,\zeta(\dd x)-\int h(x)F' \,\rho(\dd x),\quad \BP\text{-a.s.}
\end{equation*}
which is of a similar form as the discrete gradient defined in \cite{FlintTorrisi23}. 
\end{remark}

We are now in a position to extend the product rule \eqref{e:productrule}.

\begin{proposition}\label{p:productrule} Suppose that $F,G\in \operatorname{dom}(\nabla)$
and assume that $FG\in\operatorname{dom}(\nabla)$. Then \eqref{e:productrule} holds. 
\end{proposition}
\begin{proof} Let 
$H(\mu,x):=h(x)F'(\mu)$, where $h\colon\BX\to\R$ is measurable
and bounded and $F'\in\bS$. By Proposition \ref{p:divergence},
$H\in \operatorname{dom}(\delta)$ and \eqref{e:067} shows that
$\delta(H)$ is bounded. Therefore, we can apply
Proposition \ref{p:deltaF} with both $F$ and $G$ to obtain that
\begin{align*}
\BE &\int H_x(\nabla_xF)G\,\zeta(\dd x)+\BE \int H_xF\nabla_xG\,\zeta(\dd x)=\BE \delta(GH)F+\BE \delta(FH)G\\
&=\BE \delta(H)FG-\BE F\int \nabla_xG H_x\,\zeta(\dd x)
+\BE \delta(H)FG-\BE G\int \nabla_xF H_x\,\zeta(\dd x).
\end{align*}
Since $FG\in\operatorname{dom}(\nabla)$, this means that
\begin{equation}\label{e:MCT}
\BE \int H_x(\nabla_xF)G\,\zeta(\dd x)+\BE \int H_xF\nabla_xG\,\zeta(\dd x)
=\BE \int H_x\nabla_x(FG)\,\zeta(\dd x).
\end{equation}
The space $\mathbf{G}$ of all functions $H$ of the above type is closed under
(pointwise) multiplication and generates the $\sigma$-field
$\mathcal{M}_1\otimes\cX$. On the other hand, 
$\mathbf{G}$ is contained in the space of all bounded and measurable 
$H\colon\bM_1\times\BX\to\R$ satisfying \eqref{e:MCT}.
Therefore, a monotone class theorem (see e.g.\ \cite[Theorem A.4]{LastPenrose18})
shows that \eqref{e:MCT} holds for all bounded and measurable functions $H$.
This concludes the proof of \eqref{e:productrule}. 
\end{proof}

We continue with a chain rule which looks exactly the same as in the
Gaussian case.

\begin{proposition} Let $k\in\N$ and suppose that $F_1,\ldots,F_k\in\operatorname{dom}(\nabla)$.
Let $\varphi\colon\R^k\to\R$ be continuously differentiable  with bounded partial derivatives.
Then $\varphi(F_1,\ldots,F_k)\in\operatorname{dom}(\nabla)$ and
\begin{equation}
\nabla (\varphi(F_1,\ldots,F_k))=\sum^k_{i=1}\partial_i\varphi(F)\nabla F_i,\quad \BP\text{-a.s.}
\end{equation}
\end{proposition}
\begin{proof} If $F_1,\ldots,F_k\in\bS$, then the result
follows from Proposition \ref{l:nabla=nabla*} and the properties of $\nabla^*$.
To lighten the notation, we prove the general case in the case $k=1$, writing $F:=F_1$.
As in the proof of Proposition \ref{p:deltaF} we take $F_n\in\bS$, $n\in\N$, such that $F_n\to F$ in $L^2(\BP)$ and
$\nabla F_n\to \nabla F$ in $L^2(C_\zeta)$. Let us abbreviate the norm in $L^2(C_\zeta)$ by
$\|\cdot\|_\zeta$. We have
\begin{align*}
\|\varphi'(F)\nabla F-\nabla(\varphi(F_n))\|_\zeta
&=\|\varphi'(F)\nabla F-\varphi'(F_n)\nabla F_n\|_\zeta\\
&\le \|(\varphi'(F)-\varphi'(F_n))\nabla F\|_\zeta+\|\varphi'(F_n)(\nabla F-\nabla F_n)\|_\zeta.
 \end{align*}
The squared second term on the above right-hand side equals
$$
\BE\varphi'(F_n)^2\int (\nabla_x F-\nabla_x F_n)^2\,\zeta(\dd x),
$$
which tends to zero since $\varphi'$ is bounded.
The squared first term equals
\begin{equation*}
\BE(\varphi'(F)-\varphi'(F_n))^2\int (\nabla_x F)^2\,\zeta(\dd x).
\end{equation*}
Since $\varphi'$ is bounded, the integrand is bounded by an integrable random variable.
Since $F_n\to F$ in probability, it follows from dominated convergence that this first
term tends to zero as well. Hence,
\begin{equation*}
\lim_{n\to\infty}\|\varphi'(F)\nabla F-\nabla(\varphi(F_n))\|_\zeta=0.
\end{equation*}
Therefore, Lemma \ref{LemGradientAbg}
implies the assertions.
\end{proof}

\section{Covariance identities}\label{sec:covariances}

Equation \eqref{e:iso0} can be written as
\begin{equation}\label{e:cov}
\CV(F,G)=\sum^\infty_{n=1}\frac{n!}{\Pn{\theta}{2n}}\rho^{[n]}\big(f_ng_n\big),\quad F,G\in L^2(\zeta).
\end{equation}
Combining this with our expressions \eqref{kerneln} for the kernel functions,
we can easily derive explicit covariance formulas. If $F$ is in a fixed chaos and $G$ is in a finite
sum of chaoses, then their covariance is of a particularly appealing form:


\begin{theorem}\label{t:covariancechaos} Suppose that $h\in\bH_k$ for some  $k\in\N$.  Let $m\in\N$ and
$h_1,\ldots,h_m\colon\BX\to\R$ be measurable and bounded.
Then
\begin{align}\label{e:cov1}
\CV&[\zeta^k(h),\zeta^m(h_1\otimes\cdots \otimes h_m)]\\ \notag
&=\frac{1}{\theta^{(m+k)}}\sum^m_{r=1}\quad\sideset{}{^{\ne}}\sum_{i_1,\dots,i_r\in[m]}\rho^{[m-r]}\bigg(\bigotimes_{j\notin\{i_1,\ldots,i_r\}} h_j\bigg)
\int h^k_{\otimes_ {i_1, \ldots, i_r}}(\bfm{x})h(\bfm{x})\,\rho^{[k]}(\dd \bfm{x}),
\end{align} 
where $h^k_{\otimes_ {i_1, \ldots, i_r}}$ is given by \eqref{eq Def f_otimes^k}.
\end{theorem}
\begin{proof} We abbreviate $g:=\otimes^m_{j=1} h_j$. Since $g$ is bounded, we have $\zeta^m(g)\in L^2(\zeta)$.
By \eqref{e:cov} and orthogonality, 
\begin{equation}
\CV[\zeta^k(h),\zeta^m(g)]
=\frac{k!}{\Pn{\theta}{2k}}\rho^{[k]}\big(hf_k\big),
\end{equation}
where $f_k$ is the $k$-th kernel function of $\zeta^m(h^{\otimes m})$. This function is
given by \eqref{e:3.25}. Since $h\in\bH_k$, we have
\begin{equation*}
\rho^{[k]}\big(hf_k\big)
=\frac{\theta + 2k - 1}{k!}
 \frac{\Pn{(\theta+k)}{k-1}}{\Pn{(\theta+k)}{m}} 
\int \Pnm{(\rho + \delta_{\bfm{x}})}{m}(g) h(\bfm{x}) \,\rho^{[k]}(\dd \bfm{x}).
\end{equation*}
The combinatorial  coefficient in front of the integral 
simplifies to $1/\theta^{(m+k)}$.
Moreover, by Proposition \ref{PropRho+delta_x_1+...+delta_x_k} and symmetry of the measures $\rho^{[n]}$, $n\in\N$,
\begin{align*}
\Pnm{(\rho + &\delta_{\bfm{x}})}{m}(g)=\rho^{[m]}(g)\\
&+\sum_{r=1}^{m} \quad\sideset{}{^{\ne}}\sum_{i_1,\dots,i_r\in[m]} 
\rho^{[m-r]}\bigg(\bigotimes_{j\notin\{i_1,\ldots,i_r\}} h_j\bigg) \sum_{1\le n_1\le \cdots \le n_r\le k}
h_{i_1}(x_{n_1})\cdot\ldots \cdot h_{i_r}(x_{n_r}).
\end{align*}
Since $\rho^{[k]}(h)=0$, we obtain the assertion.
\end{proof}

\begin{remark}\rm Consider the assumptions of 
Theorem \ref{t:covariancechaos} and assume $m<k$. By definition of 
$h^{k}_{i_1,\ldots,i_r}$ and since $h\in\bH_k$, we have
$\int h^{k}_{i_1,\ldots,i_r}(\bfm{x})h(\bfm{x})\,\rho^{[k]}(\dd \bfm{x})=0$
for $r<k$, so that $\CV(\zeta^k(h),\zeta^m(g))=0$.
This is in accordance with 
$\zeta^m(h_1\otimes\cdots \otimes h_m)\in \bF_0\oplus\cdots\oplus \bF_m$
(see Lemma \ref{l:3.5}) and the fact that the latter space is orthogonal to $\bF_k$. 
\end{remark}

\begin{corollary}\label{c:cov} Suppose that $h\in\bH_k$ for some  $k\in\N$. Let
$f\colon\BX\to\R$ be measurable and bounded and $m\in \N$. Then
\begin{align}\label{e:cov2}
\CV&(\zeta^k(h),\zeta^m(f^{\otimes m}))\\ \notag
&=\frac{1}{\theta^{(m+k)}}\sum^m_{r=1}\frac{m!}{(m-r)!}\rho^{[m-r]}(f^{\otimes (m-r)})
\sum^r_{{\substack{j_1,\ldots,j_k=1\\j_1+\cdots+j_k=r}}}\int f(x_1)^{j_1}\cdot \ldots \cdot f(x_k)^{j_k}h(\bfm{x})
\,\rho^{[k]}(\dd \bfm{x}).
\end{align}
\end{corollary}
\begin{proof}
We apply Theorem \ref{t:covariancechaos} with $h_1=\cdots=h_m=f$. Then 
\begin{equation*}
f^k_{\otimes_ {i_1, \ldots, i_r}}(\bfm{x})=\sum_{1\le n_1\le\cdots\le n_r\le k}f(x_{n_1})\cdot \ldots \cdot f(x_{n_r})
=\sum^r_{{\substack{j_1,\ldots,j_k=0\\j_1+\cdots+j_k=r}}} f(x_1)^{j_1}\cdot \ldots \cdot f(x_k)^{j_k},
\end{equation*}
independently of $(i_1,\ldots,i_r)$. Hence, we obtain \eqref{e:cov2} from \eqref{e:cov1}
using that $h\in\bH_k$.
\end{proof}


In the special case $k=2$ Corollary \ref{c:cov} yields the following result.

\begin{corollary}\label{c:cov3} Suppose that $h\in\bH_2$. Let
$f\colon\BX\to\R$ be measurable and bounded and $m\in \N$. Then
\begin{align*}
&\CV(\zeta^2(h),\zeta^m(f^{\otimes m}))
=\frac{1}{\theta^{(m+2)}}\sum^m_{r=1}\frac{m!(r+1)}{(m-r)!}\rho^{[m-r]}(f^{\otimes (m-r)})
\int h(x,x)f(x)^r\,\rho(\dd x)\\
&+\frac{1}{\theta^{(m+2)}}\sum^m_{r=1}\frac{m!}{(m-r)!}\rho^{[m-r]}(f^{\otimes (m-r)})
\sum^{r-1}_{j=1} \int h(x_1,x_2)f(x_1)^j f(x_2)^{r-j}\,\rho^2(\dd (x_1,x_2)).
\end{align*}
\end{corollary}
\begin{proof} The result follows from \eqref{e:cov2} upon using 
$\rho^{[2]}(\dd (x_1,x_2))=(\rho+\delta_{x_1})(\dd x_2)\rho(\dd x_1)$.
\end{proof}


\begin{remark}\rm
For $m\in\N$ and $j\in [m]$ we denote by $\Pi_{m,j}$
the system of all partitions $\sigma$ of $ [m]$ whose cardinality
$|\sigma|$ equals $j$. We write $\sigma\in\Pi_{m,j}$
as $\sigma=\{I_1(\sigma),\ldots,I_j(\sigma)\}$.
Let $f_1,\ldots,f_m\colon\BX\to \R$ be bounded and measurable.
It was shown  in \cite{Ethier90} 
that 
\begin{equation}\label{e:854}
\rho^{[m]}(f_1\otimes\cdots\otimes f_m)=\sum^m_{j=1}\sum_{\sigma\in \Pi_{m,j}}
(|I_1(\sigma)|-1)!\cdot\ldots\cdot (|I_j(\sigma)|-1)!
\prod^j_{i=1}\int \prod_{k\in I_i(\sigma)}f_k(x)\,\rho(\dd x).
\end{equation}
This formula can easily be proved by induction, using the recursion
\eqref{Rekursion mu[m+n]}.  
With this expression,
Corollary~\ref{c:cov3} agrees, after translating notation, with
\cite[Proposition~4.5]{FlintTorrisi23}. More general but less
explicit moment formulas can be found in the recent preprint
\cite{DelloShiavoQuattro23}.  
\end{remark}

\section{The Poincar\'e inequality}\label{sec:Poincare}

The Poincar\'e inequality for functions of the DF process was derived
in \cite{Stannat00} using the Poincar\'e inequality for the Dirichlet distribution 
(proved in \cite{Shimakura77}) and
a suitable approximation. We can derive here this inequality directly from
the chaos expansion \eqref{EqChaosZerlegung}.

\begin{theorem} \label{ThmPoincareUngl}
Let  $F\in\operatorname{dom}(\nabla)$. Then
\begin{equation}\label{GLPoincareUngl}
\BV (F(\zeta)) \leq \frac{1}{\theta}\, \BE \int (\nabla_x F)^2\,\zeta(\dd x) . 
\end{equation}
Equality holds if and only if
there exists $g \in L^2(\rho)$ such that $F(\zeta) = \zeta(g)$ $\BP$-a.s.
\end{theorem}
\begin{proof}
From~\eqref{e:iso0} we obtain that the left-hand side of~\eqref{GLPoincareUngl} equals
\begin{equation*}
	\BE ( F(\zeta) - \BE F(\zeta) )^2 
	=  \sum_{n=1}^\infty \frac{n!}{\Pn{\theta}{2n}} \rho^{[n]}(f_n^2) ,
\end{equation*}
where $f_n$, $n\in\N$, are the kernel functions of $F$.
By \eqref{e:iso}, the right-hand side of~\eqref{GLPoincareUngl} equals
\begin{equation*}
    \frac{1}{\theta} \sum_{n=1}^\infty \frac{nn!}{\Pn{\theta}{2n}}(\theta + n -1) 
\rho^{[n]}(f_n^2). 
\end{equation*}    
Since 
$1 \leq \theta^{-1} n(\theta + n-1)$, $n\in\N$, inequality 
\eqref{GLPoincareUngl} follows. The latter inequality becomes an
equality if and only if $n=1$.
Therefore, we have equality in~\eqref{GLPoincareUngl} if and only if 
\begin{equation*}
    \rho^{[n]}(f_n^2)=0,\quad n\ge 2.
\end{equation*}
By Theorem \ref{t:chaos}, this is in turn equivalent to $F\in \mathbf{F}_0 \oplus \mathbf{F}_1$.
\end{proof}

As in Section \ref{sec:FlemingViot} we consider now an $F\in\bS$ of the form
\[
F(\mu)=\varphi\big(\mu(g_1),\ldots,\mu(g_d)\big),
\]
where $g_1,\ldots,g_d\in \mathbf{F}_b(\BX)$ and $\varphi\in C^\infty(\mathbb R^d)$. Put
\[
Z_i:=\zeta(g_i),\qquad G_i(x):=g_i(x)-Z_i,\qquad i=1,\ldots,d,\,x\in\BX,
\]
and write $Z:=(Z_1,\ldots,Z_d)$,
and $G(x):=(G_1(x),\ldots,G_d(x))^{\mathsf T}$.
By Proposition \ref{l:nabla=nabla*} the first order gradient is given by
\begin{equation}\label{smoothfirstgradient}
\nabla_xF=\sum_{i=1}^d\partial_i\varphi(Z)G_i(x)=
\big(\nabla\varphi(Z)\big)^{\mathsf T}G(x).
\end{equation}
Therefore 
\begin{equation}\label{firstgradientenergy}
\int(\nabla_xF)^2\,\zeta(dx)
=\big(\nabla\varphi(Z)\big)^{\mathsf T}A_G\nabla\varphi(Z),
\end{equation}
where $A_G$ is the random covariance matrix
\begin{equation}\label{e:AG}
A_G:=\int G(x)G(x)^{\mathsf T}\,\zeta(dx).
\end{equation}
The general Poincar\'e inequality \eqref{GLPoincareUngl} takes the explicit form 
\begin{equation}\label{e:poincaredirichlet}
\BV (F(\zeta)) 
\leq \frac{1}{\theta}\, \BE \big(\nabla\varphi(Z)\big)^{\mathsf T}A_G\nabla\varphi(Z).
\end{equation}

\section{A reverse Poincar\'e inequality}\label{s:reversePoincare}

In this section we use the chaos expansion and a second order
gradient to derive a reverse Poincar\'e inequality. Such lower variance
bounds were introduced by~\cite{SchulteTrapp2024} in the
setting of Poisson functionals.

\subsection{A second order gradient}

Denote by $\operatorname{dom}(\nabla^2)$ the set of all
$F\in L^2(\zeta)$ with chaos expansion~\eqref{EqChaosZerlegung}
such that
\begin{equation}
 \sum_{n=2}^\infty n(n-1)(\theta+n-1)(\theta+n) \frac{n!}{\theta^{(2n)}}\rho^{[n]}(f_n^2)<\infty.
\end{equation}
It can be checked that $\operatorname{dom}(\nabla^2)\subset \operatorname{dom}(\nabla)$. 
For $F\in\operatorname{dom}(\nabla^2)$, define
$\nabla^2F\colon\Omega\times\BX^2\to\R$ by
\begin{align}\label{secondgradient}
 \nabla^2_{x,y}F(\omega):=\sum_{n=2}^\infty & n(n-1)\bigg(
 \int f_n(x,y,\bz_{n-2})\,\zeta^{n-2}(\omega,\dd\bz_{n-2})
 -\int f_n(x,\bz_{n-1})\,\zeta^{n-1}(\omega,\dd\bz_{n-1})
 \notag\\
 &-\int f_n(y,\bz_{n-1})\,\zeta^{n-1}(\omega,\dd\bz_{n-1})
 +\int f_n(\bz_n)\,\zeta^n(\omega,\dd\bz_n)\bigg).
\end{align}
The next lemma justifies this definition.
As in the case $k=1$ we write
$\nabla_{x,y}^2F$ for the random variable
$\omega\mapsto (\nabla^2_{x,y}F)(\omega)$.
The Campbell measure $C_{\zeta^2}$ of $\zeta^2$ is defined analogously
to \eqref{Campbell}.

\begin{lemma}\label{lemsecondgradient}
Suppose that $F\in\operatorname{dom}(\nabla^2)$. Then the series on the
right-hand side of~\eqref{secondgradient} converges in
$L^2(C_{\zeta^2})$. Moreover,
\begin{align}\label{secondgradientisometry}
 \BE\int \big(\nabla^2_{x,y}F\big)^2\,\zeta^2(d(x,y))
 = \sum_{n=2}^\infty n(n-1)(\theta+n-1)(\theta+n) \frac{n!}{\theta^{(2n)}}\rho^{[n]}(f_n^2).
\end{align}
\end{lemma}
\begin{proof}
For $n\ge2$ and $A\subset\{1,2\}$, let
\begin{equation*}
 I_{n,A}(x,y):=\int f_n((x,y)_A,\bz_{n-|A|})\,\zeta^{n-|A|}(\dd\bz_{n-|A|}),
     \end{equation*}
where $(x,y)_A$ denotes the variables among $x,y$ whose indices
belong to $A$. Thus, if
\[
 G_n(x,y):=
 I_{n,\{1,2\}}(x,y)-I_{n,\{1\}}(x,y)
 -I_{n,\{2\}}(x,y)+I_{n,\varnothing}(x,y),
\]
then
\[
 \nabla^2_{x,y}F=\sum_{n=2}^\infty n(n-1)G_n(x,y).
\]

Let $m,n\ge2$ and $A,B\subset\{1,2\}$, and put
$k:=|A\cap B|$. Using \eqref{nmoment}
and integrating out those of $x,y$ which occur in neither kernel function,
we obtain, after relabelling the variables,
\begin{align}\label{mixedI}
 &\BE\int I_{n,A}(x,y)I_{m,B}(x,y)\,\zeta^2(\dd(x,y)) \notag\\
 &\quad = \frac{1}{\theta^{(n+m-k)}}
 \int f_n(\by_n)f_m(x_1,\ldots,x_{m-k},y_1,\ldots,y_k)\, \rho^{[n+m-k]}(\dd(\bx_{m-k},\by_n)).
\end{align}
The orthogonality relations from Corollary \ref{KorIntegrationInIH_mAndIH_N}  imply
that the right-hand side vanishes if $m\ne n$. If $m=n$, it equals
\begin{align}\label{mixedIdiagonal}
 \frac{(n-k)!}{\theta^{(2n-k)}} \int f_n^2\,\dd\rho^{[n]}.
\end{align}

It follows in particular that
\begin{equation*}
 \BE\int G_n(x,y)G_m(x,y)\,\zeta^2(d(x,y))=0,\quad m\ne n.
\end{equation*}
For $m=n$, expanding the four terms defining $G_n$ and using
\eqref{mixedIdiagonal}, the contributions corresponding to
$k=0,1,2$ have respective signed multiplicities $1,-2,1$.
Consequently,
\begin{align*}
 \BE\int G_n(x,y)^2\,\zeta^2(\dd(x,y)) =\bigg( \frac{n!}{\theta^{(2n)}}
 -2\frac{(n-1)!}{\theta^{(2n-1)}}
 +\frac{(n-2)!}{\theta^{(2n-2)}}\bigg)\rho^{[n]}(f_n^2)
\end{align*}
Since
\begin{align*}
\frac{n!}{\theta^{(2n)}}-2\frac{(n-1)!}{\theta^{(2n-1)}}
 +\frac{(n-2)!}{\theta^{(2n-2)}}
 =\frac{(n-2)!(\theta+n-1)(\theta+n)}{\theta^{(2n)}},
\end{align*}
the assertion \eqref{secondgradientisometry} follows, once
we have proved the asserted $L^2$-convergence. To this end 
we define the partial sums 
\begin{equation*}
S_m:=\sum_{n=2}^mn(n-1)G_n,\quad m\ge 2.
\end{equation*}
the established orthogonality gives that the squared $L^2$-distance
(in $L^2(C_{\zeta^2})$) between $S_{m+m_0}$ and $S_{m}$, $m_0\in\N$, is given by
\begin{equation*}
\sum^{m+m_0}_{n=m+1}n(n-1)(\theta+n-1)(\theta+n)\frac{n!}{\theta^{(2n)}}\rho^{[n]}(f^2_n)
\end{equation*}
which tends to zero as $m\to\infty$.
\end{proof}

As for the first-order gradient, $\nabla^2F$ is pathwise centred w.r.t.\ $\zeta^2$, that is,
\begin{equation}
\int \nabla_{x,y}^2F\,\zeta^2(\dd(x,y))=0,
\quad \BP\text{-a.s.}
\end{equation}
Indeed, this follows by integrating the finite partial sums in
\eqref{secondgradient} with respect to $\zeta^2$ and arguing as
in the proof of \eqref{gradientmeanzero}.

\subsection{The reverse Poincar\'e inequality}

Under an assumption on $F\in\operatorname{dom}(\nabla^2)$
we can reverse the Poincar\'e inequality as follows.

\begin{theorem}\label{thmreversepoincare}
Let $F\in\operatorname{dom}(\nabla^2)$ and assume that, for some $C\ge0$,
\begin{align}\label{reversecondition}
 \BE\int\big(\nabla^2_{x,y}F\big)^2\,\zeta^2(\dd(x,y))
\le C\,\BE\int(\nabla_xF)^2\,\zeta(\dd x).
\end{align}
Then
\begin{align}\label{reversepoincare}
 \BV(F(\zeta))\ge \frac{1}{C+\theta}\,
 \BE\int(\nabla_xF)^2\,\zeta(\dd x).
\end{align}
\end{theorem}
\begin{proof}
Set
\begin{equation*}
 v_n:=\frac{n!}{\theta^{(2n)}}\rho^{[n]}(f_n^2),
 \qquad
 \lambda_n:=n(\theta+n-1),\quad n\ge1.
\end{equation*}
By \eqref{e:iso0},
\begin{equation}\label{variancevn}
 \BV(F(\zeta))=\sum_{n=1}^\infty v_n,
\end{equation}
while \eqref{e:iso}  gives
\begin{equation}\label{firstgradientvn}
 \BE\int(\nabla_xF)^2\,\zeta(\dd x)=\sum_{n=1}^\infty\lambda_n v_n.
\end{equation}
Since $\lambda_n-\theta=(n-1)(\theta+n)$,
Lemma~\ref{lemsecondgradient} yields
\begin{equation}\label{secondgradientvn}
 \BE\int\big(\nabla^2_{x,y}F\big)^2\,\zeta^2(\dd(x,y))
 =\sum_{n=1}^\infty \lambda_n(\lambda_n-\theta)v_n.
\end{equation}
Hence \eqref{reversecondition} implies
\begin{equation*}
 \sum_{n=1}^\infty\lambda_n^2v_n \le
 (C+\theta)\sum_{n=1}^\infty\lambda_nv_n.
\end{equation*}
On the other hand, the Cauchy--Schwarz inequality gives
\begin{equation}
 \left(\sum_{n=1}^\infty\lambda_nv_n\right)^2
 \le\left(\sum_{n=1}^\infty v_n\right)
 \left(\sum_{n=1}^\infty\lambda_n^2v_n\right).
\end{equation}
Combining the last two inequalities with
\eqref{variancevn} and \eqref{firstgradientvn} proves
\eqref{reversepoincare}.
\end{proof}

\begin{remark}\label{r:reversesharp}\rm
The reverse Poincar\'e inequality is sharp on each chaos. More precisely, suppose that
$F$ belongs to the $n$-th chaos for some $n\geq 1$, that is, in
the chaos expansion \eqref{EqChaosZerlegung} of $F$ only $f_n$ is non-zero. Recalling that
$\lambda_n=n(\theta+n-1)$,
we obtain from \eqref{variancevn} and \eqref{firstgradientvn}
\begin{align*}
\BE\int (\nabla_xF)^2\,\zeta(\dd x) =\lambda_n\BV(F(\zeta)),
\end{align*}
while \eqref{secondgradientvn} yields
\begin{align*}
\BE\int(\nabla^2_{x,y}F)^2\,\zeta^2(\dd(x,y))
=\lambda_n(\lambda_n-\theta)\BV(F(\zeta)).
\end{align*}
Hence assumption \eqref{reversecondition} holds with
$C=\lambda_n-\theta=(n-1)(\theta+n)$.
For this choice of $C$, the lower bound \eqref{reversepoincare} becomes
\begin{align*}
\BV(F)\ge\frac{\lambda_n}{\theta+C}\BV(F(\zeta))=\BV(F(\zeta)),
\end{align*}
and is therefore an equality.
\end{remark}
 

In order to apply the inequality \eqref{reversepoincare} to smooth functionals
we need the following second order version of Proposition \ref{l:nabla=nabla*}.
For a sufficiently smooth functional $F\colon\bM_1\to\R$, define
\begin{equation*}
D_xF(\mu):=\left.\frac{\dd}{\dd t}F\big((1-t)\mu+t\delta_x\big)\right|_{t=0},\quad x\in\BX.
\end{equation*}
If $F\in\bS$, then Proposition \ref{l:nabla=nabla*} (or \eqref{firstgradientenergy})
shows that $\nabla_xF=D_xF(\zeta)$.

\begin{proposition}\label{l:nabla2=D2} Let $F\in\bS$. Then $F\in\operatorname{dom}(\nabla^2)$ and
\begin{equation}\label{secondgradientidentity}
\nabla^2_{x,y}F(\omega)=D_yD_xF(\zeta(\omega))+D_yF(\zeta(\omega)),\quad 
C_{\zeta^2}\text{-a.e.\ $(\omega,x,y)$}.
\end{equation}
\end{proposition}
\begin{proof} We give here an informal proof. A rigorous proof should follow the
one of Proposition \ref{l:nabla=nabla*}. Throughout we ignore $C_{\zeta^2}$-null sets.
Suppose that $F$ is given by \eqref{EqChaosZerlegung} and consider for $n\ge 2$
\begin{equation*}
F_n(\mu):=\int f_n(z_1,\ldots,z_n)\,\mu^n(\dd(z_1,\ldots,z_n)),\quad \mu\in\bM_1.
\end{equation*}
Write
\begin{equation*}
A_n(\mu):=\int f_n\,\dd\mu^n,\qquad
A_{n,x}(\mu):=\int f_n(x,z_2,\ldots,z_n)\,\mu^{n-1}(\dd(z_2,\ldots,z_n)),
\end{equation*}
and define $A_{n,x,y}(\mu)$ analogously. Then
\begin{equation*}
D_xF_n=n(A_{n,x}-A_n),
\end{equation*}
while
\begin{equation}
D_yA_{n,x}=(n-1)(A_{n,x,y}-A_{n,x}),
\qquad
D_yA_n=n(A_{n,y}-A_n).
\end{equation}
Consequently,
\begin{equation}
\begin{split}
D_yD_xF_n+D_yF_n
&=n(n-1)\big(A_{n,x,y}-A_{n,x}-A_{n,y}+A_n\big),
\end{split}
\end{equation}
which is precisely the contribution of $F_n(\zeta)$ to
$\nabla^2_{x,y}F$. This proves \eqref{secondgradientidentity},
at least informally.
\end{proof}


Using \eqref{secondgradientidentity} we can now discuss the special
case of a smooth function $F\in\bS$. We use the notation introduced before
\eqref{e:poincaredirichlet} and
assert that the second order gradient of $F$ is given by
\begin{equation}\label{smoothsecondgradient}
\nabla^2_{x,y}F=\sum_{i,j=1}^d\partial_{ij}\varphi(Z)G_i(x)G_j(y)
=G(x)^{\mathsf T}D^2\varphi(Z)G(y),
\end{equation}
where $D^2\varphi(x)$ denotes the Hessian of $\varphi$ at $x\in\R^d$.
To see this we first note that
\begin{equation*}
D_yZ_i=G_i(y),\qquad D_yG_i(x)=-G_i(y),
\end{equation*}
with a slight abuse of notation.
Hence \eqref{smoothfirstgradient} yields
\begin{equation}
\begin{split}
D_yD_xF&=\sum_{i,j=1}^d\partial_{ij}\varphi(Z)G_i(x)G_j(y)
-\sum_{i=1}^d\partial_i\varphi(Z)G_i(y)\\
&=\sum_{i,j=1}^d\partial_{ij}\varphi(Z)G_i(x)G_j(y)-D_yF.
\end{split}
\end{equation}
Together with \eqref{secondgradientidentity}, this proves
\eqref{smoothsecondgradient}.

We next rewrite the two quantities occurring in the reverse
Poincar\'e inequality. Using the symmetry of the Hessian we have
\begin{align*}
(G(x)^{\mathsf T}D^2\varphi(Z)G(y))^2
&=G(y)^{\mathsf T}D^2\varphi(Z)G(x)G(x)^{\mathsf T}D^2\varphi(Z)G(y) \\
&=\operatorname{tr}(D^2\varphi(Z)G(x)G(x)^{\mathsf T}D^2\varphi(Z)G(y)G(y)^{\mathsf T}),
\end{align*}
where we have used the cyclic property of the trace operator $\operatorname{tr}(\cdot)$.
Recalling the definition \eqref{e:AG} of the random $d\times d$ matrix $A_G$,
we therefore obtain from Fubini's theorem
\begin{align*}
\iint (G(x)^{\mathsf T}D^2\varphi(Z)G(y))^2\,\zeta(\dd x)\,\zeta(\dd y)
=\operatorname{tr}\big(D^2\varphi(Z)A_GD^2\varphi(Z)A_G)\big).
\end{align*}
Hence we get from \eqref{smoothsecondgradient},
\begin{equation}\label{secondgradientenergy}
\iint(\nabla^2_{x,y}F)^2\,\zeta(\dd x)\,\zeta(\dd y)
=\operatorname{tr}\big(\big(D^2\varphi(Z)A_G\big)^2\big).
\end{equation}
Taking into account \eqref{e:poincaredirichlet}, the assumption 
in the reverse Poincar\'e inequality takes the form
\begin{equation}\label{smoothreversecondition}
\BE\big[\big(\operatorname{tr}\big(\big(D^2\varphi(Z)A_G\big)^2\big)\big]
\leq C\,\BE\big[\big(\nabla\varphi(Z)\big)^{\mathsf T}
A_G\nabla\varphi(Z)\big].
\end{equation}
Consequently, the reverse Poincar\'e inequality yields the lower variance bound
\begin{equation}\label{smoothreversepoincare}
\BV (F(\zeta))
\ge \frac{1}{C+\theta}\,\BE\big[\big(\nabla\varphi(Z)\big)^{\mathsf T}A_G
\nabla\varphi(Z)\big].
\end{equation}

Let us specialize this to an ordinary Dirichlet distribution. Take
$\mathbb X=\{1,\ldots,m\}$, identify $\zeta=(\zeta_1,\ldots,\zeta_m)$
and suppose that $\zeta$ has distribution $\operatorname{Dir}(\alpha_1,\ldots,\alpha_m)$.
Choosing $g_i=\I_{\{i\}}$, we have $Z=\zeta$ and
$G(i)=e_i-\zeta$, where $e_i$ is the $i$-th unit vector in $\R^m$.
Therefore
\begin{equation}\label{dirichletmatrix}
A_G=\sum_{i=1}^m\zeta_i(e_i-\zeta)(e_i-\zeta)^{\mathsf T}\nonumber\\
=\operatorname{diag}(\zeta)-\zeta\zeta^{\mathsf T},
\end{equation}
where we have used $\zeta_1+\cdots+\zeta_m=1$. Writing
\begin{equation*}
A(z):=\operatorname{diag}(z)-zz^{\mathsf T},
\end{equation*}
condition \eqref{smoothreversecondition} becomes
\begin{equation}\label{dirichletreversecondition}
\BE \big[\operatorname{tr}((D^2\varphi(\zeta)A(\zeta))^2)\big]
\leq C\,\BE\big[(\nabla\varphi(\zeta))^{\mathsf T}A(\zeta) \nabla\varphi(\zeta)\big],
\end{equation}
and the reverse Poincar\'e inequality gives
\begin{equation}\label{dirichletreversepoincare}
\BV \varphi(\zeta)
\ge\frac{1}{C+\theta}\BE[(\nabla\varphi(\zeta))^{\mathsf T}A(\zeta)\nabla\varphi (\zeta)].
\end{equation}

\subsection{The beta distribution}

In this subsection we suppose that $X$ has the beta distribution
$\operatorname{Beta}(\alpha,\beta)$,
where $\alpha,\beta>0$.
We can identify $X$ with the Dirichlet random vector
$\zeta:=(X,1-X)$. For a smooth function $f$, consider $F(\zeta)=f(X)$.
We apply \eqref{dirichletreversepoincare} with the function
$\varphi(x,y):=f(x)$. Then $\nabla\varphi(Z)=(f'(X),0)$ and
the only non-zero entry of $D^2\varphi(Z)$ is $f''(X)$ in 
the top-left corner. The top-left corner entry of $A_G$ equals  $X(1-X)$.
Therefore we obtain from \eqref{firstgradientenergy}
\begin{equation}\label{betafirstenergy}
\int(\nabla_xF)^2\,\zeta(\dd x)=X(1-X)f'(X)^2,
\end{equation}
whereas \eqref{secondgradientenergy} yields
\begin{equation}\label{betasecondenergy}
\iint(\nabla^2_{x,y}F)^2\,\zeta(\dd x)\zeta(\dd y)
=X^2(1-X)^2f''(X)^2.
\end{equation}
It is convenient to introduce the first- and second-order energies
\begin{equation*}
 E_1(f):=\BE\big[X(1-X)f'(X)^2\big],\qquad
 E_2(f):=\BE\big[X^2(1-X)^2f''(X)^2\big].                                
\end{equation*}
Then the assumption \eqref{smoothreversecondition} can be written as
\begin{equation}\label{betareverse}
E_2(f) \le C E_1(f),
\end{equation}
and the resulting reverse Poincar\'e inequality is
\begin{equation}\label{betareversepoincare}
\BV (f(X)) \ge \frac{1}{C+\theta}E_1(f),
\end{equation}
where $\theta=\alpha+\beta$.
Recall that the Poincar\'e inequality for the beta distribution is
\begin{equation*}
\BV(f(X))\le \frac1\theta E_1(f).
\end{equation*}
Thus \eqref{betareversepoincare} is a reverse Poincar\'e inequality on classes
of functions for which the second-order energy can be controlled by
the first-order energy.

\begin{example}\rm
A simple sufficient condition for assumption  \eqref{betareverse} is
\begin{equation}\label{betapointwise}
 x(1-x)f''(x)^2\le C f'(x)^2,\quad x\in(0,1).
\end{equation}
In particular, if $f'$ does not vanish and
\begin{equation*}
\sup_{0<x<1}\sqrt{x(1-x)} \left|\frac{f''(x)}{f'(x)}\right|\le L,
\end{equation*}
then \eqref{betareverse} holds with $C=L^2$.
\end{example}

\begin{example}\rm
Given $t\in\R$ let
\begin{equation*}
f(x)=e^{tx},\quad x\in(0,1).
\end{equation*}
Since $f''=t f'$ and $x(1-x)\le1/4$, condition
\eqref{betapointwise} holds with $C=t^2/4$. Hence
\begin{equation*}
\BV(e^{tX}) \ge \frac{1}{\theta+t^2/4}
 \BE\big[X(1-X)t^2e^{2tX}\big].
\end{equation*}
More generally, if
\begin{equation*}
\left|\frac{f''(x)}{f'(x)}\right|\le L, \quad 0<x<1,
\end{equation*}
then one may take $C=L^2/4$.
\end{example}

\begin{example}\rm
Assume that $f(x)=x^p$, where $p\ge 2$ is an integer. Then
\begin{align*}
E_1(f)&=p^2\BE[X^{2p-1}(1-X)]=p^2\beta\frac{\alpha^{(2p-1)}}{\theta^{(2p)}},\\
E_2(f)&=p^2(p-1)^2\BE[X^{2p-2}(1-X)^2]\\
&=p^2(p-1)^2\beta(\beta+1)\frac{\alpha^{(2p-2)}}{\theta^{(2p)}}.
\end{align*}
Consequently,
\begin{equation*}
 \frac{E_2(f)}{E_1(f)} =
 (p-1)^2\frac{\beta+1}{\alpha+2p-2} =:C_p,
\end{equation*}
and hence
\begin{equation}\label{betapowerbound}
\BV(X^p) \ge \frac{p^2\beta\alpha^{(2p-1)}}{(\theta+C_p)\theta^{(2p)}}.
\end{equation}
This bound also has the correct asymptotic order as $p\to\infty$.
Indeed,
\begin{equation*}
\BV(X^p) =
 \frac{\alpha^{(2p)}}{\theta^{(2p)}}-\left(\frac{\alpha^{(p)}}{\theta^{(p)}}\right)^2
 \sim\frac{\Gamma(\theta)}{\Gamma(\alpha)}(2p)^{-\beta},
\end{equation*}
where we have used that $a^{(r)}=\Gamma(a+r)/\Gamma(a)$
along with  the $\frac{\Gamma(p+a)}{\Gamma(p+b)}\sim p^{a-b}$, for $a,b\in\R$.
Similarly,
\begin{equation*}
 E_1(f) \sim\frac{\beta}{2}\frac{\Gamma(\theta)}{\Gamma(\alpha)} 2^{-\beta}p^{1-\beta},
 \qquad
 C_p\sim\frac{\beta+1}{2}p.
\end{equation*}
It follows that
\begin{equation*}
\lim_{p\to\infty} \frac{E_1(f)}{(\theta+C_p)\BV(X^p)}=\frac{\beta}{\beta+1}.
\end{equation*}
Thus the lower bound \eqref{betapowerbound} has the correct
asymptotic order. 
\end{example}

There is an extensive literature on upper and lower variance bounds
for functions of random variables.  In the present context, particularly
relevant are the Poincar\'e-type inequalities in \cite{Johnson93} and
their extensions in \cite{Afendras13}. 
For beta distributions these results are based
on expansions in Jacobi polynomials and yield bounds involving
weighted derivatives of the function.

The second-order lower bound in \cite{Johnson93} (Theorem 2 with $m=0$) takes the form
\begin{equation*}
 \BV(f(X))\ge \frac1\theta E_1(f)-\frac{1}{2\theta(\theta+1)}E_2(f).
\end{equation*}
Consequently, under assumption $E_2(f)\le C E_1(f)$,  
\begin{equation}\label{johnsonbeta}
\BV(f(X)) \geq
 \frac1\theta\left(1-\frac{C}{2(\theta+1)}\right)E_1(f).
\end{equation}
A direct comparison with \eqref{betareversepoincare}  gives
\begin{equation*}
 \frac1\theta\left(1-\frac{C}{2(\theta+1)}\right)-\frac1{\theta+C}
=\frac{C(\theta+2-C)}{2\theta(\theta+1)(\theta+C)}.
\end{equation*}
Thus the bound \eqref{johnsonbeta} is stronger than \eqref{betareverse} for $0<C<\theta+2$, 
while \eqref{betareverse} is stronger for $C>\theta+2$; the two bounds agree for
$C=0$ and $C=\theta+2$. Moreover, the coefficient in \eqref{johnsonbeta} becomes
non-positive for $C\geq 2(\theta+1)$, whereas the coefficient in \eqref{betareverse}
remains positive for every finite $C\geq0$.
In view of Remark~\ref{r:reversesharp}, the threshold $\theta+2$ has a simple
interpretation. For a functional in the $n$-th chaos, equality in
the reverse Poincar\'e inequality holds with
\begin{equation*}
C=(n-1)(\theta+n).
\end{equation*}
For $n\geq2$ this is at least $\theta+2$. Thus the range
$0<C<\theta+2$, in which \eqref{johnsonbeta}  is stronger, contains no
such value. For $n=1$ we have $C=0$, and the Poincar\'e inequality
and its reverse both hold with equality.

The paper \cite{Afendras13} contains a family of refinements of
Poincar\'e type bounds that exploit the explicit spectral structure
available in the beta case and can yield sharper bounds. In contrast,
\eqref{betareversepoincare} is a special case of a reverse Poincar\'e
inequality for general functionals of a Dirichlet process on an
arbitrary measurable space. Its proof requires neither an explicit
system of orthogonal polynomials nor estimates of individual chaos
coefficients, but follows directly from the isometry properties of
the first- and second-order gradients and the Cauchy--Schwarz
inequality.

\appendix
\section{Orthogonality relations}


Let $\mu$ be a $\sigma$-finite measure on $\BX$.
Let $m,n\in\N$ and $B \in \cX^{\otimes (n+m)}$. A key ingredient of our proof is the recursion
\begin{align} \label{Rekursion mu[m+n]}
        \Pnm{\mu}{n+m} (B) = \iint \I_B (x_1, \ldots, x_{n+m})& 
\, \Pnm{(\mu + \delta_{x_1} + \cdots + \delta_{x_{m}})}{n} (\dd (x_{m+1}, \ldots, x_{m+n})) \\ \notag
&      \times  \Pnm{\mu}{m}(\dd( x_1, \ldots, x_m)).
\end{align}
This can be proved by induction.

\begin{lemma} \label{LemAllgIsoFormel1Funktion}
Let $m,n, r\in\N_0$ so that $r< n$. Further, let $f\colon \BX^{m+r+n} \to \R$ be such that $f\in L^1(\Pnm{\rho}{m+r+n})$ and 
$f(x_1,\ldots, x_{m+r}, \cdot) \in \bH_n$ for $\Pnm{\rho}{m+r}$-almost all $(x_1, \ldots, x_{m+r}) \in \BX^{m+r}$. Then 
\begin{align*}
\int f&(x_1, \ldots, x_m, y_1, \ldots, y_r, y_1, \ldots, y_n) \, \Pnm{\rho}{m+n}(\dd (x_1, \ldots, x_m, y_1, \ldots, y_n)) \\
 &=   \sideset{}{^{\ne}}\sum_{i_1,\dots,i_{n-r}\in [m]}\int f(x_1, \ldots, x_m, y_1,\ldots, y_r,y_1,\ldots, y_r, x_{i_1}, \ldots, x_{i_{n-r}}) \\ 
&\qquad\qquad\qquad\qquad\qquad \times\,\Pnm{\rho}{m+r}(\dd (x_1, \ldots, x_m, y_1, \ldots, y_r)) 
\end{align*}
whenever $m \ge n-r$. 
If $m< n-r$, the integral on the left-hand side vanishes. 
\end{lemma}
\begin{proof}
In the case $m\ge n-r$
the recursion~\eqref{Rekursion mu[m+n]} and the defining property of $\bH_n$ yield
\begin{align*}
     &\int f(\bfm{x}_m, \bfm{y}_r, \bfm{y}_n) \,\Pnm{\rho}{m+n}(\dd (\bfm{x}_m,\bfm{y}_n)) \\
     &=\iint f(\bfm{x}_m, \bfm{y}_r, \bfm{y}_n) \, (\rho + \delta_{\bfm{x}_m} + \delta_{\bfm{y}_{n-1}})(\dd y_n) \, 
\rho^{[m+n-1]}(\dd (\bfm{y}_{n-1},\bfm{x}_m))\\
     &=\iint \sum_{i_1=1}^m f(\bfm{x}_m, \bfm{y}_r, \bfm{y}_{n-1}, x_{i_1}) 
\,\Pnm{(\rho+ \delta_{\bfm{y}_{n-1}})}{m}(\dd \bfm{x}_m) \,\Pnm{\rho}{n-1}(\dd  \bfm{y}_{n-1}). 
\end{align*}
Further iterations give
\begin{align*}
\int &f(\bfm{x}_m, \bfm{y}_r, \bfm{y}_n)\, \Pnm{\rho}{m+n}(\dd (\bfm{x}_m,\bfm{y}_n)) \\
    &= \int \sideset{}{^{\ne}}\sum_{i_1,\dots,i_{n-r}\in [m]} f(\bfm{x}_m, \bfm{y}_r, \bfm{y}_{r}, x_{i_1},\ldots, x_{i_{n-r}}) 
\, \Pnm{\rho}{m+r}(\dd (\bfm{x}_m, \bfm{y}_{r})),
\end{align*}
as asserted.
If $m<n-r$, then the above argument gives
\begin{align*}
\int f&(\bfm{x}_m, \bfm{y}_r, \bfm{y}_n) \, \Pnm{\rho}{m+n}(\dd (\bfm{x}_m,\bfm{y}_n)) \\ 
     &= m!
    \int f(\bfm{x}_m,  \bfm{y}_{r}, \bfm{y}_{n-m}, \bfm{x}_m) 
\, \Pnm{\rho}{n}(\dd (\bfm{x}_m, \bfm{y}_{n-m})).
\end{align*}
This comes to zero after applying the recursion \eqref{Rekursion mu[m+n]} one more time.
\end{proof}

\begin{corollary}\label{KorIntegrationInIH_mAndIH_N}
Let $n,l, k \in \N_0$ with $n \ge l \ge k$. Then $g \in \bH_n$ and $h \in L^2(\Pnm{\rho}{l})$ satisfy
\begin{align*}
    \int & g(y_1, \ldots, y_{n}) h(x_1, \ldots, x_{l-k}, y_1, \ldots, y_k)  \, \Pnm{\rho}{l+n-k}(\dd  (x_1, \ldots, x_{l-k}, y_1,\ldots,y_{n})) \\
    &= \I\{l=n\} (n-k)! \int g(z) h(z) \, \Pnm{\rho}{n}(\dd  z).
\end{align*}
\end{corollary}

\begin{proof}
If $k=n$, then $n=l=k$ and there is nothing to prove.
Now assume $k<n$ and define
$f\colon \BX^{l+n} \to \R$ by $f(\bfm{x}_{l}, \bfm{y}_n) \coloneqq h(\bfm{x}_{l}) g(\bfm{y}_n)$.
Since $g$ and $h$ are square-integrable, the Cauchy--Schwarz inequality shows that $f\in L^1(\Pnm{\rho}{l+n})$.
Applying Lemma~\ref{LemAllgIsoFormel1Funktion} with $r=k$ and $m=l-k$, 
we obtain in the case $l-k \ge n - k$ (i.e.\ $l\ge n$) that
\begin{align*}
    \int & h(\bfm{x}_{l-k}, \bfm{y}_{k})g(\bfm{y}_n) \, \Pnm{\rho}{l-k+n}(\dd (\bfm{x}_{l-k}, \bfm{y}_n)) \\
&=  \sideset{}{^{\ne}}\sum_{i_1,\dots,i_{n-k}\in [l-k]}  \int h(\bfm{x}_{l-k}, \bfm{y}_{k}) g(\bfm{y}_{k}, x_{i_1}, \ldots, x_{i_{n-k}}) 
\,\Pnm{\rho}{l}(\dd (\bfm{x}_{l-k}, \bfm{y}_{k})).  
\end{align*}
Otherwise the integral vanishes. 
\end{proof}


\begin{corollary} \label{c:A3}
Let $m,n\in\N_0$ and $f\in L^1(\Pnm{\rho}{m+n+1})$ be such that $f(x_1, \ldots, x_{m+1}, \cdot) \in \bH_n$  for $\Pnm{\rho}{m+1}$-almost all 
$(x_1, \ldots, x_{m+1}) \in \BX^{m+1}$ and $f(\cdot, x_{m+1}, \ldots, x_{m+n+1})\in \bH_m$  for $\Pnm{\rho}{n+1}$-almost all 
$(x_{m+1}, \ldots, x_{m+n+1}) \in \BX^{n+1}$. Then 
\begin{align*}
    &\int f(x_1, \ldots, x_{m+n+1}) \, \Pnm{\rho}{m+n+1}(\dd (x_1,\ldots,x_{m+n+1})) \\
    &= \I\{m=n+1\} m! \int f(x_1,\ldots,x_m, x_m,x_1,\ldots,x_{m-1})\, \Pnm{\rho}{m}(\dd (x_1,\ldots,x_m))  \\    
    &\quad \begin{multlined}[0.9\textwidth]
        + \I\{m=n\} m! \left( m\int f(x_1, \ldots, x_m, x_m,x_1, \ldots, x_m) \,\Pnm{\rho}{m}(\dd (x_1, \ldots, x_m)) \right. \\
        \left. + \int f(x_1,\ldots, x_{m+1}, x_1,\ldots,x_m)\, \Pnm{\rho}{m+1}(\dd(x_1,\ldots,x_{m+1})) \right)
    \end{multlined}\\ 
&\quad + \I\{m=n-1\} (m+1)! \int f(x_1,\ldots,x_{m+1}, x_1,\ldots,x_{m+1})\, \Pnm{\rho}{m+1}(\dd (x_1,\ldots,x_{m+1})). 
\end{align*}
\end{corollary}
\begin{proof}
Applying Lemma~\ref{LemAllgIsoFormel1Funktion} with $r=0$ and $m+1$ in place of $m$ yields
\begin{equation} \label{Gl 1}
    \int f(\bfm{x}_{m+n+1}) \, \Pnm{\rho}{m+n+1}(\dd \bfm{x}_{m+n+1}) 
    = \sideset{}{^{\ne}}\sum_{i_1,\dots,i_{n}\in [m+1]} \int  f(\bfm{x}_{m+1}, x_{i_1},\ldots, x_{i_{n}}) \, \Pnm{\rho}{m+1}(\dd \bfm{x}_{m+1})
\end{equation}
if $m+1\ge n$. Otherwise, the integral vanishes.
 Using the symmetry of $f$ in its last arguments yields the first assertion in the case $m+1=n$.
Suppose that $m+1\ge n$ and let $i_1, \ldots, i_n \in  [m+1]$ be pairwise distinct. 
If there exists $j \in \{1, \ldots, m\} \setminus \{i_1,\ldots, i_{n}\}$, the defining properties of $\bH_m$ give
\begin{align*}
\int f&(\bfm{x}_{m+1}, x_{i_1},\ldots, x_{i_{n}}) \, \Pnm{\rho}{m+1}(\dd \bfm{x}_{m+1}) \\
    &=  \iint  f(\bfm{x}_{m+1}, x_{i_1},\ldots, x_{i_{n}}) \, (\rho + \delta_{\bfm{x}_{m+1}^{(j)}})(\dd x_j)\, \Pnm{\rho}{m}(\dd \bfm{x}_{m+1}^{(j)}) \\
    &= \int f(\bfm{x}_{m+1}^{(j)}, x_{m+1}, x_{i_1},\ldots, x_{i_{n}}) \, \Pnm{\rho}{m}(\dd \bfm{x}_{m+1}^{(j)}),
\end{align*}
where $\bfm{x}_{m+1}^{(j)}$ arises from  $\bfm{x}_{m+1}$ by dropping $x_j$.
If in addition there exists
$k\in \{1, \ldots, m\} \setminus \{j, i_1,\ldots, i_{n}\}$,
integrating with respect to
$(\rho + \delta_{\bfm{x}_{m+1}^{(j,k)}})(\dd x_k) \Pnm{\rho}{m-1}(\dd
\bfm{x}_{m+1}^{(j,k)}) $ shows that the integral vanishes.  This
establishes the assertion in the case $m>n+1$.  In the case $m=n+1$
the only terms remaining in~\eqref{Gl 1} are those corresponding to
sequences $(i_1,\ldots, i_{m-1})\in [m]^{m-1}$ with pairwise distinct entries. 
The result in this case is thus a consequence of the symmetry of $f$ in its last $n$
arguments.   The case $m=n$ does also follow  
from the symmetry of $f$ by distinguishing the cases where 
$\{i_1,\ldots,i_m\}$ contains $m+1$ or where it does not.
The remaining cases follow by swapping the roles of $m$ and $n$.  
\end{proof}

\begin{corollary}\label{c:A4}
Let the assumptions of Corollary \ref{c:A3} be satisfied. 
Then
\begin{align*}
\int f&(x_1, \ldots, x_{m+1}, x_{m+1}, \ldots, x_{m+n}) \, \Pnm{\rho}{m+n}(\dd (x_1,\ldots,x_{m+n})) \\
&= \I\{m=n-1\} m! \int f(x_1,\ldots,x_{m+1}, x_1,\ldots,x_{m+1})\, \Pnm{\rho}{m+1}(\dd (x_1,\ldots,x_{m+1}))  \\
&\quad + \I\{m=n\} m! \int f(x_1, \ldots, x_m, x_m,x_1, \ldots, x_m) \,\Pnm{\rho}{m}(\dd (x_1, \ldots, x_m)). 
\end{align*}
\end{corollary} 
\begin{proof}
This time we apply Lemma~\ref{LemAllgIsoFormel1Funktion} with $r=1$. This shows 
that the integral on the left-hand side vanishes unless $m\ge n-1$, in which case it equals
\begin{align*} 
    &\int f(\bfm{x}_{m}, y_1, \bfm{y}_n) \, \Pnm{\rho}{m+n}(\dd (\bfm{x}_{m}, \bfm{y}_n)) \\
    &= \sideset{}{^{\ne}}\sum_{i_1,\dots,i_{n-1}\in [m]} \int  f(\bfm{x}_{m}, y_1, y_1, x_{i_1},\ldots, x_{i_{n-1}}) \, \Pnm{\rho}{m+1}(\dd (\bfm{x}_{m},y_1)).
\end{align*}
If $m=n-1$, we obtain the assertion from the symmetry of $f$ in its
last $n$ arguments.  If $m=n$, we can again use the recursive structure 
of $\rho^{[m+1]}$ and the definition of $\bH_m$ to get the result.
Assume finally, that $m>n$. Then there exist for each
$(i_1,\ldots, i_{n-1})\in [m]^{n-1}$ with pairwise distinct entries
distinct $j,k \in \{1, \ldots, m\}\setminus \{i_1, \ldots, i_{n-1}\}$ and thus
the integral comes to zero.
\end{proof}

\begin{corollary}\label{c:A6} 
Let $m,n\in\N$ and $f\in \bH_n$, $h\in L^2(\Pnm{\rho}{m+1})$ such that $h(x,\cdot)\in\bH_m$ for all $x\in\BX$. Then
\begin{align*}
    &\iint h(x,x_1, \ldots, x_m)f(x_{m+1}, \ldots, x_{m+n})  \, (\rho + \delta_{x_1} + \cdots + \delta_{x_m})(\dd x) \, \Pnm{\rho}{m+n}(\dd (x_1, \ldots, x_{m+n})) \\
    &= \I\{m=n\}m! \int h(x_{m+1}, x_1, \ldots, x_m) f(x_1, \ldots, x_m) \, \Pnm{\rho}{m+1}(\dd (x_1, \ldots, x_{m+1})) \\
    &\quad + \I\{m=n+1\} m! \int h(x_m, x_1, \ldots, x_m) f( x_1, \ldots, x_{m-1}) \, \Pnm{\rho}{m}(\dd (x_1, \ldots, x_m)).
\end{align*}
\end{corollary}

\begin{proof}
We apply Lemma~\ref{LemAllgIsoFormel1Funktion} with $r=0$ and the function
$\widetilde f\colon \BX^{m+n}\to\R$ given by
$\widetilde f(\bfm{x}_{m}, \bfm{y}_{n}) \coloneqq \int
h(t,\bfm{x}_m) \, (\rho + \delta_{\bfm{x}_{m}})(\dd  t)f(\bfm{y}_{n})$.
This shows that the integral on the left-hand side of the asserted identity vanishes unless $m\ge n$ in which case 
\begin{align*} 
 \iint  h&(x,\bfm{x}_{m})f(\bfm{y}_{n}) \, (\rho + \delta_{\bfm{x}_{m}})(\dd x) \, \Pnm{\rho}{m+n}(\dd (\bfm{x}_{m},\bfm{y}_{n})) \\
    &= \sideset{}{^{\ne}}\sum_{i_1, \ldots,i_{n}\in [m]} \iint h(x,\bfm{x}_{m})f( x_{i_1}, \ldots, x_{i_{n}}) \, (\rho + \delta_{\bfm{x}_{m}})(\dd x)  
\, \Pnm{\rho}{m}(\dd \bfm{x}_{m})\\
&= \sideset{}{^{\ne}}\sum_{i_1, \ldots,i_{n}\in [m]} \int h(x,\bfm{x}_{m})f( x_{i_1}, \ldots, x_{i_{n}}) \,
\, \Pnm{\rho}{m+1}(\dd (x,\bfm{x}_{m})),
\end{align*}
where we have used the recursion~\eqref{Rekursion mu[m+n]}. 
If $m=n$, we can use the symmetry  of $f$ to obtain the assertion.
If $m=n+1$ we can use the recursion \eqref{Rekursion mu[m+n]} and the definition of
$\bH_m$ to obtain the result. If $m\ge n+2$, we apply the recursion twice to obtain that the integral comes to $0$. 
\end{proof}

\section{Integral formulas}

In this subsection we deal with the behaviour of the mapping $\mu\mapsto \mu^{[m]}$
under the addition of Dirac measures. We start with introducing some 
notation.  Let $m, r\in\N$ with $m\ge r$ and
$i_1, \ldots, i_r \in  [m]$ pairwise distinct.   
Suppose
$1 \le j_1 < j_2 < \cdots < j_{m-r}\le m$ are such that
$\{j_1, \ldots, j_{m-r}\} =  [m] \setminus \{i_1, \ldots, i_r\}$. Given a function $f \colon \BX^m \to \R$, 
we define $f_{i_1, \ldots, i_r} \colon \BX^m \to \R$ by
\begin{equation} \label{eq Def f_i}
    f_{i_1, \ldots, i_r} (x_1, \ldots, x_m) \coloneqq f(\widetilde{x}_1, \ldots, \widetilde{x}_m), 
\end{equation}
where
\begin{align*}
    \widetilde{x}_{i_1} = x_1, \, \ldots, \, \widetilde{x}_{i_r} = x_r, \quad \widetilde{x}_{j_1} = x_{r+1}, \,\ldots, \, \widetilde{x}_{j_{m-r}} = x_m.
\end{align*}
That is, in order to compute $f_{i_1, \ldots, i_r} (x_1, \ldots, x_m)$
for $(x_1, \ldots, x_m)\in\BX^m$, the function $f$ is evaluated at the
point whose $i_k$-th coordinate is $x_k$ for $k \in [r]$ and whose
remaining coordinates are filled with $x_{r+1}, \ldots, x_m$.
Finally, for $k\in \N$ and $x_1, \ldots, x_k, z_1, \ldots, z_{m-r} \in \BX$ let
\begin{equation} \label{eq Notation f_(i_1, ..., i_r)^k 1}
    f_{i_1, \ldots, i_r}^k (x_1, \ldots, x_k, z_1, \ldots, z_{m-r}) 
\coloneqq \sum_{1\le j_1 \le \ldots \le j_r \le k} f_{i_1, \ldots, i_r} (x_{j_1}, \ldots, x_{j_r}, z_1, \ldots, z_{m-r})
\end{equation}
if $m>r$. If $m=r$, define
\begin{equation} \label{eq Notation f_(i_1, ..., i_r)^k 2}
    f_{i_1, \ldots, i_r}^k (x_1, \ldots, x_k) \coloneqq \sum_{1\le j_1 \le \ldots \le j_r \le k} f_{i_1, \ldots, i_r} (x_{j_1}, \ldots, x_{j_r}).
\end{equation}

\begin{Lemma} \label{LemRekursionrho+delta_x} Let $\mu$ be a
$\sigma$-finite measure on $\BX$ and $x \in \BX$. Then we have for each
$m\in\N$ and each measurable $g\colon \BX^m \rightarrow [0,\infty]$ 
that
\begin{equation}\label{e:702}
\Pnm{(\mu + \delta_x)}{m}(g) = 
\Pnm{\mu}{m}(g) + \sum_{i=1}^m  \Pnm{(\mu + \delta_x)}{m-1}(g_i(x,\cdot)).
\end{equation}
This remains true for all measurable and integrable $g\colon \BX^m \rightarrow \R$  
which are integrable with respect to $\mu+\delta_x$.
\end{Lemma}
\begin{proof} We prove the first assertion
by induction over $m\in\N$.
The base case is
\begin{equation*}
	(\mu + \delta_x)(g) =\mu (g) +  g(x).
\end{equation*}
We now assume that the assertion holds for all measurable functions 
$g\colon \BX^m \rightarrow [0,\infty]$. Let
$g\colon \BX^{m+1} \rightarrow[0,\infty]$ be measurable.
From the recursion~\eqref{Rekursion mu[m+n]} we obtain
\begin{equation*}
	\Pnm{(\mu + \delta_x)}{m+1}(g)  
= \iint  g(\bfm{y}_{m+1}) \, (\mu + \delta_x + \delta_{\bfm{y}_{m}})(\dd y_{m+1}) \, \Pnm{(\mu + \delta_x)}{m}(\dd\bfm{y}_{m}).
\end{equation*}
Applying the induction hypothesis to the mapping
$\bfm{y}_{m}\mapsto (\mu + \delta_x + \delta_{\bfm{y}_{m}})(g(\bfm{y}_{m},\cdot))$ 
and using the recursion \eqref{Rekursion mu[m+n]} gives 
\begin{align*}
\Pnm{(\mu + \delta_x)}{m+1}&(g)  
= \int  g(\bfm{y}_{m+1}) \, \Pnm{\mu}{m+1}(\dd \bfm{y}_{m+1}) 
+ \int g(\bfm{y}_{m},x)  \, \Pnm{\mu}{m}(\dd \bfm{y}_{m})  \\
&+ \sum_{i=1}^m \iint  g_{i}(x, \bfm{y}_{m-1},y_{m+1})\, (\mu + \delta_{\bfm{y}_{m-1}} + \delta_x)(\dd y_{m+1})\, 
\Pnm{(\mu+ \delta_x)}{m-1}(\dd \bfm{y}_{m-1}) \\
&+ \sum_{i=1}^{m} \int g_{i}(x,\bfm{y}_{m-1},x)\, \Pnm{(\mu+ \delta_x)}{m-1}(\dd \bfm{y}_{m-1}).
\end{align*}
Therefore,       
\begin{align}\notag
\Pnm{(\mu + \delta_x)}{m+1}(g)&=\mu^{[m+1]}(g)
    +\mu^{[m]}(g_{m+1}(x,\cdot))     
+ \sum_{i=1}^m \Pnm{(\mu+ \delta_x)}{m}(g_{i}(x, \cdot))\\  \label{e:912}   
       &\quad + \sum_{i=1}^{m} \int g_{m+1,i}(x,x,\bfm{y}_{m-1})\,\Pnm{(\mu+ \delta_x)}{m-1}(\dd \bfm{y}_{m-1}). 
\end{align}
An application of the induction hypothesis to the function $\bfm{y}_{m} \mapsto g_{m+1}(x,\bfm{y}_{m})$ shows
\begin{align*}
(\mu+\delta_x)^{[m]}(g_{m+1}(x,\cdot)) = \mu^{[m]}(g_{m+1}(x,\cdot))
+\sum_{i=1}^m  (\mu+\delta_x)^{[m-1]}(g_{m+1,i}(x,x, \cdot)).
\end{align*}
Inserting this into \eqref{e:912} yields the result.

Assume now that $g\colon \BX^m \rightarrow \R$ is measurable and integrable w.r.t.\ $\mu+\delta_x$.
Applying \eqref{e:702} to the function $|g|$ shows that all resulting terms in \eqref{e:702} are finite.
Therefore, we obtain the assertion from \eqref{e:702} by decomposing $g$ into its negative
and positive part.
\end{proof}

The next proposition deals with the addition of several Dirac measures.

\begin{proposition} \label{PropRho+delta_x_1+...+delta_x_k}
Suppose that $\mu$ is a $\sigma$-finite measure. Then we have for all $m,k\in \N$, $x_1, \ldots, x_k \in \BX$ and all
measurable functions $f\colon \BX^m \to [0,\infty]$ 
that
\begin{align} \label{eq PropRho+delta_x_1+...+delta_x_k}
   \int& f (z) \,\Pnm{(\mu + \delta_{x_1} + \cdots + \delta_{x_k})}{m}(\dd z) 
    = \int f(z)  \,\Pnm{\mu}{m}(\dd z)\\ 
&+ \sum_{r=1}^{m} \quad \sideset{}{^{\ne}}\sum_{i_1,\dots,i_r\in [m]} 
\int f_{i_1, \ldots, i_r}^k (x_1, \ldots, x_k, z_1, \ldots, z_{m-r}) \,\Pnm{\mu}{m-r}(\dd (z_1, \ldots, z_{m-r})). \nonumber
\end{align}
This remains true for all measurable $f\colon \BX^m \rightarrow \R$  
which are integrable with respect to $\mu + \delta_{x_1} + \cdots + \delta_{x_k}$.
\end{proposition}

\begin{proof}
We prove the assertion by induction on $m$. If $m = 1$ and
$f\colon \BX^m \to [0,\infty]$ is measurable, then 
\begin{equation*}
    (\mu + \delta_{\bfm{x}_{k}}) \left( f \right) 
    = \mu \left(f \right) + \sum_{i=1}^k f (x_i) 
    = \mu \left(f \right) + f^k_1(x_1, \ldots, x_k)
\end{equation*}
by definition of $f^k_1$. 
For the induction step we assume that the assertion holds for some
$m\in\N$, all $k\in\N$ and all $x_1, \ldots, x_k \in \BX$ as well as
all measurable $f\colon \BX^m\to[0,\infty]$. To establish the formula for
$m+1$, we use induction on $k\in\N$.  
Take a measurable $f\colon\BX^{m+1}\to[0,\infty]$ and $x_1\in\BX$. 
By Lemma~\ref{LemRekursionrho+delta_x}, 
\begin{equation*}
     \Pnm{(\mu + \delta_{x_1})}{m+1} (f ) 
    = \Pnm{\mu}{m+1}(f ) + \sum_{i=1}^{m+1} \int f_i(x_1,\bfm{y}_{m}) \,\Pnm{(\mu + \delta_{x_1})}{m}(\dd\bfm{y}_{m}).
\end{equation*}
Let $i\in  [m+1]$. The induction hypothesis in the induction on $m$
applied to the function
$\bfm{y}_{m} \mapsto  f_i(x_1,\bfm{y}_{m})$ 
shows that the second term  on the above right-hand side equals
\begin{align*}
    &\sum_{i=1}^{m+1} \int f_i(x_1,\bfm{y}_{m}) \, \Pnm{\mu}{m}(\dd\bfm{y}_{m}) 
    + \sum_{r=1}^m \quad \sideset{}{^{\ne}}\sum_{i_1,\dots,i_{r+1}\in [m+1]} \int f_{i_1, \ldots, i_{r+1}}^1 (x_1, \bfm{z}_{m-r}) \, \Pnm{\mu}{m-r}(\dd \bfm{z}_{m-r}) \\
    &= \sum_{r=1}^{m+1} \quad \sideset{}{^{\ne}}\sum_{i_1,\dots,i_r\in [m+1]} \int f_{i_1, \ldots, i_{r}}^1 (x_1, \bfm{z}_{m+1-r}) \, \Pnm{\mu}{m+1-r}(\dd \bfm{z}_{m+1-r}).
\end{align*}
This finishes the base case of the induction on $k$. 
Next we assume that the assertion holds for some $k\in\N$.
Let $x_1,\ldots,x_{k+1}\in\BX$. By 
Lemma~\ref{LemRekursionrho+delta_x}, 
\begin{equation*}
    \Pnm{(\mu + \delta_{\bfm{x}_{k+1}})}{m+1} (f) 
= \Pnm{(\mu +  \delta_{\bfm{x}_{k}})}{m+1}(f) 
+ \sum_{i=1}^{m+1} \int f^{(i)} (\bfm{y}_{m})\, \Pnm{(\mu +  \delta_{\bfm{x}_{k+1}})}{m}(\dd\bfm{y}_{m}),
\end{equation*}
where $f^{(i)}(\bfm{y}_{m}) \coloneqq f_i(x_{k+1},\bfm{y}_{m})$ for $\bfm{y}_{m} \in \BX^m$ and $i\in  [m+1]$.
Using the induction hypothesis in the induction on $k$, we obtain
\begin{equation*}
    \Pnm{(\mu + \delta_{\bfm{x}_{k}})}{m+1}(f) 
    =  \Pnm{\mu}{m+1}(f)  + \sum_{r=1}^{m+1} \quad \sideset{}{^{\ne}}\sum_{i_1,\dots,i_{r}\in [m+1]} \int f_{i_1, \ldots, i_r}^k (\bfm{x}_{k}, \bfm{z}_{m+1-r})\, \Pnm{\mu}{m+1-r}(\dd \bfm{z}_{m+1-r}).
\end{equation*}
Let $i\in  [m+1]$. The induction hypothesis in the induction on $m$ yields
\begin{align*}
\int & f^{(i)} (\bfm{y}_{m}) \, \Pnm{(\mu + \delta_{\bfm{x}_{k+1}})}{m}(\dd\bfm{y}_{m}) \\
    &= \Pnm{\mu}{m}(f^{(i)}) + \sum_{r=1}^m \quad \sideset{}{^{\ne}}\sum_{i_1,\dots,i_r\in [m]} \int (f^{(i)})_{i_1, \ldots, i_r}^{k+1} (\bfm{x}_{k+1}, \bfm{z}_{m-r}) \,\Pnm{\mu}{m-r}(\dd \bfm{z}_{m-r}).
\end{align*}
Combining these findings, we obtain
\begin{align*}
&\Pnm{(\mu + \delta_{\bfm{x}_{k+1}})}{m+1} (f) = \Pnm{\mu}{m+1}(f) \\
&\quad + \sum_{r=1}^{m+1} \quad \sideset{}{^{\ne}}\sum_{i_1,\dots,i_{r}\in [m+1]} \int f_{i_1, \ldots, i_r}^k (\bfm{x}_{k}, \bfm{z}_{m+1-r}) \,\Pnm{\mu}{m+1-r}(\dd \bfm{z}_{m+1-r}) 
    + \sum_{i=1}^{m+1} \Pnm{\mu}{m}(f^{(i)}) \\
&\quad+  \sum_{i=1}^{m+1} \sum_{r=1}^m \quad \sideset{}{^{\ne}}\sum_{i_1,\dots,i_r\in [m]} 
\int (f^{(i)})_{i_1, \ldots, i_r}^{k+1} (\bfm{x}_{k+1}, \bfm{z}_{m-r}) \,\Pnm{\mu}{m-r}(\dd \bfm{z}_{m-r})\\
&=:\Pnm{(\mu + \delta_{\bfm{x}_{k+1}})}{m+1} (f)+A+B.
\end{align*}
We need to show that
\begin{equation}\label{e:619}
\sum_{r=1}^{m+1} \quad \sideset{}{^{\ne}}\sum_{i_1,\dots,i_{r}\in [m+1]} \int f_{i_1, \ldots, i_r}^{k+1} (\bfm{x}_{k+1}, \bfm{z}_{m+1-r}) 
\,\Pnm{\mu}{m+1-r}(\dd \bfm{z}_{m+1-r})=A+B.
\end{equation}
To this end, we note that
for $r\in  [m+1]$, pairwise distinct $i_1, \ldots, i_r\in  [m+1]$ 
and $\bfm{z}_{m+1-r}\in \BX^{m+1-r}$ we have by definition
that
\begin{align*}
    &f_{i_1, \ldots, i_r}^{k+1} (\bfm{x}_{k+1}, \bfm{z}_{m+1-r}) 
    = \sum_{1\le j_1 \le \cdots \le j_r \le k+1} f_{i_1, \ldots, i_r} (x_{j_1}, \ldots, x_{j_r}, \bfm{z}_{m+1-r}) \\
    &= \sum_{1\le j_1 \le \cdots \le j_r \le k} f_{i_1, \ldots, i_r} (x_{j_1}, \ldots, x_{j_r}, \bfm{z}_{m+1-r})\\
&\quad     + \sum_{1\le j_1 \le \cdots \le j_{r-1} \le k+1} f_{i_1, \ldots, i_r} (x_{j_1}, \ldots, x_{j_{r-1}}, x_{k+1}, \bfm{z}_{m+1-r}) \\
     &= f_{i_1, \ldots, i_r}^{k} (\bfm{x}_{k}, \bfm{z}_{m+1-r}) 
+\sum_{1\le j_1 \le \cdots \le j_{r-1} \le k+1} f_{i_1, \ldots, i_r} (x_{j_1}, \ldots, x_{j_{r-1}}, x_{k+1}, \bfm{z}_{m+1-r})
\end{align*}
if $r\ge 2$. In the case $r=1$, it holds
\begin{align*}
    f_{i_1}^{k+1} (\bfm{x}_{k+1}, \bfm{z}_{m}) 
   & = \sum_{1\le j_1  \le k+1} f_{i_1} (x_{j_1}, \bfm{z}_{m}) \\
    &= \sum_{1\le j_1 \le k} f_{i_1} (x_{j_1}, \bfm{z}_{m})
     +  f_{i_1} (x_{k+1}, \bfm{z}_{m}) 
     = f_{i_1}^{k} (\bfm{x}_{k}, \bfm{z}_{m}) + f^{(i_1)}(\bfm{z}_{m}).
\end{align*}
Therefore, the left-hand side of \eqref{e:619} equals
\begin{align*}
&\sum_{r=1}^{m+1} \quad\sideset{}{^{\ne}}\sum_{i_1,\dots,i_{r}\in [m+1]} \int f_{i_1, \ldots, i_r}^{k} (\bfm{x}_{k}, \bfm{z}) \,\Pnm{\mu}{m+1-r}(\dd \bfm{z}) 
    + \sum_{i_1=1}^{m+1} \int f^{(i_1)}(\bfm{z}) \,\Pnm{\mu}{m}(\dd \bfm{z}) \\
    & + \sum_{r=2}^{m+1} \quad\sideset{}{^{\ne}}\sum_{i_1,\dots,i_{r}\in [m+1]} \int \sum_{1\le j_1 \le \cdots \le j_{r-1} \le k+1} 
f_{i_1, \ldots, i_r} (x_{j_1}, \ldots, x_{j_{r-1}}, x_{k+1}, \bfm{z}) \,\Pnm{\mu}{m+1-r}(\dd \bfm{z}).  
\end{align*}
Here the first term equals $A$, while the second equals
\begin{equation*}
\sum_{r=1}^m \quad \sideset{}{^{\ne}}\sum_{i_1,\dots,i_r\in [m]} \sum_{i=1}^{m+1} \int \sum_{1\le j_1 \le \cdots \le j_r \le k+1} 
(f^{(i)})_{i_1, \ldots, i_r} (x_{j_1}, \ldots, x_{j_r}, \bfm{z}_{m-r})
\, \Pnm{\mu}{m-r}(\dd \bfm{z}_{m-r})=B.
\end{equation*}
This finishes the induction on $k$ and hence also on $m$.

The second assertion follows as in the proof of Lemma \ref{LemRekursionrho+delta_x}.
\end{proof}

Let $m, r\in\N$ with $m\ge r$ and
$f_1, \ldots, f_m \colon \BX \to \R$.  The tensor product
$(\otimes_{j=1}^m f_j) \colon \BX^m \to \R$ is defined by
\begin{equation} \label{Def Tensorprodukt}
    \bigg(\bigotimes_{j=1}^m f_j\bigg) (x_1, \ldots, x_r) \coloneqq \prod_{j=1}^m f_j (x_j).
\end{equation}
Given $f_1, \ldots, f_m$ and pairwise distinct $i_1, \ldots, i_r \in  [m]$,  
we define
\begin{align} \label{eq Def f_otimes}
    f_{\otimes_ {i_1, \ldots, i_r}} \coloneqq \bigotimes_{j=1}^r f_{i_j}.\qquad 
f^{\otimes_{i_1, \ldots, i_r}} \coloneqq \bigotimes_{j\notin \{i_1, \ldots, i_r\}} f_j,
\end{align}
where $f^{\otimes_{i_1, \ldots, i_m}} \coloneqq 1$. 
For $k\in \N$ and $x_1, \ldots, x_k \in \BX$ let
\begin{equation} \label{eq Def f_otimes^k}
    f_{\otimes_{i_1, \ldots, i_r}}^k (x_1, \ldots, x_k) \coloneqq \sum_{1\le j_1 \le \cdots \le j_r \le k} f_{\otimes_{i_1, \ldots, i_r}} (x_{j_1}, \ldots, x_{j_r}).
\end{equation}

\begin{corollary} \label{Kor Rho+delta_x_1+...+delta_x_k}
Let $\mu$ be a $\sigma$-finite measure on $\BX$. 
Suppose that $k, m \in\N$, $x_1, \ldots, x_k \in \BX$ and
$f_1, \ldots, f_m \colon \BX \to \R$ are measurable functions such that
$\otimes_{i=1}^m f_i$ is integrable with respect to
$\Pnm{(\mu + \delta_{x_1} + \cdots + \delta_{x_k})}{m}$.
Then
\begin{align*}
   &\int \bigg(\bigotimes_{i=1}^m f_i\bigg) (z) \,\Pnm{(\mu + \delta_{x_1} + \cdots + \delta_{x_k})}{m}(\dd z) \\
    &= \int \bigg(\bigotimes_{i=1}^m f_i\bigg)(z)  \,\Pnm{\mu}{m}(\dd z) 
+ \sum_{r=1}^m \quad\sideset{}{^{\ne}}\sum_{i_1,\dots,i_r\in [m]} f_{\otimes_{i_1, \ldots, i_r}}^k (x_1, \ldots, x_k) \int f^{\otimes_{i_1, \ldots, i_r}}(z)  \,\Pnm{\mu}{m-r}(\dd z). 
\end{align*}
\end{corollary} 

\section{A summation formula for rising factorials}

The following formula has been used in the proof of  Proposition \ref{l:nabla=nabla*}.
Since we have not found the result in the literature, we give the simple proof.

\begin{Lemma} \label{LemHypergeoSumme 1}
Let $\theta >0$ and $m\in \N$, $m\ge 2$. 
For each $j \in [m-1]$ it holds
\begin{equation*}
\sum_{n=j}^{m} (-1)^{n-j} \frac{\theta + 2n - 1}{(n-j)!}\Pn{(\theta+j)}{n-1} 
= (-1)^{m-j} \frac{\Pn{(\theta+j)}{m}}{(m-j)!}. 
\end{equation*}
\end{Lemma}
\begin{proof} Let $S$ denote the left-hand side of the asserted identity.
We have
\begin{align*}
S= (\theta +2j-1) \Pn{(\theta+j)}{j-1} &- (\theta + 2j + 1)\Pn{(\theta+j)}{j} \\
&+  \sum_{n=j+2}^{m} (-1)^{n-j} \frac{\theta + 2n - 1}{(n-j)!}\Pn{(\theta+j)}{n-1}.
\end{align*}
Using 
the definition of rising factorials, we can write $S$ as
\begin{equation*}
S= \Pn{(\theta+j)}{j} \left( -(\theta + 2j) 
+ \sum_{n=j+2}^{m} (-1)^{n-j} \frac{\theta + 2n - 1}{(n-j)!}\prod_{k=j}^{n-2} (\theta + k +j) \right).
\end{equation*}
In particular, we obtain the assertion for $j=m-1$.
In the case $j\le m-2$ 
we have
\begin{equation*}
S=    \Pn{(\theta+j)}{j+1} \left( -1 + \frac{\theta + 2j + 3}{2!} 
+ \sum_{n=j+3}^{m} (-1)^{n-j} \frac{\theta + 2n - 1}{(n-j)!}\prod_{k=j+1}^{n-2} (\theta + k +j) \right),
\end{equation*}
that is
\begin{equation*}
S=\frac{\Pn{(\theta+j)}{j+2}}{2} \left( 1 + 2 \sum_{n=j+3}^{m} (-1)^{n-j} \frac{\theta + 2n - 1}{(n-j)!}\prod_{k=j+2}^{n-2} (\theta + k +j) \right).
\end{equation*}
In particular, we obtain the assertion for $j=m-2$. 
The general case follows inductively.
\end{proof}

\bigskip
\noindent
{\bf Acknowledgment of AI assistence:}
The third version of this paper contains a new section on a reverse
Poincar\'e inequality. In preparing this section, the authors benefited
from discussions with ChatGPT-6. In particular, the AI helped to
strengthen an earlier version of the inequality, to check proofs and
calculations, and to review the relevant literature.

\end{document}